\pdfoutput=1
\documentclass{amsart}

\usepackage{amsmath}%
\usepackage{amsfonts}%
\usepackage{amssymb}%
\usepackage{graphicx}
\usepackage{graphicx, xcolor}              
\usepackage{amsmath}               
\usepackage{amsfonts}              
\usepackage{amsmath}
\usepackage{amsthm}  
 \usepackage{float} 
 \usepackage{MnSymbol}
 \usepackage{tikz}
\usetikzlibrary{shapes.geometric}
\usetikzlibrary{arrows.meta,arrows}

\usepackage{caption}
\usepackage{wrapfig}
\usepackage{subcaption}
\usepackage{color}

\usepackage{mathrsfs}
\usepackage{url}
\usepackage{cite}

\newtheorem{note}{Note}
\newtheorem{theorem}{Theorem}

\theoremstyle{plain}

\newtheorem{corollary}{Corollary}

\newtheorem{definition}{Definition}
\newtheorem{example}{Example}

\newtheorem{lemma}{Lemma}

\newtheorem{proposition}{Proposition}
\newtheorem{remark}{Remark}

\numberwithin{equation}{section}

\begin{document}
\title{Quantum invariants of knotoids}

\author{Neslihan G{\"u}g{\"u}mc{\"u}}
\author{Louis H.Kauffman}

\address{Neslihan G{\"u}g{\"u}mc{\"u}:  Department of Mathematics, Izmir Institute of Technology, G\"ulbah\c ce Campus, Izmir 35430, Turkey and Mathematics Institute, Georg-August Universitat G\"ottingen, Bunsenstrasse 3-7, G\"ottingen 37073, Germany}
\address{Louis H.Kauffman:Department of Mathematics, Statistics and Computer
Science, University of Illinois at Chicago, 851 South Morgan St., Chicago
IL 60607-7045, U.S.A. and
Department of Mechanics and Mathematics
Novosibirsk State University
Novosibirsk
Russia}
\email{neslihangugumcu@iyte.edu.tr; neslihan.gueguemcue@math.uni-goettingen.de} \email{kauffman@math.uic.edu}

\begin{abstract}
In this paper, we construct quantum invariants for knotoid diagrams in $\mathbb{R}^2$. The diagrams are arranged with respect to a given direction in the plane ({\it Morse knotoids}). A Morse knotoid diagram can be decomposed into basic elementary diagrams each of which is associated to a matrix that yields solutions of the quantum Yang-Baxter equation. We recover the bracket polynomial, and define the rotational bracket polynomial, the binary bracket polynomial, the Alexander polynomial, the generalized Alexander polynomial and an infinity of specializations of the Homflypt polynomial for Morse knotoids via quantum state sum models. \end{abstract}
\maketitle
\section{Introduction}

A \textit{Morse knotoid diagram}  is a knotoid diagram in the plane arranged with respect to the bottom to top vertical direction of the plane so that every horizontal line meets the diagram in at most one minimum, maximum or a crossing point. The plane where a knotoid diagram lies, can be interpreted as the spacetime plane with time the vertical axis and space the horizontal. Then, a Morse knotoid diagram can be interpreted as a vacuum to vacuum process that includes creations of particles (\textit{cup}s), a finite number of interactions among the particles (\textit{crossing}s) and annihilations of the particles (\textit{cap}s). These events are associated with matrices with indices on diagrams (which can be interpreted as \textit{spins} on particles)  and  the probability amplitude of the vacuum to vacuum process  is obtained by summing the products of the amplitudes of each internal configuration over all indices. In order to have a topological amplitude we make constraints so that the amplitude remains invariant under the regular isotopy moves including the braiding moves. Note that we use the word amplitude for motivation, as these amplitudes are not necesarily physical and are not constrained to take values in the complex numbers.

Many knot/link invariants associated with quantum amplitudes have been constructed so far. Prior to such knot and link invariants, C.N. Yang used this interpretation, in terms of creations, annihilations and interactions, in a miniature model of quantum field theory. He created an equation, now called the \textit{Yang-Baxter equation}, so that the amplitudes of interacting particles would be the same for patterns with equivalent permutations. In the braiding context the Yang-Baxter equation corresponds to the braiding relation. Thus a matrix equation originating in a quantum physical model is useful for doing topology.

In our formulation of quantum invariants we use a given direction in the diagram plane to stand for time so that the cups and caps and interactions can be seen to occur in a temporal order. This can be formalized by a Morse function on the diagram, hence the term Morse knotoids or Morse diagrams. It can also be formalized by the notion of a category, and then time's arrow becomes the arrow directions for the morphims in the category. From a categorical point of view, knot, link and knotoid diagrams become morphisms in a braided category. All of these points of view are useful and we will move through all of them in our constructions.

Let us now give an outline of this paper. In Section \ref{sec:Morseknotoids} we introduce Morse knotoid diagrams and the topological moves for them. In Section \ref{sec:rotation} we define a numerical invariant for Morse knotoids that we call the \textit{rotation number} and in Section \ref{sec:rotbracket} we define a rotational extension of the Kauffman bracket polynomial for Morse knotoids.  In Section \ref{sec:Cat} we study Morse knotoid diagrams via category theoretical descriptions. In Section \ref{sec:quantum} we construct unoriented invariants for Morse knotoids  in the form of a partition function, such as the bracket partition function and the binary bracket polynomial. We show that both invariants admit a state sum model that yields solutions to the Yang-Baxter equation. In Section \ref{sec:orientedquantum}, we introduce a general schema to generate oriented quantum invariants for Morse knotoids. By specializing this schema, we define the Alexander polynomial, the generalized Alexander polynomial and an infinity of analogs of specializations of the classical Homflypt polynomial.



\section{Morse knotoids}\label{sec:Morseknotoids}
\begin{definition} \cite{Tu}\normalfont
A \textit{knotoid diagram} in $\mathbb{R}^2$ or in $S^2$, namely a \textit{planar knotoid diagram} and a \textit{spherical knotoid diagram}, respectively, is a generic immersion of the unit interval $[0,1]$ into $\mathbb{R}^2$ or $S^2$, with the assumption that it contains only transversal double points endowed with under/over information, and two distinct endpoints the \textit{leg} and the \textit{head} as the images of $0$ and $1$, respectively. A knotoid diagram is endowed with a natural orientation from its leg to its head.
\end{definition}

In this section, we introduce a special class of planar knotoid diagrams, called \textit{Morse knotoid diagrams}. 

\begin{definition} \rm 
A {\it Morse knotoid diagram} is a knotoid diagram in $\mathbb{R}^2$, equipped with a height function, such that every horizontal line meets the knotoid diagram transversally with at most one critical point. A \textit{critical point} of the height function can be one of the following: a \textit{crossing point} of a classical crossing of the Morse knotoid diagram, a \textit{minimum} or a \textit{maximum} of the curve with respect to  the height function, or an \textit{endpoint} of the knotoid. We further assume that the arcs containing the endpoints are vertical arcs in a neighborhood of the endpoints with respect to the height function. See \ref{fig:morseknotoiddi} for an example. %
\end{definition}
\begin{figure}[H]
  \includegraphics[scale=.2]{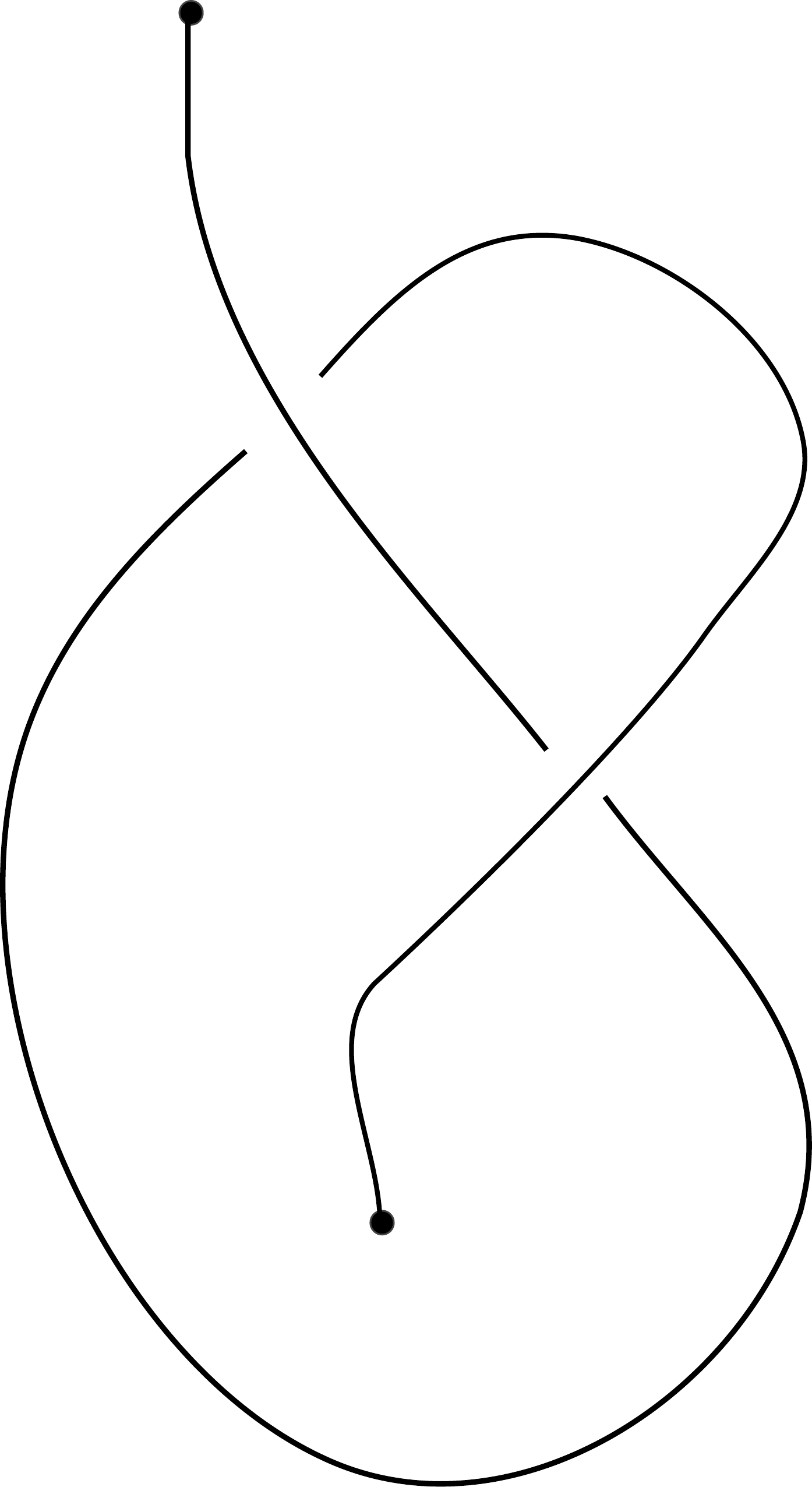}
\caption{A Morse knotoid diagram}
  \label{fig:morseknotoiddi}
\end{figure}%

A \textit{multi-knotoid diagram} in $\mathbb{R}^2$ or in $S^2$ is a union of a knotoid diagram with a number of knots \cite{Tu} and the definition of a Morse knotoid diagram extends to Morse multi-knotoid diagram directly. 

\subsection{Isotopy moves on knotoid diagrams }

Knotoid diagrams in $\mathbb{R}^2$ are classified up to the isotopy relation generated by the \textit{knotoid Reidemeister moves} and plane isotopy. Knotoid Reidemeister moves are classical Reidemeister moves that take place in local disks free of the endpoints of knotoid diagrams, see Figure \ref{fig:reidemistermoves} for the move list. The endpoints of a knotoid diagram may be displaced by the knotoid isotopy but It is forbidden to move an endpoint over or under a strand. A \textit{knotoid} in $\mathbb{R}^2$ (or a \textit{planar knotoid}) is an equivalence class of knotoid diagrams in $\mathbb{R}^2$, considered up to the induced isotopy relation \cite{Tu}. 
\begin{definition}\label{def:defone}\normalfont
A knotoid diagram in $\mathbb{R}^2$ or in $S^2$ is called a knot-type knotoid diagram if its endpoints lie in the same local region of the plane determined by the graph underlying the knotoid diagram.
\end{definition}
We will call a planar or spherical knotoid a \textit{knot-type knotoid} if it admits at least one representative diagram with endpoints lying in the same region, otherwise it will be called a \textit{proper knotoid}. Note that, in \cite{Tu}, the term knot-type is suggested for spherical knotoids since there is a one-to-one correspondence between the set of knots in $\mathbb{R}^3$ and the set of spherical knotoids whose endpoints can be brought to the same region determined by diagrams. For planar knotoids, this correspondence is no longer a one-to-one correspondence, see \cite{Tu, GK1}. With Definition \ref{def:defone} we extend the use of the term.

We assume an isotopy relation for Morse knotoid diagrams generated by the moves shown in Figure \ref{fig:moveset} that comprise cancellation/insertion of sequential local maxima and minima on an arc called \textit{min-max moves}, swinging of an arc with respect to a local maximum or minimum called the \textit{slide moves},  the vertical Reidemeister~II  type moves and the Reidemeister~III type braiding moves, plus planar isotopies that displaces the endpoints vertically or horizontally, see Figure \ref{fig:verticalmoves}. Note that horizontal Reidemeister~II type moves can be generated by the vertical Reidemeister ~II type move plus the slide moves.  
\begin{figure}[H]
  \includegraphics[scale=.4]{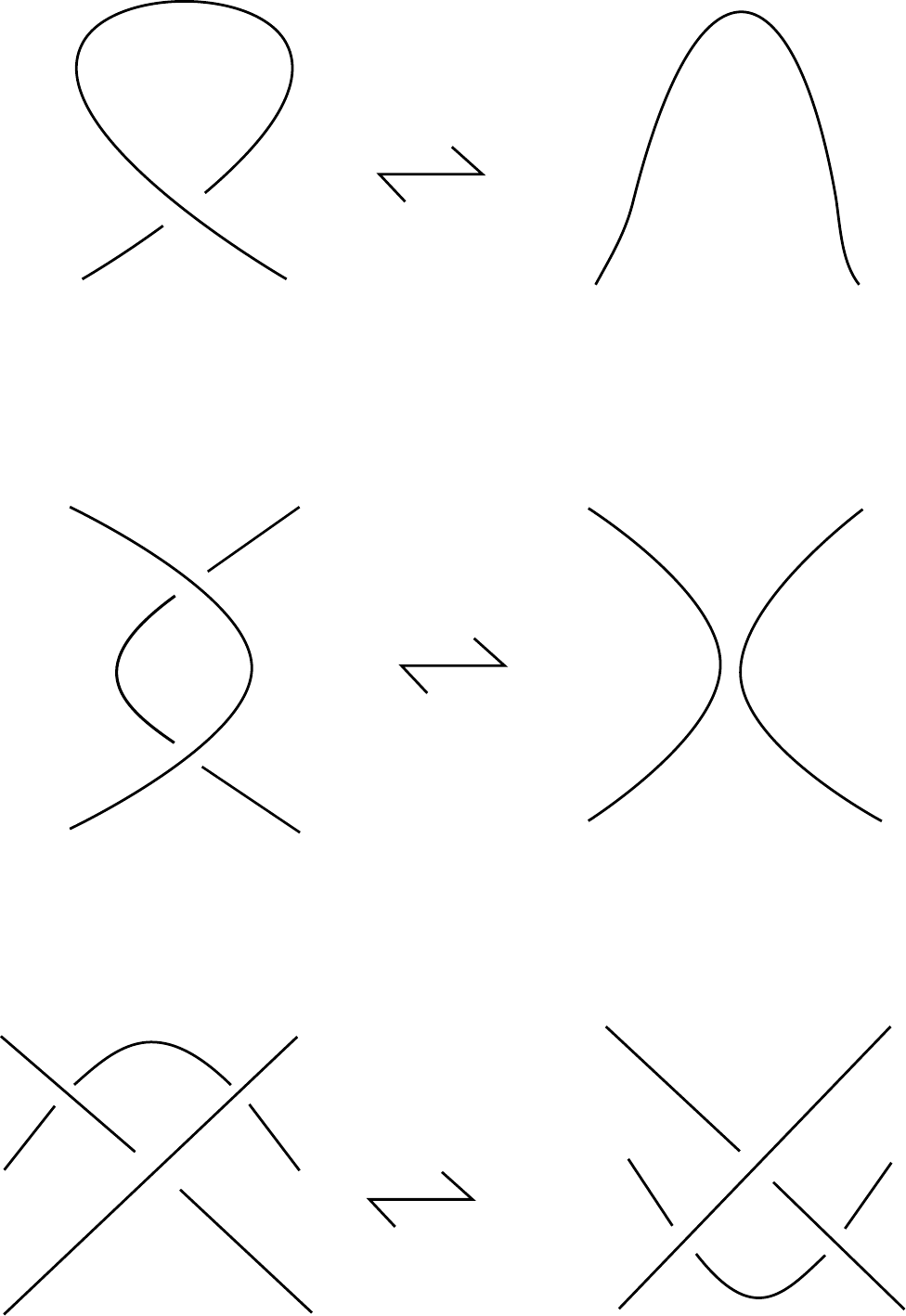}
\caption{Classical Reidemeister moves}
  \label{fig:reidemistermoves}
\end{figure}%
\begin{figure}[H]
\centering
\includegraphics[width=.65\textwidth]{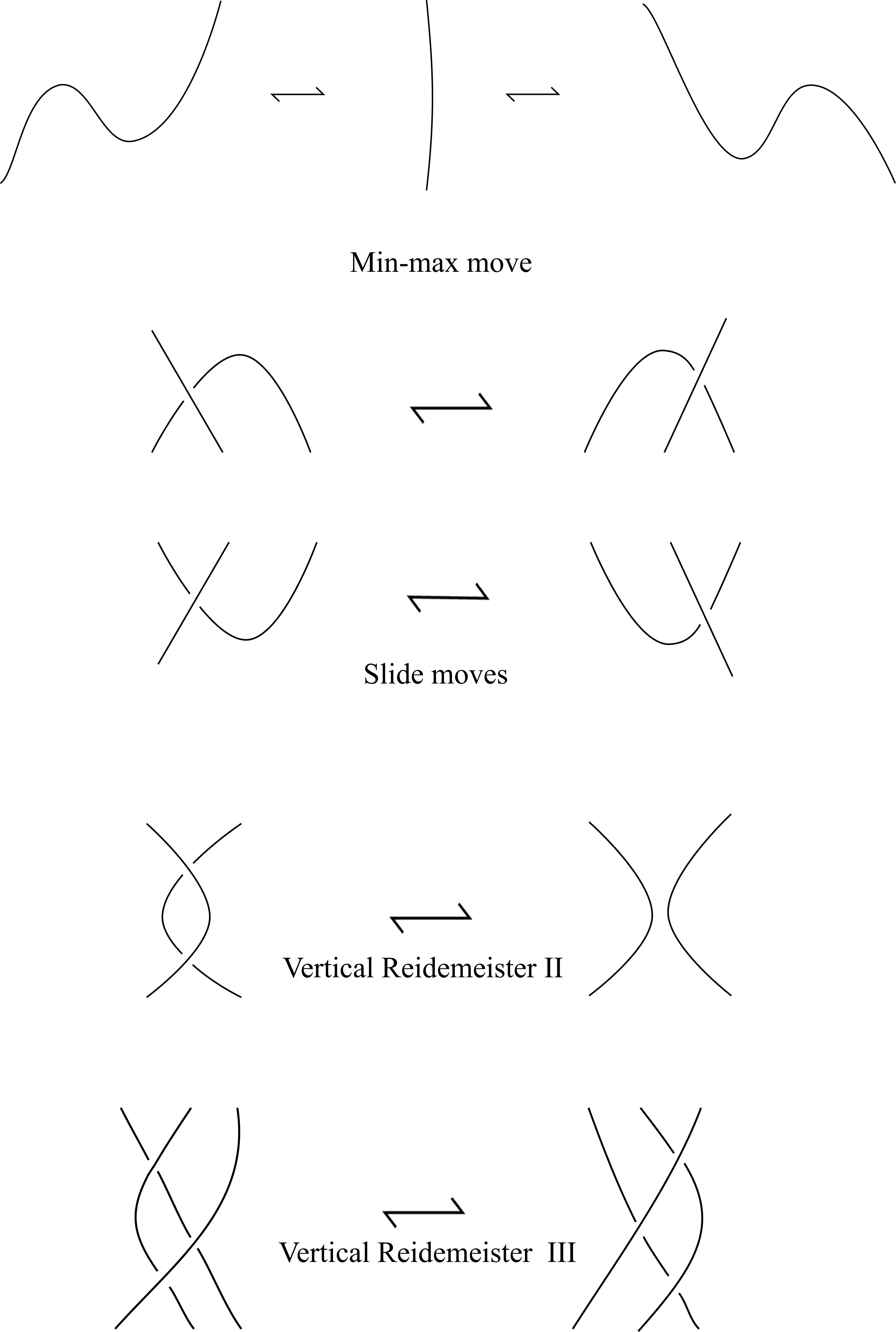}
\caption{The Morse isotopy moves}
\label{fig:moveset}
\end{figure}
Similar to knotoids in $\mathbb{R}^2$, it is forbidden to pull an endpoint of a Morse knotoid diagram over or under a transversal strand, so as to avoid unknotting of a Morse knotoid diagram.  Furthermore, we do not allow planar rotations of endpoints, in analogy to having the ends of a tangle fixed. As a result, the  directions of the tangent vectors are preserved. We call this restricted version of regular isotopy of planar knotoid diagrams {\it Morse isotopy}  of knotoids and the corresponding isotopy classes of Morse knotoid diagrams \textit{Morse knotoids}.


\begin{figure}[H]
\centering
\includegraphics[width=.35\textwidth]{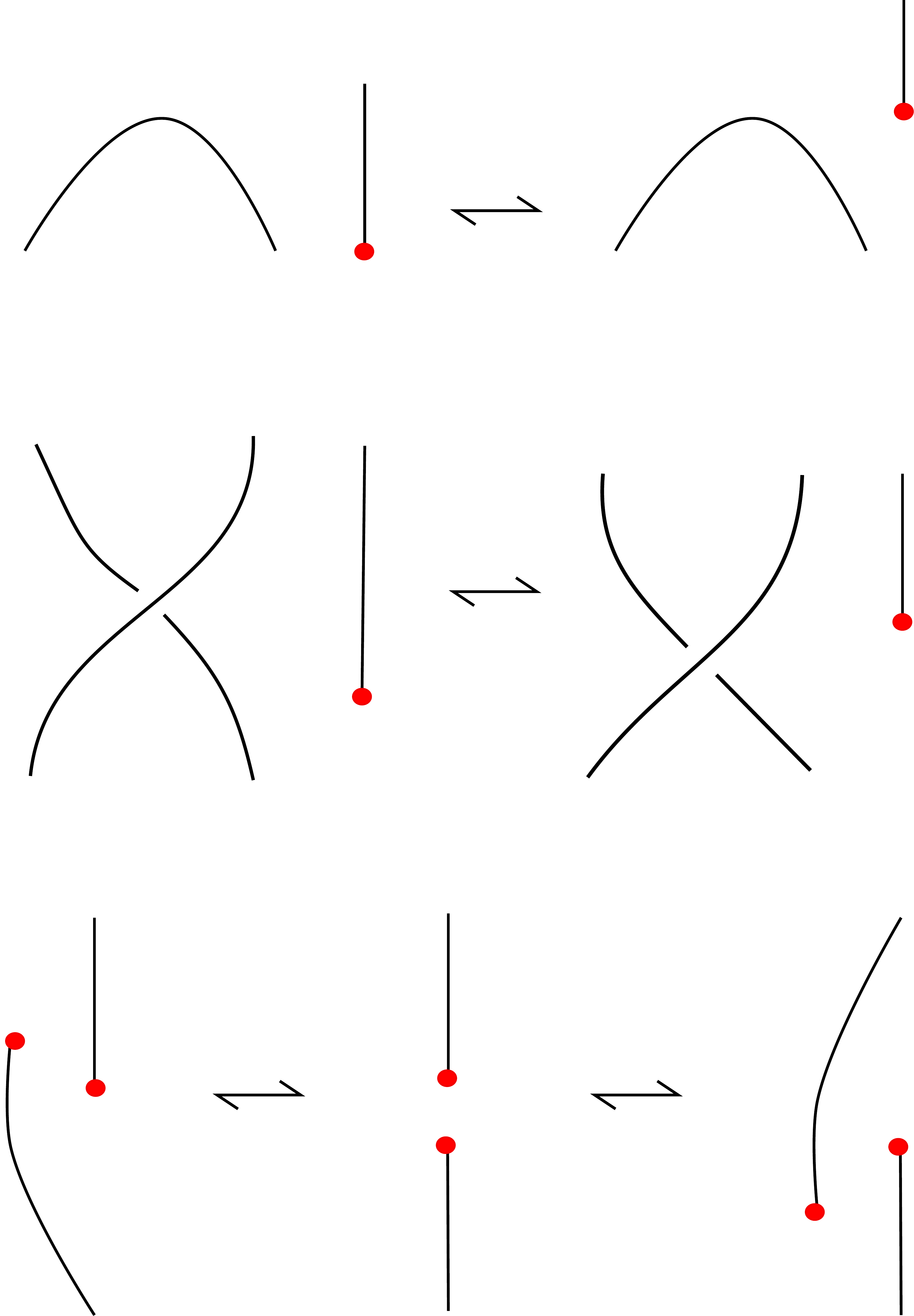}
\caption{Vertical and horizontal shifting moves of the endpoints}
\label{fig:verticalmoves}
\end{figure}

Any knotoid diagram in $\mathbb{R}^2$ can be transformed into a Morse form via regular isotopy that is generated by the second and third Reidemeister moves, vertical shifting and rotation of the endpoints. We call a knotoid diagram in Morse form \textit{standard} if the tangent vectors at its endpoints are both directed upwards with respect to the bottom to top vertical direction of the plane. The Morse knotoid diagram given in Figure \ref{fig:morseknotoiddi} is a standard Morse knotoid diagram.

\begin{proposition}
 Two knotoid diagrams are equivalent via regular isotopy if and only if their standard Morse diagrams are equivalent in the Morse category via Morse isotopy. 
 \end{proposition}
 \begin{proof}
 The proof proceeds in the same fashion with the proof in \cite{Yet}. 
 \end{proof}
 
\begin{remark}\normalfont
Quantum invariants as we shall define them, are often not invariant under the first Reidemeister move (just as the Kauffman bracket polynomial is not so invariant).  In doing a theory of quantum invariants, we exclude the first Reidemeister move and work with regular isotopy with the restriction on the endpoints when we work with open ended objects. To obtain invariants for knotoids, we can often normalize an invariant of regular isotopy with a term induced by the first Reidemeister move as we do for the bracket polynomial.
\end{remark}



A Morse knotoid diagram can be  also endowed with an orientation from its leg to its head. Oriented Morse isotopy on oriented Morse knotoid diagrams are generated by the oriented versions of the Morse isotopy moves. In the sequel we will work with both oriented and non-oriented Morse knotoids.

\subsubsection{Virtual closure}
The theory of virtual knots was introduced by the second author in \cite{Ka1,Ka2}. A \textit{virtual knot diagram} in $\mathbb{R}^2$ is a knot diagram in $\mathbb{R}^2$ with classical and \textit{virtual} crossings represented by circles placed around transversal intersection points of the diagram. Virtual knot diagrams are considered up to the \textit{virtual isotopy} relation that is generated by the classical Reidemeister moves and the virtual Reidemeister moves $VRI, VRII, VRIII$ and the mixed virtual move, given in Figure \ref{fig:vr}.  A \textit{virtual knot} is an isotopy class of virtual knot diagrams up to the virtual isotopy.

\begin{figure}[H]
  \includegraphics[width=.4\linewidth]{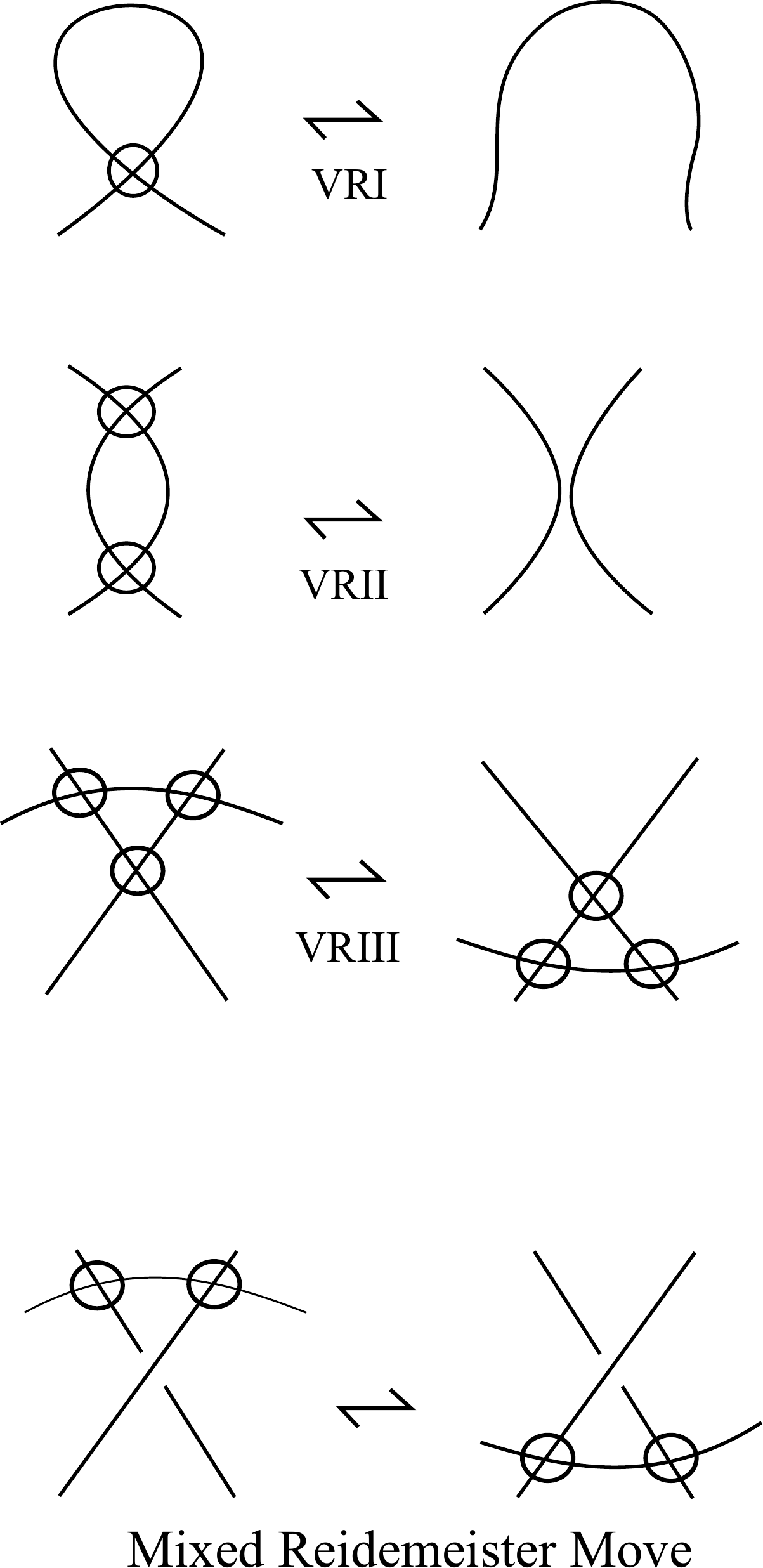}
\caption{The virtual Reidemeister moves}
\label{fig:vr}
\end{figure}


The endpoints of a knotoid diagram can be connected with an arc by declaring each intersection of the arc with the diagram as a virtual crossing.  By doing so, we obtain a virtual knot diagram (or a virtual link in the case of a multi-knotoid diagram). In fact, we have a well-defined mapping, named as the \textit{virtual closure}, from the set of multi-knotoids to the set of (oriented) virtual links \cite{Tu, GK1}.  The virtual closure can be naturally applied for Morse multi-knotoid diagrams to obtain virtual link diagrams in the Morse form.

\begin{proposition}
The virtual closure of a knotoid diagram in $\mathbb{R}^2$ gives the same virtual knot as the virtual closure of its Morse knotoid representation up to virtual isotopy.

\end{proposition}
\begin{proof}
The standard Morse knotoid representation of a knotoid $K$ in $\mathbb{R}^2$ is equivalent to $K$ via knotoid isotopy, and the virtual closure is a well-defined mapping on the set of isotopy classes of multi-knotoids. Then the statement follows.
\end{proof}
It is often the case that an invariant of virtual knots and links can be reformulated for knotoids in $\mathbb{R}^2$ by examining the closure relationship. In this paper, quantum invariants we construct for Morse knotoids and knotoids in $\mathbb{R}^2$, can be formulated as \textit{rotational invariants} of rotational virtual knots and links. In virtual knot theory, \textit{rotational equivalence} is the isotopy relation on virtual knot diagrams that is generated by the classical and virtual Reidemeister moves except the first virtual move \cite{Karot}, and rotational virtual knots are considered up to this relation. In rotational virtual knot theory, a virtual curl can not be directly simplified but two opposite virtual curls can be created or destroyed by using the Whitney Trick, shown in Figure \ref{fig:whitney} where self-intersections of the curve are regarded as virtual crossings. 

\begin{figure}[H]
\centering
\includegraphics[width=.8\textwidth]{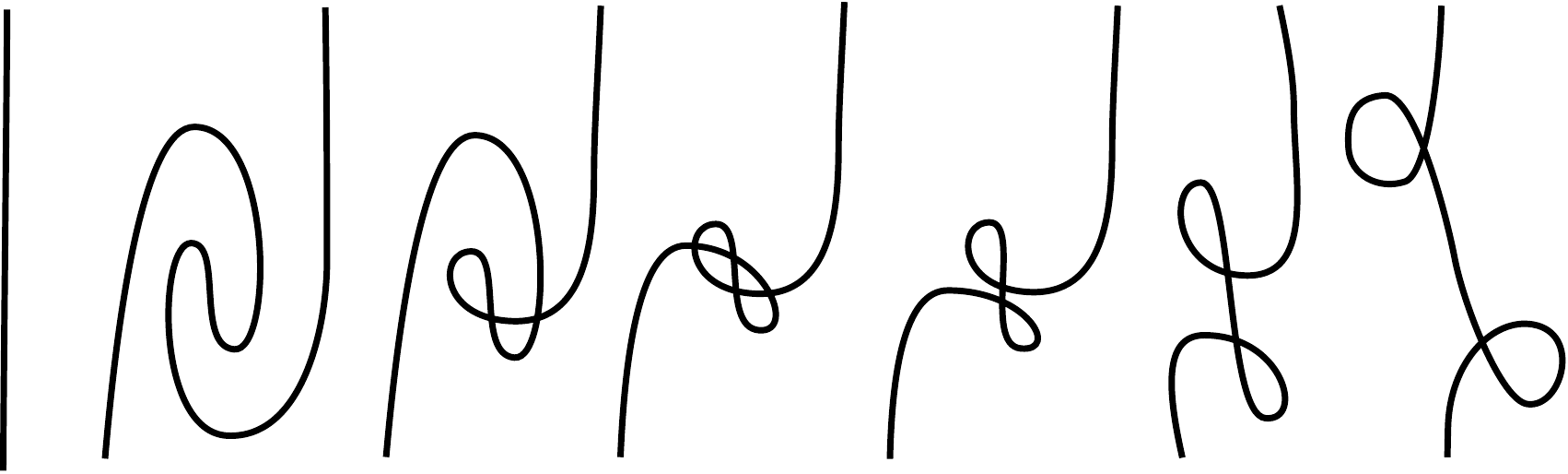}
\caption{Whitney trick for an immersed curve}
\label{fig:whitney}
\end{figure}


\subsection{Rotation number for oriented Morse knotoid diagrams}\label{sec:rotation}
Let $K$ be an oriented Morse knotoid diagram. The \textit{rotation number} for $K$, $rot(K)$ is a real number that is defined as the half of the total number of counterclockwise oriented cups and caps minus the total number of clockwise oriented cups and caps on $K$. In other words, cups and caps of $K$ are endowed with a sign induced by the orientation: The right-pointed maxima and left-pointed minima are signed with $-1$, and left-pointed maxima and right-pointed minima are signed with $+1$, as shown in Figure \ref{fig:sign}. Then, the \textit{rotation number} of $K$, $rot(K)$ is the half of the sum of the signs on the the maxima and minima of $K$.

\begin{theorem} \label{thm:rotationn}
The rotation number is a Morse isotopy invariant.
\end{theorem}
\begin{proof}

The rotation number remains invariant under the Morse isotopy moves since they do not add or delete any cups or caps, except the min-max moves, in which a sequential minimum and maximum with opposite signs are added/deleted. It is clear that the total contribution to the rotation number of these moves is zero.
\end{proof}

 \begin{figure}[H]
\centering
\includegraphics[width=.5\textwidth]{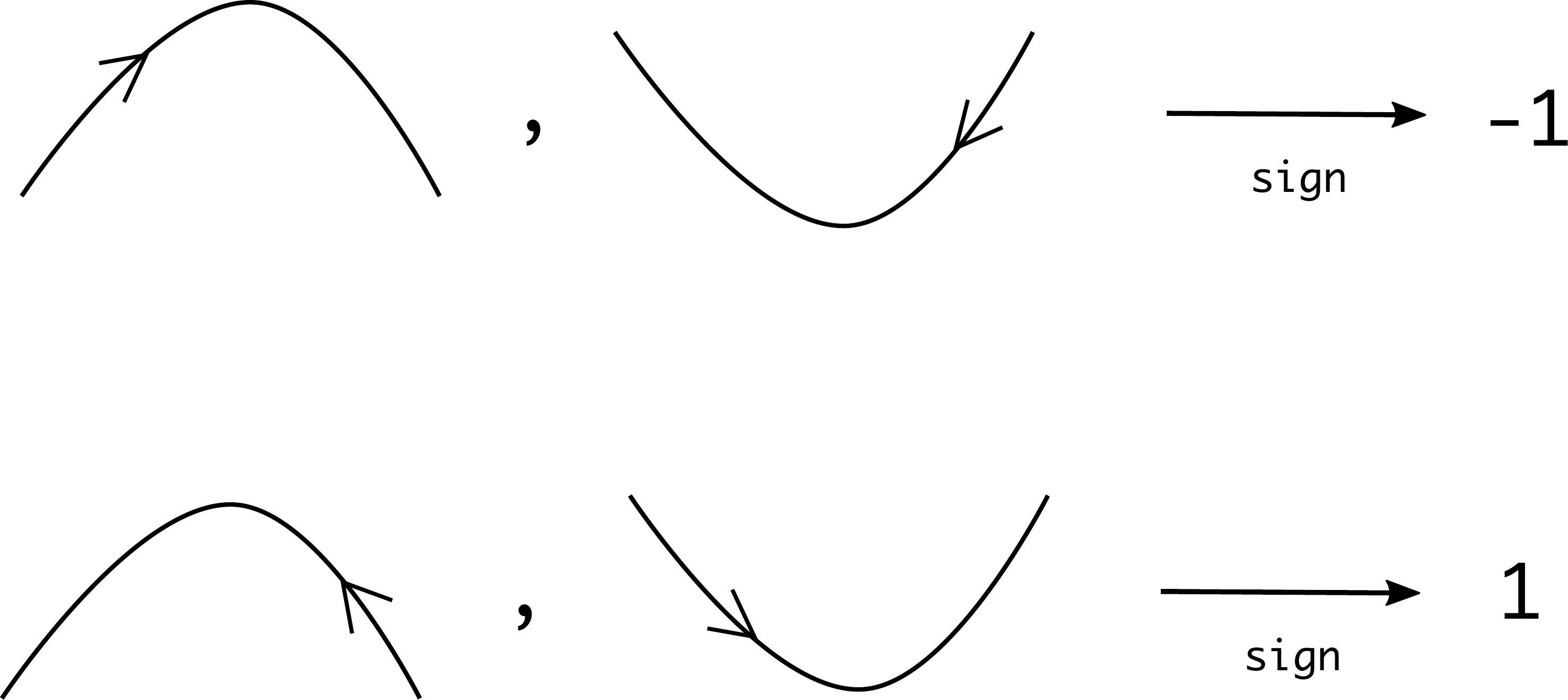}
\caption{Signs of cups and caps}
\label{fig:sign}
\end{figure}


\subsubsection{A rotational bracket polynomial for Morse knotoids}\label{sec:rotbracket}
A rotational extension of the Kauffman bracket polynomial can be defined for Morse knotoids based on the rotation number defined for open-ended state components. Let $K$ be a Morse knotoid diagram. Each crossing of $K$ is smoothed in two ways as in the bracket case to obtain the collection of bracket states. In each state, there exists  exactly one open-ended component that is a simple arc containing the two endpoints of $K$ and a number of circular components. We assume an orientation for the open-ended state components from the leg of $K$ to its head. The {\it rotation number} for an oriented open-ended state component is defined as follows.
\begin{definition}\normalfont
Let $\lambda$ be an oriented open-ended state component in a bracket state of $K$. The \textit{rotation number} of $\lambda$, $rot(\lambda)$ is the half of the sum of the signs of cups and caps on $\lambda$.
\end{definition}

\begin{definition}[Equivalent Definition] \normalfont
The \textit{rotation number} of $\lambda$ is equal to the total turn of the tangent vector along  $\lambda$. 
\end{definition}

It is clear that the rotation number of an oriented open-ended state component remains invariant under Morse isotopy moves.

\begin{definition} \normalfont
A \textit{spiral Morse knotoid diagram} with rotation number $n$ is the Morse knotoid diagram with only $n$ cups and $n$ caps, oriented in the same direction. 
\end{definition}
 \begin{figure}[H]
\centering
\includegraphics[width=.25\textwidth]{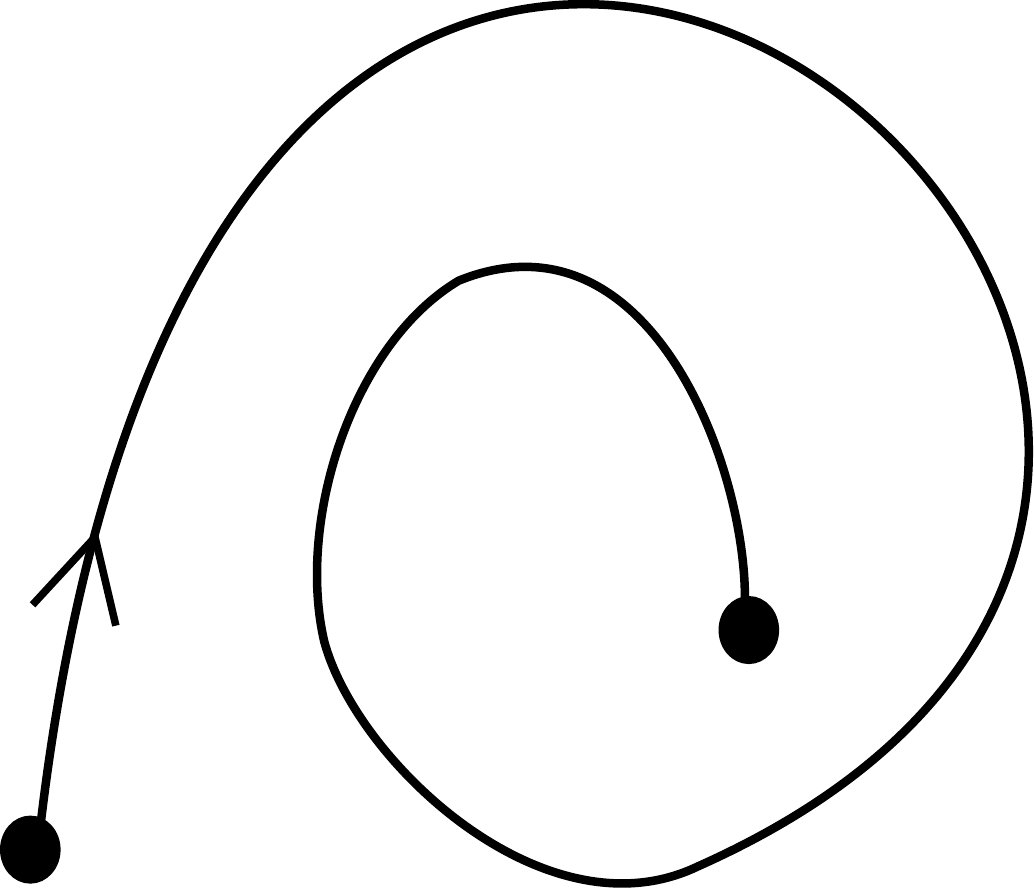}
\caption{An oriented spiral diagram with rotation number $\frac{-3}{2}$.}
\label{fig:trick}
\end{figure}

\begin{lemma}\label{lem:spi}\normalfont
Two oriented spiral Morse knotoid diagrams $S_1, S_2$ are Morse isotopic if and only if they have the same tangent directions at their endpoints and the same rotation number.
\end{lemma}
\begin{proof}
Assume first that $S_1$, $S_2$ are two Morse isotopic spiral Morse knotoid diagrams.  By definition, Morse isotopy moves preserve the tangent directions at the endpoints and they preserve the rotation number. Therefore, we deduce that $S_1$ and $S_2$ have the same rotation number.

Now, let $S_1$, $S_2$ be two oriented spiral Morse knotoid diagrams whose rotation numbers are the same, say  $n > 0$. (For $n=0$, the statement clearly follows). 
Since the two spiral diagrams have the same rotation number, by definition, they both have $n$ cups and $n$ caps oriented in the same direction. The spiral diagrams $S_1$ and $S_2$ can be found in exactly two forms: they turn either inwards or outwards, starting from the leg to the head. Planar isotopies involving the shifting of the strands and the endpoints horizontally and vertically turns one spiral to the other one, by preserving the rotation number. See Figure \ref{fig:schange} that illustrates such transformation. The assumption on the endpoint directions is clearly a necessary condition for the Morse isotopy relation between $S_1$ and $S_2$.

 \end{proof}
  \begin{figure}[H]
\centering
\includegraphics[width=1\textwidth]{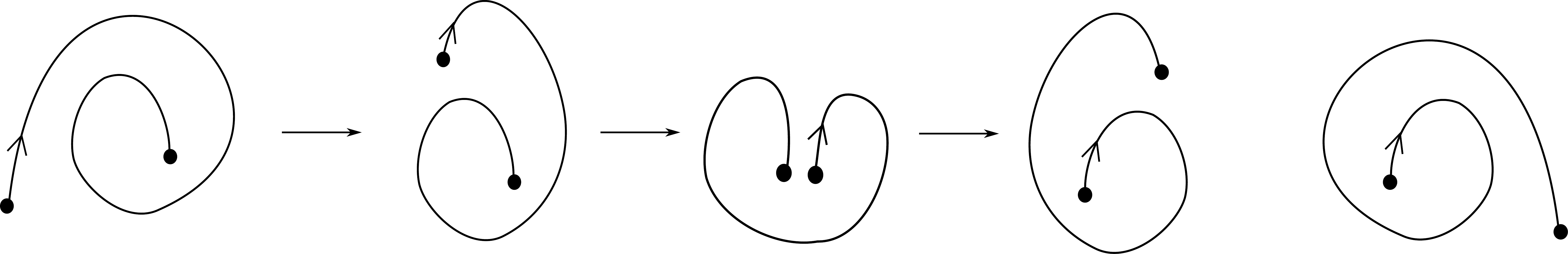}
\caption{Changing a spiral diagram that turns inward into a spiral diagram that turns outward}
\label{fig:schange}
\end{figure}

\begin{definition}\normalfont
A spiral Morse knotoid diagram is an \textit{in-going spiral diagram} if it turns inwards from its leg to its head. 
\end{definition}

\begin{theorem}\label{thm:rota}
Two open-ended oriented state components  whose endpoints have the same tangent directions are Morse isotopic if and only if their rotation numbers are equal.
\end{theorem}
\begin{proof}
 The rotation number is invariant under the Morse isotopy restricted to the open-ended state components with the assumed conventions on their endpoints.\\
Now let us show that if the rotation numbers of two open-ended state components with the coinciding tangent directions at their endpoints  are equal then the state components are Morse isotopic. We do this by proving the following claim.

{\bf Claim:}
An open-ended oriented state component $\lambda$ is Morse isotopic to the in-going spiral Morse knotoid diagram with the same rotation number and the same tangent direction at the endpoints. Specifically,  $\lambda$ has rotation number $0$ if and only if  $\lambda$ is Morse isotopic to a vertical strand with the tangent vectors  at its ends are both directed up.

{\bf Proof of the claim:}\\
For proving this claim, we use two types of Morse isotopy moves: One type is the planar isotopy move called the min-max move  that cancels a pair of sequential maximum and minimum with opposite signs. We illustrate the non-oriented version of the min-max move in Figure \ref{fig:moveset}.  The other type of planar isotopy moves that will be used is what we call \textit{Whitney trick for simple arcs} or the \textit{S-moves}, that is a vertical shifting of the endpoints followed by a max-min move. See Figure \ref{fig:trick}. 
 \begin{figure}[H]
\centering
\includegraphics[width=.75\textwidth]{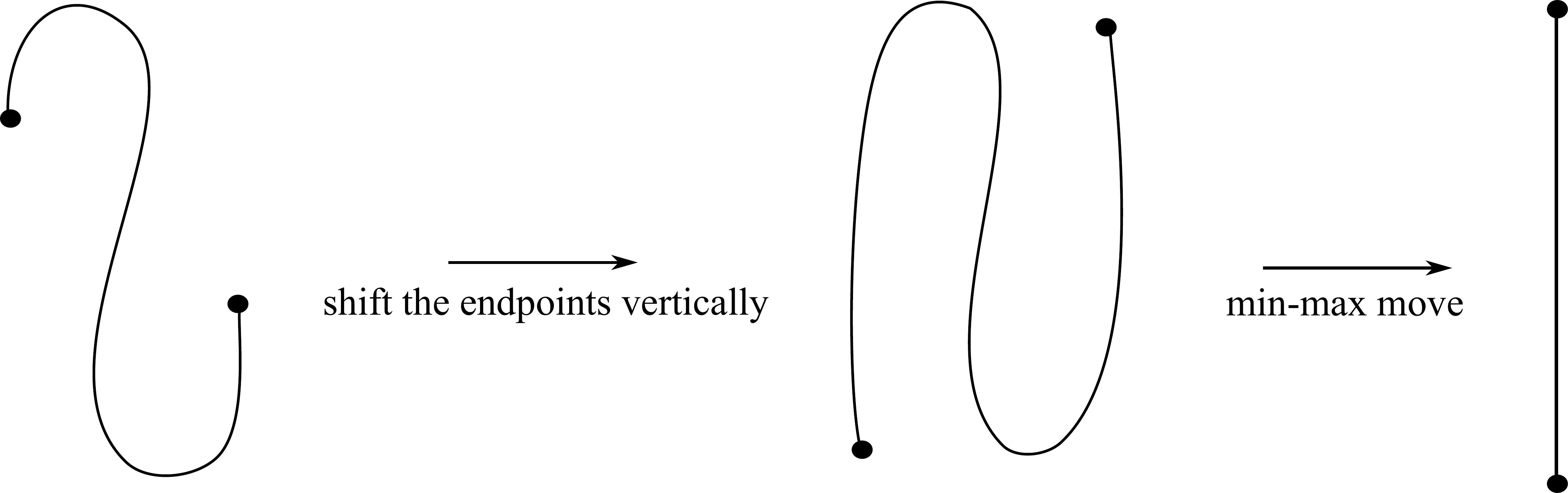}
\caption{The Whitney trick for open-ended state components }
\label{fig:trick}
\end{figure}

Let $\lambda$ be an oriented open-ended state component with rotation number $n \neq 0$. 

As we travel along $\lambda$ starting from its leg, it is clear that the total number of cups and caps along $\lambda$ is at least $|2n|$ since the rotation number is $n$. We first eliminate each pair of sequential maximum and minimum that have opposite signs by the min-max moves. We may require to shift the strands and the endpoints vertically or horizontally  to be able to apply min-max moves without creating any crossings. 
After the elimination, new pairs of  sequential cups and caps with opposite signs might be created in the resulting diagram, which will be eliminated by the min-max moves. 
Clearly, eliminating the canceling pairs does not have any effect on the rotation number. Thus, the final diagram is an oriented diagram free of any pairs of sequential cups and caps and it has rotation number $n$. Thus, we transformed $\lambda$ into a \textit{spiral} diagram. By Lemma \ref{lem:spi}, the statement follows.

If $\lambda$ has rotation number $0$ then $\lambda$ has necessarily up-up directed endpoints, since otherwise the rotation number would be a non-zero half integer. Let the up directed endpoint be the leg of $\lambda$. Since $\lambda$ has trivial rotation number, any cups and caps of $\lambda$ can be brought to be paired by planar isotopy including vertical shiftings of the endpoints so that they are canceled pairwise. By the vertical shifting moves, that pull one endpoint up and one endpoint down, $\lambda$ can be transformed into a vertical strand (trivial strand). This proves the claim.

By proving the claim, we can deduce the required statement.


\end{proof}

\begin{remark}\normalfont
Theorem \ref{thm:rota} can be viewed as a specialization of the Whitney-Graustein theorem \cite{W} for immersed curves in Morse form with endpoint directions specified as we have done.  In the next section we formulate the Morse category for such diagrams. Two such curves are regularly homotopic (in the Whitney-Graustein sense) exactly when one can be obtained from the other by max-min cancellations, and flat versions of the Morse isotopy equivalences. One then uses this equivalence relation to obtain a Morse category proof of the fully generalized Whitney-Graustein Theorem. We omit the details of this proof, since we only need the argument here for curves without crossings.
\end{remark}

 \begin{definition}\normalfont
Let $K$ be a Morse knotoid diagram. The {\it  rotational bracket polynomial} of $K$ is defined as
$$(K)_{rot} = \sum_{\sigma} < K | \sigma > {\delta}^ {m} \lambda^{n},$$
where $\sigma$ is a state,  $< K | \sigma >$ is the product of the coefficients of the smoothings in $\sigma$,  $\delta= -A^2 - A^{-2}$ and $m
$ is the number of circular components in $\sigma$, and  $\lambda$ is the variable assigned to the open-ended state component in $\sigma$ endowed with orientation and  $n$ is the  rotation number of the open-ended state component.
\end{definition}
\begin{theorem}
The rotational bracket polynomial is an invariant of Morse knotoids.

\end{theorem}
\begin{proof}
The proof is based on the invariance of the usual bracket polynomial combined with the invariance of the rotation number. 
\end{proof}
In Figure \ref{fig:rotbrac} we compute the rotational bracket polynomial of an example of a Morse knotoid diagram. Explicit calculation shows $(K)_
{rot}=  \lambda(1 - A^{-4}) + A^2 \lambda^{-1}$.

\begin{figure}[H]
\centering 
\includegraphics[width=1\textwidth]{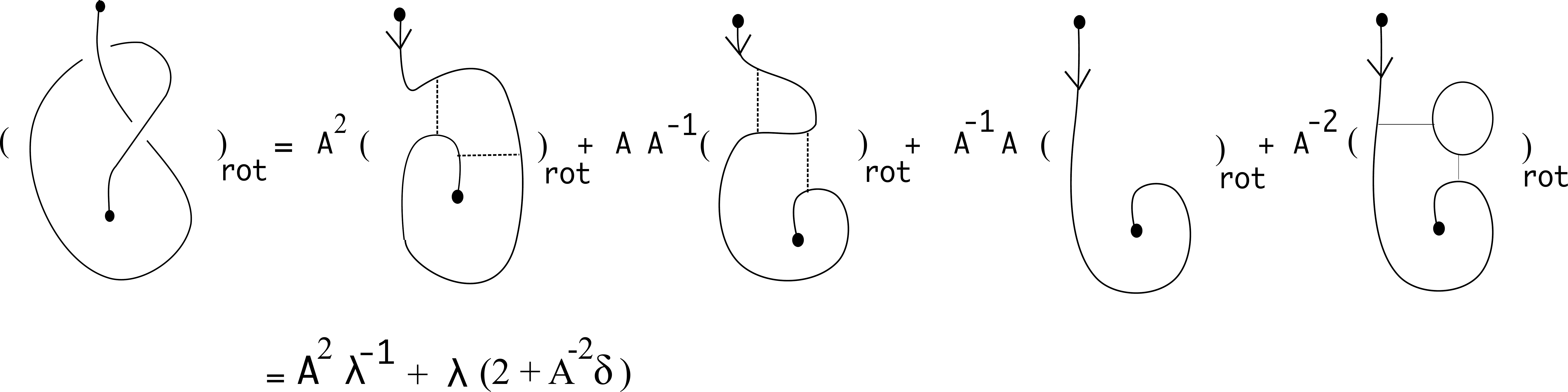}
\caption{Evaluation of the rotational bracket polynomial}
\label{fig:rotbrac}
\end{figure}

\section{Categories for Morse knotoids} \label{sec:Cat}
\subsection{An extended tangle category}
We set a category where Morse knotoids appear as morphisms, by extending the category of tangles in the Morse form. The \textit{objects} of this extended category of tangles in the Morse form, \textbf{Tan} are $[0]$,$[1$],$[2]$,...,$[n]$,  where $n$ is a natural number, and the morphisms of this category are generated by the following elementary tangles: a single vertical strand, right and left-handed crossings, a cap, a cup, and two vertical strands initiating or terminating with a dot, respectively. See Figure \ref{fig:elttangles}. 
 \begin{figure}[H]
\centering
\includegraphics[width=1\textwidth]{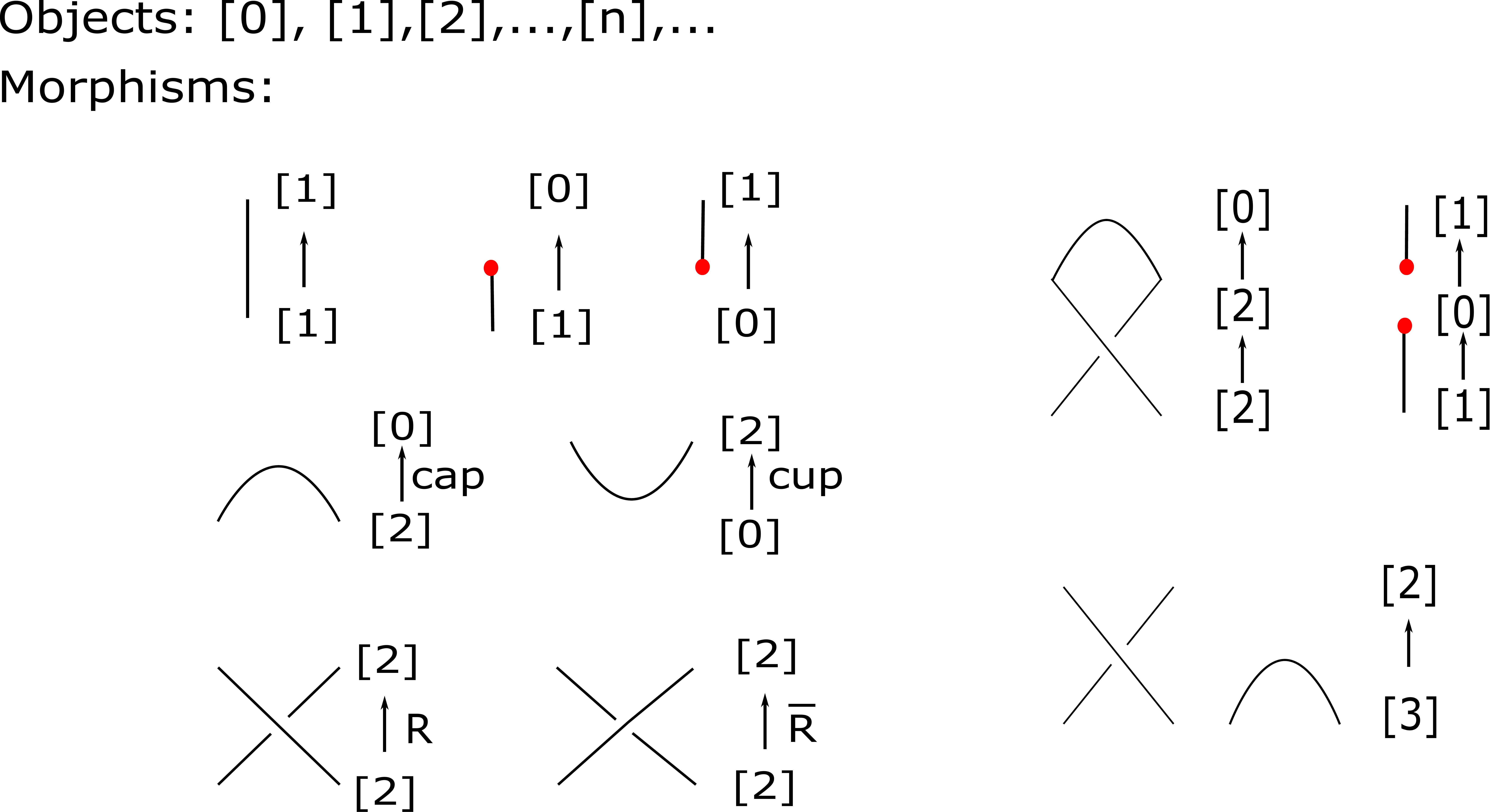}
\caption{The objects and morphisms of the category of Morse tangles}
\label{fig:elttangles}
\end{figure}
For a tangle morphism  $[n] \rightarrow [m]$, $[n]$ and $[m]$ denote the number of free ends of the tangle lying on bottom and top lines, respectively. In particular, the vertical strand $I$ is a morphism from $[1]$ to $[1]$, the crossing morphisms denoted by $R, \overline{R}$ are morphisms from $[2]$ to $[2]$, a cup tangle is a morphism $\cup: [0] \rightarrow [2]$ and the cap tangle is a morphism $\cap: [2] \rightarrow [0]$. A vertical strand initiating with a dot 
\setlength{\unitlength}{.75ex}
\begin{picture}(1.5,1.5)(0,1.8)
        \put(1,0.5){\line(0,1){4}}
        \put(1,0.5){\circle*{1}}
\end{picture} is a morphism from  $[0]$ to $[1]$, a vertical strand terminating with a dot \setlength{\unitlength}{.75ex} \begin{picture}(2,2)(0,2.5)
      \put(1,4){\circle*{1}}
        \put(1,0.5){\line(0,1){4}}
  \end{picture}  is a morphism from  $[1]$ to $[0]$. In this sense the endpoints of a Morse knotoid diagram are not free ends. They are morphisms beginning or terminating in the vacuum $[0]$.

The composition $\mu_2\mu_1$ of two morphisms $\mu_1: [n] \rightarrow [m]$ and $\mu_2: [m] \rightarrow [k]$ where $m \neq 0$, is formed by placing the tangle corresponding to $\mu_2$ on top of the tangle corresponding to $\mu_1$ and joining the output free ends of $\mu_1$ with the corresponding input free ends of $\mu_2$. See Figure \ref{fig:elttangles}. It is clear that the composition operation defined is associative.

 The composition of the morphisms \setlength{\unitlength}{.75ex}
\begin{picture}(1.5,1.5)(0,1)
        \put(1,0.5){\line(0,1){4}}
        \put(1,0.5){\circle*{1}}
\end{picture}: $[1] \rightarrow [0]$ and \setlength{\unitlength}{.75ex} \begin{picture}(2,2)(0,1)
      \put(1,4){\circle*{1}}
        \put(1,0.5){\line(0,1){4}}
  \end{picture}: $[0] \rightarrow [1]$ is a morphism from $[1]$ to $[1]$ given by the disjoint union as shown in Figure \ref{fig:elttangles}.

A tensor structure on {\bf Tan} is given by stacking the elementary tangles side by side from left to right. This gives the tensor product $\otimes$, defined as $[n] \otimes [m] = [n+m]$, where $[n], [m]$ are objects of the category. Note here that $[0] \otimes [n] = [n] = [n] \otimes [0]$. The crossing morphisms $R, \overline{R}$ extend to the \textit{braiding maps},\\
 
  $\sigma_i = I^{\otimes i-1} \otimes R \otimes I^{\otimes k-1}  : [i + k] \rightarrow [i +
 k]$\\
 
  $\overline{\sigma_i}= I^{\otimes i-1} \otimes \overline{R} \otimes I^{\otimes k-1}  : [i + k] \rightarrow [i + k]$,\\
 for any $i, k \geq 1$. This makes the category {\bf Tan}  a braided monoidal category with identity $[0]$, in the sense of Joyal and Street \cite{JS}. Finally, for {\bf Tan} to be a topological category, we assume the identities that are generated by the Morse isotopy moves on the composition and tensor product of the morphisms of {\bf Tan}. 

We call a morphism in {\bf Tan} a \textit{Morse diagram}. Classical and virtual knots, tangles in the Morse form and Morse knotoid diagrams are examples of Morse diagrams.
 A Morse knotoid diagram does not contain any free ends. For this reason any Morse knotoid diagram $K$ can be viewed in {\bf Tan} as a morphism from $[0]$ to $[0]$ that is decomposed into a finite number of tensor products of the elementary morphisms. Figure \ref{fig:comp} illustrates an example. 

\begin{figure}[H]
\centering
\includegraphics[width=.4\textwidth]{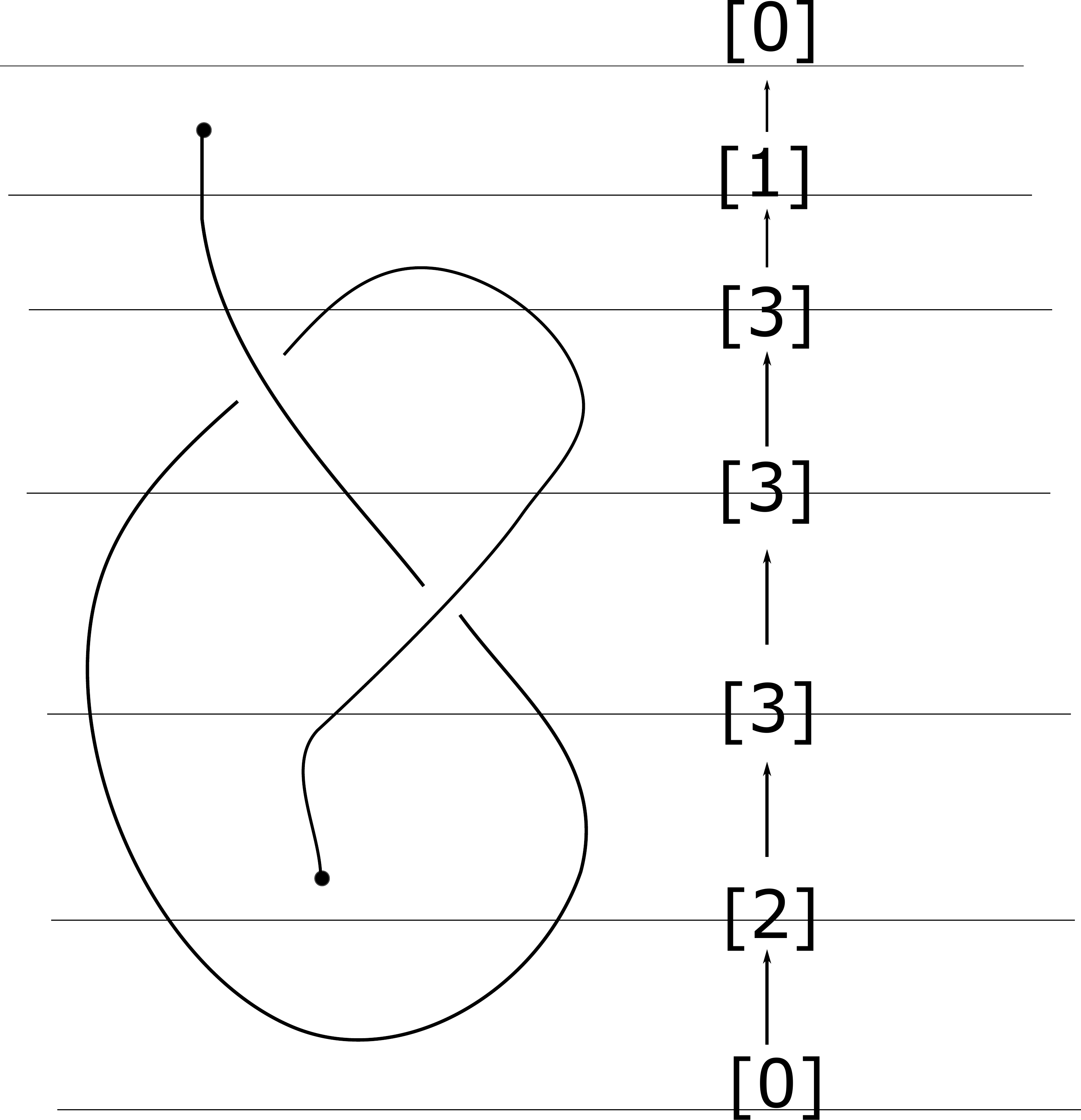}
\caption{A Morse knotoid diagram as composition of morphisms in {\bf Tan}}
\label{fig:comp}
\end{figure}

\subsection{An abstract tensor category}
In this section, we define an abstract tensor category. The reader will see that this category is almost the same as the tangle category {\bf Tan} that we discuss above. The abstract tensor category is designed so that there can be a direct functor from this category to matrix algebra, and so that it parallels the structure of partition functions in physics. In this category each morphism is notated with dots at its free ends,  as illustrated in Figure \ref{fig:t1}. These dots are mapped by a matrix algebra functor to matrix indices. We give the details of the matrix representation in Section \ref{sec:quantumone}. Let us now give a precise description for the abstract tensor category.

The objects of the abstract tensor category {\bf ATC}, are generated by the objects $k$ and $V$, where $V$ can be considered as a  free module and $k$ as the ground ring of $V$. The tensor product of $V$ with itself gives rise to distinct objects\\
 $\underbrace{V\otimes V\otimes... \otimes V }_{n} = V^{\otimes n}$ for every $n\geq 1$. 
 The following identities hold on the product of the objects. \\
i.) $(V\otimes V) \otimes V = V \otimes (V\otimes V),$\\
ii.) $k \otimes V= V = V \otimes k$.

 The morphisms of this category are generated by the \textit{identity} operator $I: V \rightarrow V$, the cup and cap operators, $\cup: k \rightarrow V \otimes V$, $\cap: V \otimes V  \rightarrow k$, \\
 two braiding operators $R,  \overline{R} : V \otimes V \rightarrow V \otimes V,$\\
 and the operators $\eta: k \rightarrow V$, $\epsilon: V\rightarrow k$.
  
 The tensor product  $A \otimes B$ of two operators $A: V^{\otimes n} \rightarrow V^{\otimes m}$  and $B: V^{\otimes i }\rightarrow V^{\otimes j}$ is defined as $A \otimes B : V^{\otimes n+i} \rightarrow V^{\otimes m+j}$. Then, a Morse knotoid diagram $K$ can be seen as a composition of the tensor products of operators, as exempflied in Figure \ref{fig:t1}.
    \begin{figure}[H]
    \centering
    \includegraphics[width=.5\textwidth]{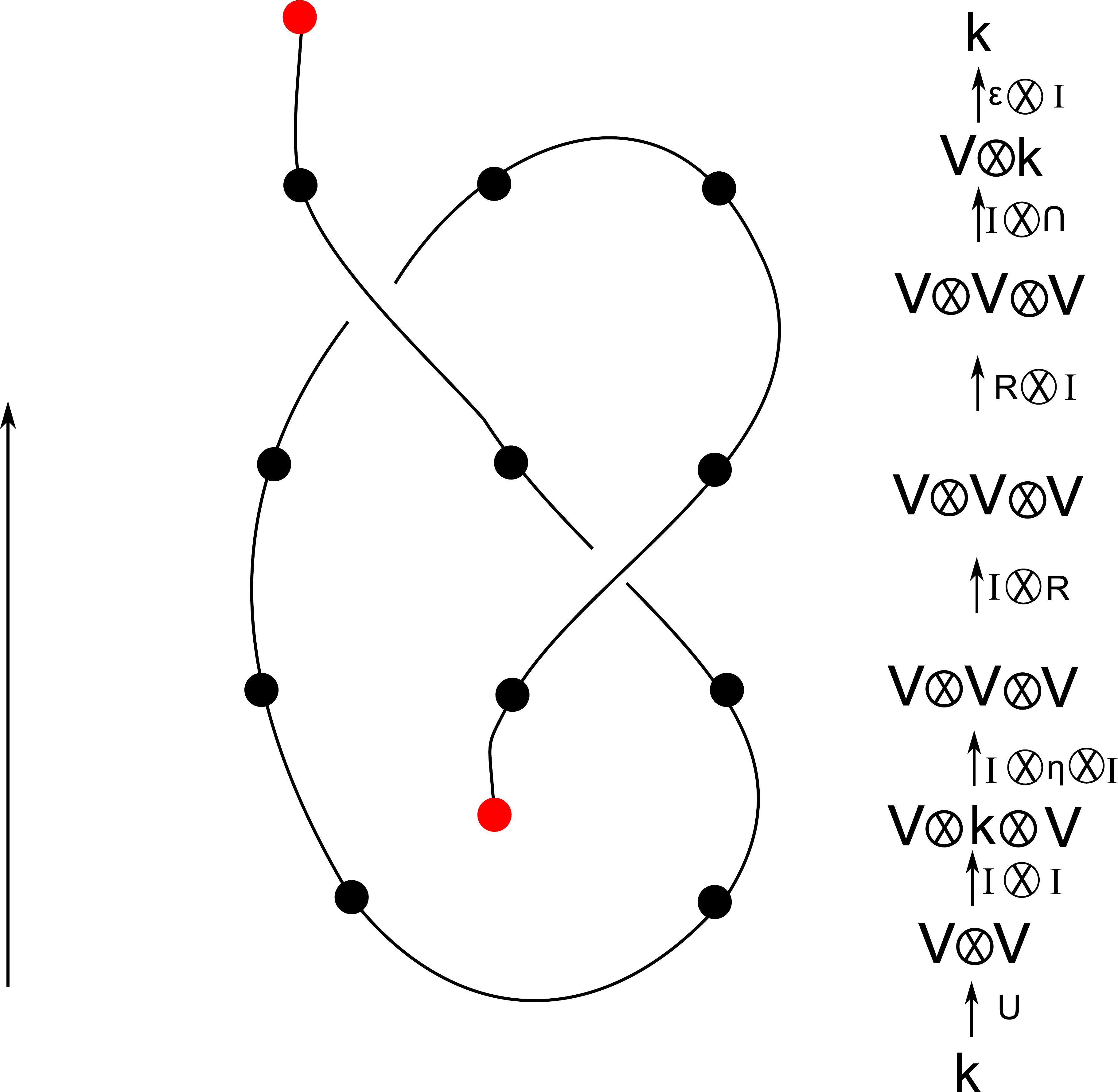}
    \caption{A Morse knotoid diagram as composition of morphisms in {\bf ATC}}
    \label{fig:t1}
    \end{figure}

We assume the morphisms $R, \overline{R}$ to be characterized only by the identities motivated by the type II and type III Morse isotopy moves.
$$R\overline{R}=I^{\otimes 2}= \overline{R}R,$$
$$(I \otimes R)(R \otimes I)(I \otimes R) = (R \otimes I)(I \otimes R)(R\otimes I)$$
$$(I \otimes \overline{R})(\overline{R} \otimes I)(I \otimes \overline{R}) = (\overline{R} \otimes I)(I \otimes \overline{R})(\overline{R}\otimes I)$$

Furthermore, the following identities that are the categorical descriptions of vertical  and horizontal shiftings in the plane,
$$(I^{\otimes n} \otimes B)(A\otimes I^{\otimes i})=(A\otimes I^{\otimes j})(I^{\otimes m} \otimes B) = A \otimes B,$$
where $A: V^{\otimes m} \rightarrow V^{\otimes n}$, $B: V^ {\otimes i } \rightarrow V^{\otimes j}$, and 

$$A \otimes B = AB = B \otimes A,$$
where $A : k \rightarrow V^{\otimes n}$, $B: V^{\otimes m} \rightarrow k$,  are  assumed for the tensor product of morphisms.
See Figure \ref{fig:tensor1} for illustrations of these identities, respectively.


It is clear that {\bf ATC} is a braided monoidal category with braiding maps generated by $R, \overline{R}$. Furthermore, there is a functor from {\bf Tan} to {\bf ATC} sending $[0]$ to the ground ring $k$,  $[n]$ to $V^{ \otimes n}$, for any $n >0$, and the cup morphism to $\cup$ operator, the cap morphism to $\cap$ operator, the crossing morpisms to $R, \overline{R}$, respectively and  \setlength{\unitlength}{.75ex} \begin{picture}(2,2)(0,2.5)
      \put(1,4){\circle*{1}}
        \put(1,0.5){\line(0,1){4}}
  \end{picture}  to $\epsilon$, \setlength{\unitlength}{.75ex}
\begin{picture}(1.5,1.5)(0,1.8)
        \put(1,0.5){\line(0,1){4}}
        \put(1,0.5){\circle*{1}}
\end{picture} to $\eta$.

 \begin{figure}[H]
    \centering
    \includegraphics[width=1\textwidth]{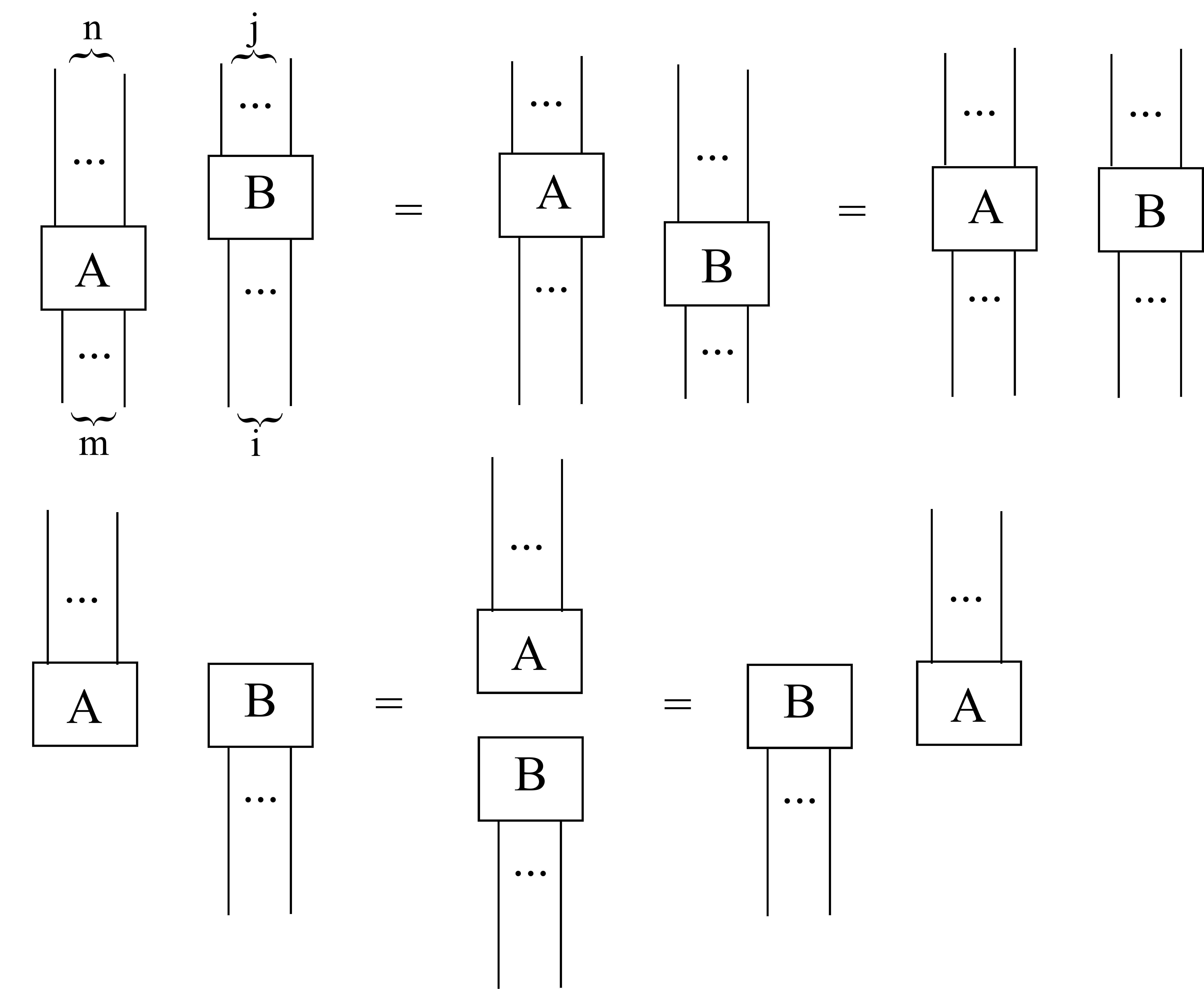}
    \caption{Vertical and horizontal shiftings}
    \label{fig:tensor1}
    \end{figure}

\section{Quantum invariants of Morse knotoids}\label{sec:quantum}
In the previous section, we have discussed how Morse knotoids can be seen as morphisms of a category. In this section we show explicitly how such morphisms become invariants of  knotoids by using an appropriate functor to module categories. In fact, we shall begin this section with a specific description of a {\it partition function} for knotoids, where we associate directly matrices to the cups, caps and crossings of a Morse knotoid diagram. The partition function is then a generalized matrix product, as we detail below.

\subsection{A matrix of partition functions of Morse knotoids}\label{sec:quantumone}
In this subsection we give a method of assigning invariants to the knotoid diagrams in the Morse form that is analogous to partition functions in statistical mechanics. We show in Section \ref{sec:cattwo} below how this 'tensor network' definition can be formulated in terms of a functor on the category.

Let $\mathcal{I}= \{1,2,...,n\}$ denote an index set for some $n \in \mathbb{Z}^+$ and ${\bf k}$ denote a commutative ring. We label each free end of the elementary morphisms in the Morse category {\bf Tan} with an index $i \in \mathcal{I}$. In this way, each elementary morphism will be assigned to a matrix. In particular, a vertical strand whose input and output free ends are labeled with $i$ and $j$, respectively, is identified with the Kronecker delta $\delta_{i}^{j}$. See Figure \ref{fig:krnckr}. Thus, the $n \times n$ identity matrix with the entries $\delta_{i}^{j}$, $i, j \in \mathcal{I}$ is associated to a vertical strand.

  \begin{figure}[H]
    \centering
    \includegraphics[width=1\textwidth]{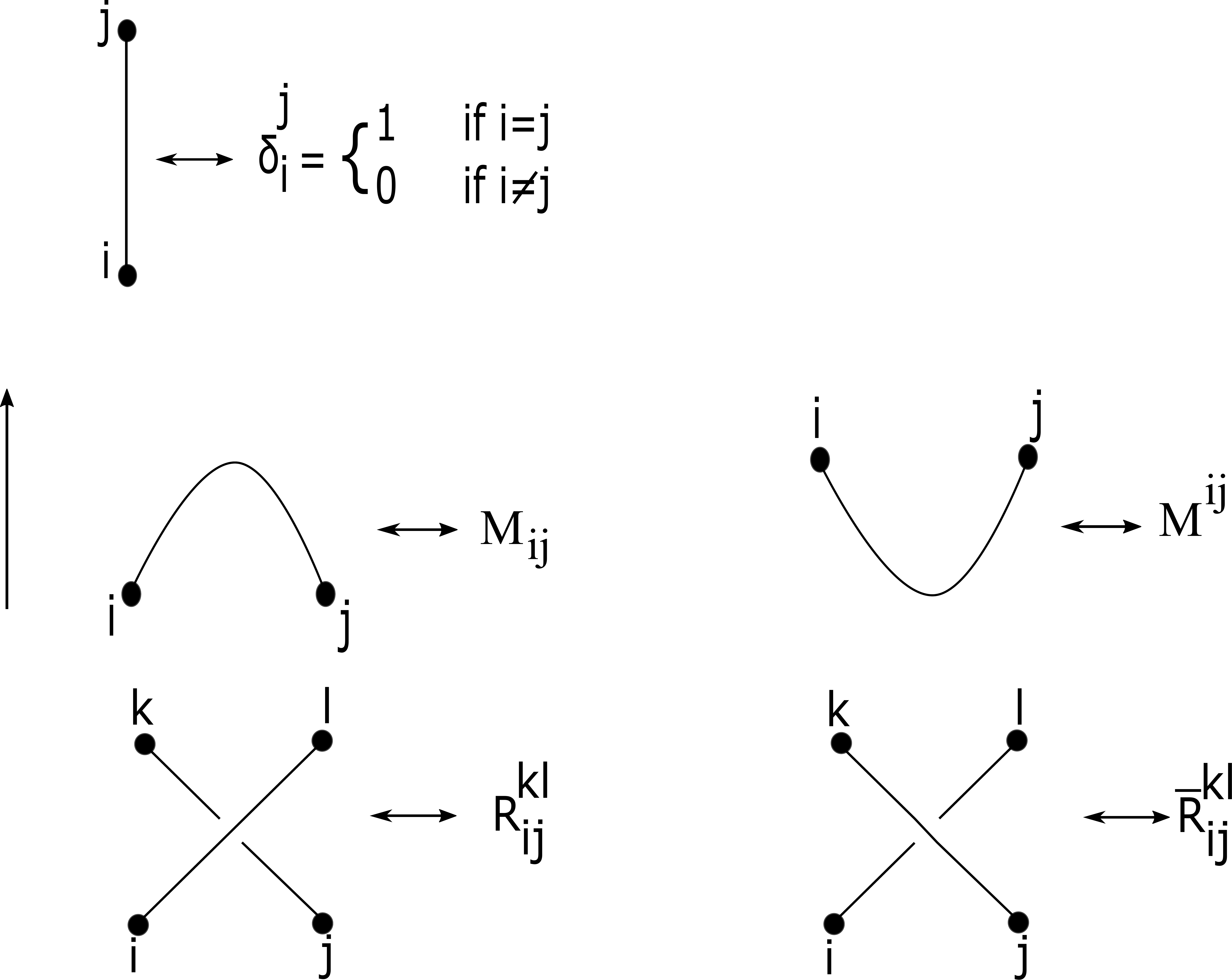}
    \caption{Elementary tangles as matrix elements}
    \label{fig:krnckr}
    \end{figure}

In a similar manner, to a cup and a cap, square matrices $[M^{ij}], [M_{ij}] \in M_{n\times n} (k)$ with entries $M^{ij}, M_{ij}$ are associated, respectively. To a right-handed and left-handed crossing, square matrices $R= [R^{kl}_{ij}],  \overline{R}=[\overline{R}^{kl}_{ij}] \in M_{2n \times 2n} (k)$ are associated. 

\begin{remark}\normalfont
 Note that in the tensor network formulation a vertical segment with top or bottom endpoint labeled with $i$, always behaves as a Kronecker delta with respect to any other index on the line.
 \end{remark}





%
 By labeling each input and output of the elementary tangles with indices from the chosen index set, a Morse diagram can be assigned to the  product of the matrix entries corresponding to the elementary tangles and the Kronecker deltas corresponding to the vertical vertical strands at each level, with labels over them.  
 
\begin{definition}\normalfont
A \textit{partition function} of a Morse diagram in the Morse category {\bf Tan} is defined as the matrix obtained by taking the sum of the products of the scalar matrix entries including the 
Kronecker deltas, over the repeated indices on the internal configurations (cups, caps, crossings or vertical strands). 
\end{definition}

Clearly, a partition function of a closed knotted diagram in the Morse form takes values in the ground ring $k$. For instance, the value of a partition function of a circle is a scalar given by  
  $ \sum_{a,b \in \{1,...,n\}}M_{ab} M^{ab}$. 
  
 If a Morse diagram is a tangle or a knotoid diagram, each choice of indices on its free ends or on its endpoints, respectively,  yields a partition function. Then, we have a matrix of partition functions for the tangle or knotoid diagram. We call this matrix the \textit{partition function matrix} of the Morse diagram. In particular, if $K$ is a Morse knotoid diagram, there are $n$ choices for labeling each of its endpoints with indices from the set $\mathcal{I}$.  For each choice of labelings at the endpoints, a partition function is assigned to $K$. We denote the partition function of the Morse knotoid diagram $K$ with its endpoints labeled with $a,b \in \mathcal{I}$, by $Z_{a}^{b}$, where $a$ holds for the label at the endpoint of $K$ with lower height and $b$ holds for the label at endpoint of $K$ at the higher height. See Figure \ref{fig:par}.
  Then, the partition function matrix of $K$ is an $n \times n$ matrix that is given with entries $Z_{a}^{b} \in k$, and is denoted by $Z_K= [Z_{a}^{b}]$.

For example, the partition function of a vertical strand with fixed labels on its input and output ends is given by the concatenation of two Kronecker deltas and the summation over the internal index $i$, $\sum_{i \in \mathcal{I}} \delta_{a}^{i} \delta_{i}^{b} = \delta_{a}^{b}$. Thus the partition function of a vertical strand with fixed input output indices $a, b$ is equivalent to the Kronecker delta. This expression can be further reduced by labeling the vertical strand with only one index at an interior point. See Figure \ref{fig:reduced}. Assuming one label for each vertical strand results in a Morse diagram with one label between any two critical points.  Figure \ref{fig:par} illustrates the Morse knotoid diagram given in Figure \ref{fig:t1} with the reduced number of labels on it. We call this labeled Morse diagram the \textit{reduced diagram}. 
Reducing the number of labels on a Morse diagram in this way induces a Kronecker delta-free definition for the partition function for Morse diagrams. 
\begin{figure}[H]
  \centering
  \includegraphics[width=.2\textwidth]{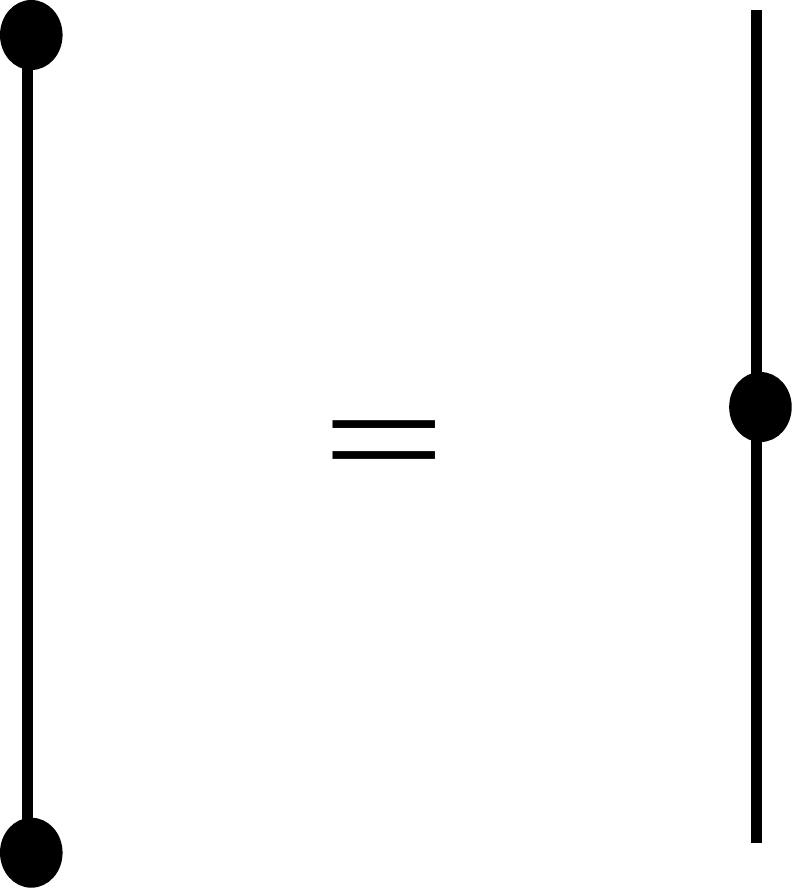}
  \caption{Reduction of labelling on a vertical strand}
  \label{fig:reduced}
  \end{figure}
  \begin{figure}[H]
 \centering
 \includegraphics[scale=.25]{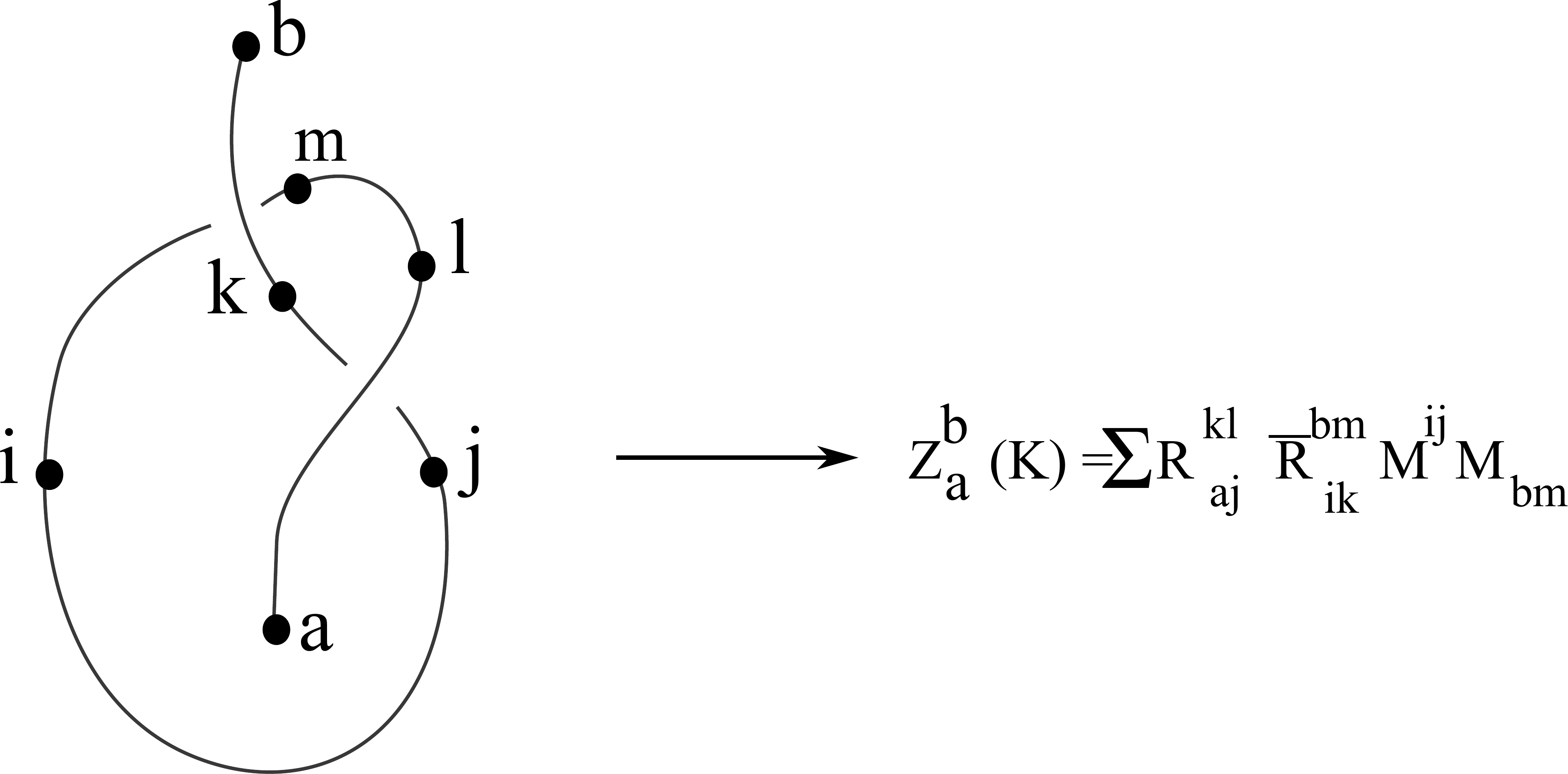}
 \caption{A reduced Morse knotoid diagram}
  \label{fig:par}
  \end{figure}

\begin{definition}\normalfont
The {\it reduced partition function} of a Morse diagram is the sum of the products of the scalar matrix entries, over the repeated indices on the internal configurations of the reduced Morse diagram. \end{definition}

\begin{lemma}
The partition function matrix and the reduced partition function matrix of a Morse knotoid diagram are equal to one another.\end{lemma}

\begin{proof}
By the discussion above, it is explicit that reducing the labels on the internal vertical strands of a Morse knotoid diagram will not affect the partition function. Therefore, the partition function matrix remains the same.
\end{proof}
We will be working with reduced labeled diagrams and reduced partition functions as it is most efficient in  writing.

 In order for the partition function $Z^{b}_{a}$ to turn out to be a topological invariant for the isotopy class of a Morse knotoid diagram $K^{a}_{b}$ whose endpoints are labeled with $a$ and $b$ from $\mathcal{I}$, we assume the identities on the cup, cap and crossing matrices given in Figure \ref{fig:zero} and their variations induced by variations of the Morse isotopy moves.

\begin{figure}[H]
\centering \includegraphics[scale=.25]{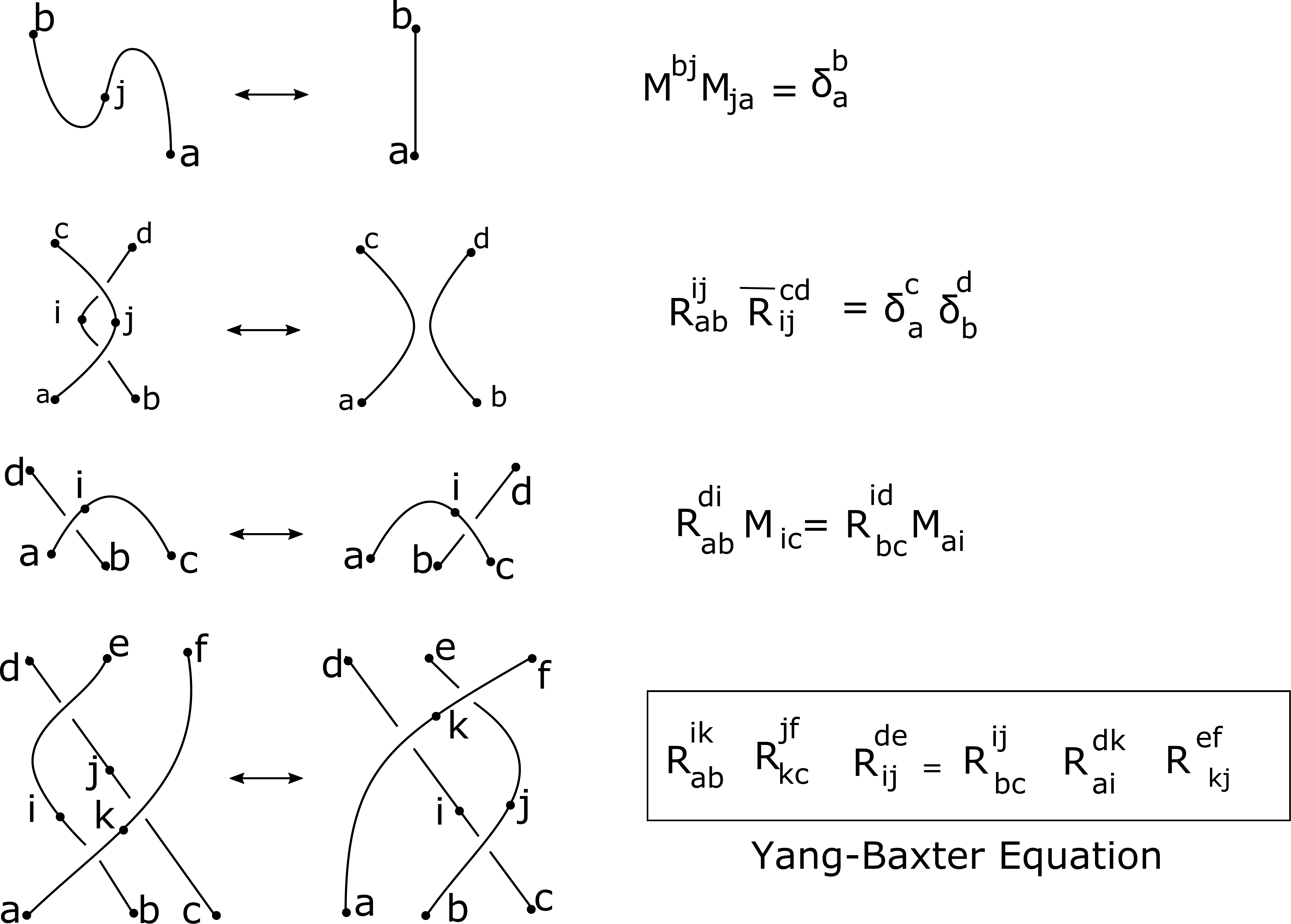}
\caption{The identities for the matrices}
\label{fig:zero}
\end{figure}

 As we see from the figure, the invariance under the vertical type II Morse isotopy move imposes that the matrix $\overline{R}$ is the inverse of the matrix $R$. The invariance under the given vertical type III Morse isotopy move and its mirror image impose the following identities on the matrices $R$ and $\overline{R}$, respectively. These identities are known as the \textit{Yang-Baxter equation}.  
$$R^{ij}_{ab}R^{kf}_{jc}R^{de}_{ik} = R^{ij}_{bc} R^{dk}_{ai} R^{ef}_{kj},$$
$$\overline{R}^{ij}_{ab}\overline{R}^{kf}_{jc} \overline{R}^{de}_{ik} = \overline{R}^{ij}_{bc}\overline{R}^{dk}_{ai} \overline{R}^{ef}_{kj}.$$
We have the following theorem.

\begin{theorem}\label{thm:relations}
Suppose that matrices $R$ and $\overline{R}$ are inverses, and satisfy the Yang-Baxter equation and the interrelation with $M^{ij}$ and $M_{ij}$ given by the slide moves. Suppose that $[M_{ij}]$ and $[M^{ij}]$ matrices are inverses.  Let  $Z_{b}^{a}$ be the $n \times n$ matrix whose entries are the values of the partition functions of Morse knotoid diagrams $K^{b}_{a}$ whose endpoints are labeled with $a, b$ as $a, b$ runs over $\mathcal{I}$. Then $Z_{b}^{a}$ is a Morse knotoid invariant.
\end{theorem}

Partition functions of Morse knotoids can be given in the form of a state sum, like in the case of the bracket polynomial of classical knots and links \cite{Ka5}. We will discuss this in the subsequent sections. 

Our partition functions are constructed in analogy with the partition functions in statistical mechanics where they are defined on a graph corresponding to a physical system as generalized matrix products, and the matrices assigned to the nodes of the graph are called the vertex weights. See \cite{Jones1, Ka4}. Before giving examples of Morse knotoid invariants that can be modeled as partition functions, we would like to highlight the relation between the category of Morse knotoids {\bf Tan}, the category of abstract tensors {\bf ATC} and the category of matrices {\bf MAT}. 

\subsubsection{An overview of the interrelations of the categories}\label{sec:cattwo}
 A linear operator can be assigned to each generating morphism (or elementary tangle) in {\bf Tan} as follows.
 Let $V$ be a module of finite rank over a commutative ring ${\bf k}$ with basis $\{e_i\}_{i \in \mathcal{I}}$ where $\mathcal{I} = \{1,...,n\}$. Let $\{e_{ij}\}_{i,j \in \mathcal{I}}= \{e_i \otimes e_j\}_{i, j \in \mathcal{I}}$ denote the tensor basis for $V \otimes V $. 
 
Each index in $\mathcal{I}$ represents a basis element of $V$ in a unique way with the identification $a \leftrightarrow e_a$. Then, a vertical strand can be assigned to the identity operator $I~:~V\rightarrow V$ given by\\
 $$I (e_a) = \sum_{b \in \mathcal{I}}\delta^{b}_{a} e_a,$$
where $\delta^{b}_{a}$ is the Kronecker delta. \\

The operator associated with the cup tangle $\cup: k \rightarrow V \otimes V$ is given by 
$$\cup(1) = \sum_{a, b \in \mathcal{I}} M^{ab} e_{ab},$$ 
and the operator associated with the cap tangle $\cap: V\otimes V \rightarrow k$ is given by 
$$\cap(e_{ab}) = M_{ab},$$ 
where $M^{ab}$ and $M_{ab}$ are the matrix entries associated to the labeled cup and cap.

The right and the left-handed crossing operators $R, \overline{R} : V \otimes V \rightarrow V \otimes V$  are determined with the $R$ and $\overline{R}$ matrices: 
$$R(e_{ij})= \sum_{k, l \in I} R^{kl}_{ij} e_{kl},$$
$$\overline{R}(e_{ij})= \sum_{k, l \in I} \overline{R^{kl}_{ij}} e_{kl}. $$

Finally, a vertical strand initiating with an endpoint labeled with $i$ is associated to an operator $\eta_i:  {\bf k} \rightarrow V$  such that $\eta_i(1)=e_i$, and a vertical strand terminating with an endpoint labeled with $j$ is associated to an operator $\epsilon_j: V \rightarrow {\bf k}$ such that $\epsilon_j(e_k)=1$ if $j=k$, $0$ otherwise.\\

For instance, a circle decomposes into a cup and a cap operator and takes value that is equal to the value of its partition function as shown below.\\
$$\cap \circ \cup(1)= \cap ( \sum_{a, b \in \mathcal{I}} M^{ab} e_{ab})= \sum_{a, b \in \mathcal{I}} M^{ab} \cap(e_{ab})=  \sum_{a,b \in \mathcal{I}}M_{ab} M^{ab}. $$

This observation generalizes to the following lemma.

\begin{lemma}\label{lem:composible} 

Let $S$ and $T$ be composible morphisms (that is, the number of output free ends of $T$ is equal to the number of input free ends of $S$) in the Morse category {\bf Tan}. 
 Let $[S], [T]$ be the matrices of the morphisms $S$, $T$ that determine the corresponding linear operators, respectively. Let  the composition $S \circ T$ be denoted by $ST$ and $Z(ST)$ denote the partition function matrix of $ST$. Then we have,
 \begin{center}
   $ Z (ST) = Z(S)Z(T) = [S][T] = $ Matrix of the morphism $ST$.  
   \end{center}
\end{lemma} 
\begin{proof}
By definition, the partition function matrix $Z(S)$ of a morphism $S: V ^{\otimes k} \rightarrow V^{\otimes m}$ where $k, m \in \mathbb{Z}^{+} $, is the matrix of partition function evaluations $Z^{j_{1}...j_{m}}_{{i_1}...{i_k}}(S)$ over the labelings of the diagram of the morphism. That is, $Z(S)=  [Z^{j_{1}...j_{m}}_{{i_1}...{i_k}}(S)]$. Clearly, $Z(S) = [S]$. Similarly, $Z(T)= [Z^{j_{1} j_{2}...j_{k}}_{j_{1}...j_{n}}(T)]$ and so $Z(T) = [T]$ for a morphism $T: V^{\otimes n} \rightarrow V^{\otimes k}$, $n \in \mathbb{Z}^{+}$.

 The partition function of the composition $ST$ is obtained 
 by summing  the products of the matrix entries of $Z(S)$ and the matrix entries of $Z(T)$ over all index assignments to the common indices of $S$ and $T$. This is just a description of the matrix product of the matrix $[S]$ and the matrix $[T]$. 
 
The composed morphism $ST: V^{\otimes n} \rightarrow V^{\otimes m}$ is given by the following.  \\

 $S T (e_{i_{1}...i_{n}}) = S (\sum_ {\substack{a_1, ..., a_k \in I}} T^{a_1...a_k}_{i_1...i_n} e_{a_{1}...a_{k}}) = \sum_ {\substack{a_1, ..., a_k \in I}} T^{a_1...a_k}_{i_1...i_n} S(e_{a_1...a_{k}})$\\
 
 \hspace{2cm} $ = \sum_ {\substack{a_1, ..., a_k \in I}} T^{a_1...a_k}_{i_1...i_n} S^{i_{1} ...i_{m}}_{a_{1}...a
 _{k}} e_{i_{1}... i_{m}},$\\
 
 where $e_{i_{1}...i_{n}}= e_{i_1} \otimes ...\otimes e_{i_n}$, for $i_{k} \in \mathcal{I}$. \\
 
 Thus, the matrix of the morphism $ST$ is equal to the partition function matrix of $ST$. \\

\end{proof}

  
\begin{note} \normalfont
The reader can verify that the matrix of the composition of the morphisms $(R\otimes I)(I \otimes R)(R \otimes I)$ in the Yang-Baxter equation is given by $\sum_{i,j,k \in I}  R^{ij}_{ab}R^{kf}_{jc}R^{de}_{ik}$ that is equal to the partition function of the corresponding morphism with input and output labels  $a,b,c$ and $d,e,f$, respectively.

\end{note}

We have the following theorem for Morse knotoid diagrams.
\begin{theorem}
Let $K$ be a Morse knotoid diagram. The matrix whose entries are the values of the partition functions $Z^{a}_{b}$ assigned to $K^{a}_{b}$, for all $a,b \in \mathcal{I}$, is equal to the matrix of the linear morphisms associated to $K$.
\end{theorem}

\begin{proof}
By definition of a Morse knotoid diagram, $K$ can be sliced horizontally so that each horizontal strip contains at most one critical point. Therefore, $K$ can be regarded as a composition of a finite number of linear morphisms associated to elementary tangles stacked with vertical strands forming $K$.  Then, by Lemma \ref{lem:composible} the statement follows.

 \end{proof}


 \subsection{The bracket partition function}\label{sec:bracpartition}
The bracket polynomial is a regular isotopy invariant of knotoids  in $S^2$ and $\mathbb{R}^2$ \cite{Tu}. In this section, we show that it can be defined as a Morse isotopy invariant of Morse knotoids in the form of a partition function.  

Let $K$ be a Morse knotoid diagram and $V$ be a free module of rank $2$ over $\mathbb{C}[A , A^{-1}]$, the ring  of Laurent polynomials with complex coefficients. Let $\{e_1, e_2\}$ be a basis for $V$.   We label the endpoints of $K$ with $a,b \in\{1,2\}$. We will consider the cases where the tangent directions at the endpoints differ.  We denote the Morse knotoid diagram $K$ with labels $a,b$ at its endpoints by $K_{b}^a$ if the tangent vectors at the endpoints of $K$ are directed upwards, $K_{ab}$ or $K^{ab}$ if the endpoints of $K$ are directed upwards and downwards or downwards and upwards , respectively, with respect to the bottom to top vertical direction of the plane.
 
Let the cups and caps of $K$ are both assigned to the matrix  $M$, where
 $$ M= \begin{bmatrix}
0& iA \\
-iA^{-1} & 0
\end{bmatrix}.
$$
That is, $M^{ij}=M_{ij}$ for each $i,j \in \mathcal{I}$, and let the right-handed and left-handed crossings are assigned to the matrices $R, \overline{R}$, where\\

$R =  \begin{bmatrix}
A &0 & 0& 0\\
0&0& A^{-1}& 0\\
0 & A^{-1}  &A- A^{-3} &0\\
0& 0& 0& A \end{bmatrix}, 
\overline{R} =  \begin{bmatrix}
A^{-1} &0 & 0& 0\\
0& A^{-1}-A^3&A& 0\\
0 & A  &0 &0\\
0& 0& 0& A^{-1} \end{bmatrix}$.
\\
\begin{lemma}\label{lem:id}
The entries of the crossing matrices $R$ and $\overline{R}$ satisfy the following identities.
 \begin{equation}
R^{ij}_{kl}= A \delta^{i}_{k}\delta^{j}_{l}  + A^{-1} M^{ij}M_{kl}, \label{eqn:id}
\end{equation}

\begin{equation}\label{eqn:id2}
\overline{R}^{ij}_{kl} =  A ^{-1} \delta^{i}_{k}\delta^{j}_{l}  + A M^{ij}M_{kl}.
\end{equation}
\end{lemma}
\begin{proof}
The identities can be verified directly by substituting the values of the Kronecker deltas and entries of the $M$ matrix for the corresponding crossing matrix entries.
\end{proof}


\normalfont
With this choice of the $M$ and $R$ matrices, we have two possible ways to construct the bracket polynomial as a partition function. The first construction is of a more direct nature, given as follows. 
\begin{definition}\normalfont
The \textit{bracket partition function} $<K_{b}^a>$ of a Morse knotoid diagram $K^{a}_{b}$ with fixed indices $a$ and $b$ at the endpoints of $K$ is the sum of products of the entries $M$ and $R$ matrices corresponding to the internal configurations of $K_{b}^a$ labeled with indices from $\mathcal{I}$, and the sum is taken over $I$. \\
See Figure \ref{fig:bracket} for a computational example.
\end{definition}
Note that the bracket partition functions $<K_{ab}>$, $<K^{ab}>$ of the Morse knotoid diagrams $K_{ab}$ and $K^{ab}$ with fixed labels $a,b$ and up-down, down-up  tangent directions at the endpoints, respectively, are defined as above.

\begin{proposition}\label{prop:bracket}
The bracket partition function of a Morse knotoid diagram with fixed labels on its endpoints is invariant under the Morse isotopy moves. 
\end{proposition}
\begin{proof}
It is clear that the Morse isotopy moves do not change the labeling on the endpoints.
We need to check whether the $M$ and $R$ matrices chosen for the bracket partition function, satisfy the relations given in Theorem \ref{thm:relations}. The verification is straightforward and left to the reader. 
 \end{proof}
  \begin{proposition}
The bracket partition function of $K_{b}^a$ (and of $K_{ab}$, $K^{ab}$) induces a $2 \times 2$ matrix  $[<K_{b}^a>]$, as the indices at the endpoints $a, b$ run over $\{1,2\}$. The induced matrix $[<K_{b}^a>]$ is a matrix invariant of Morse knotoids.
\end{proposition}
\begin{proof}
Since there are two indices to label each endpoint of $K$, four partition functions are assigned to $K$ with specific choices of indices at its endpoints. Let the  $ij^{th}$ entry of the $2\times 2$ matrix $[<K_{b}^a>]$  be the partition bracket function $<K_{j}^i>$ assigned to $K_{j}^{i}$, for $i, j\in \{1,2\}$.  $<K_{j}^i>$ is a Morse isotopy invariant for every $a,b \in \{1, 2\}$, by Proposition \ref{prop:bracket}. Then, the $2\times 2$ matrix $[<K_{b}^a>]$ is invariant under the Morse isotopy. 
\end{proof}
\begin{figure}[H]
\centering
\includegraphics[width=.75\textwidth]{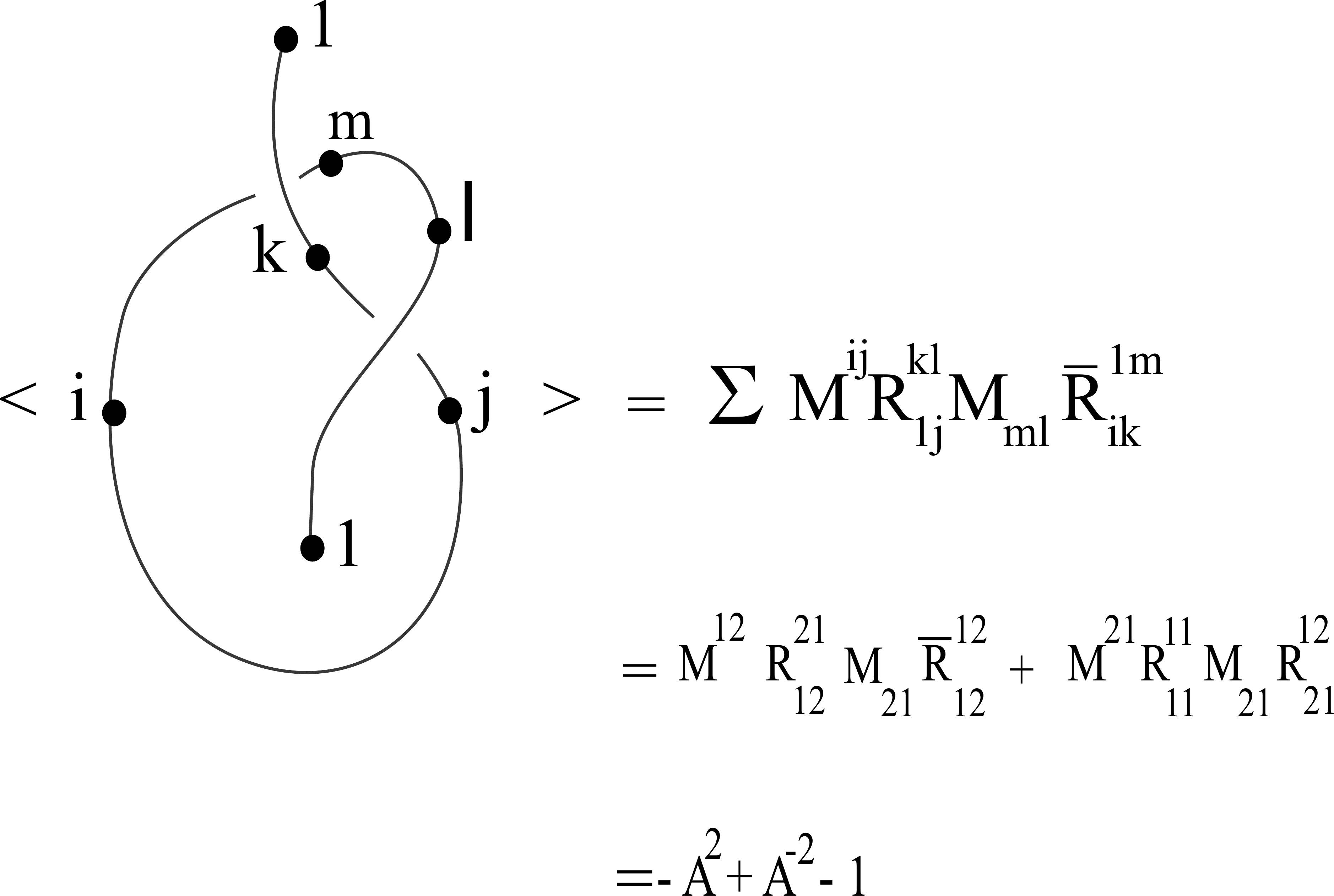}
\caption{The evaluation of $<K_1^1>$}
\label{fig:bracket}
\end{figure}

\subsubsection{The bracket partition function via state sum}

As the reader will notice, the identities \ref{eqn:id} and \ref{eqn:id2} in Lemma \ref{lem:id} are in accordance with the bracket state expansion at a crossing given in \cite{Tu}. This gives the idea of defining the bracket partition function as a summation over the states obtained by this expansion. Each crossing of the Morse knotoid diagram $K_{b}^a$ with fixed labels $a,b$ at its endpoints, is smoothed in two possible ways as in the usual bracket case. From each possible combination of smoothing, we  obtain a finite number of circular  state components and exactly one open-ended state component containing the two endpoints labeled with indices $a$ and $b$. Since the partition function value of a circle is $\sum_{\substack{i,j \in \{1,2\}}} M^{ij}M_{ij} = (M^{12})^2 + (M_{21})^2 = -A^2 -A^{-2}$, where $M$ is the matrix we have assumed above  for the cups and caps, each circular state component contributes to the bracket partition function with the usual bracket value, see Figure \ref{fig:circ}.  
\begin{figure}[H]
\centering
\includegraphics[width=.35\textwidth]{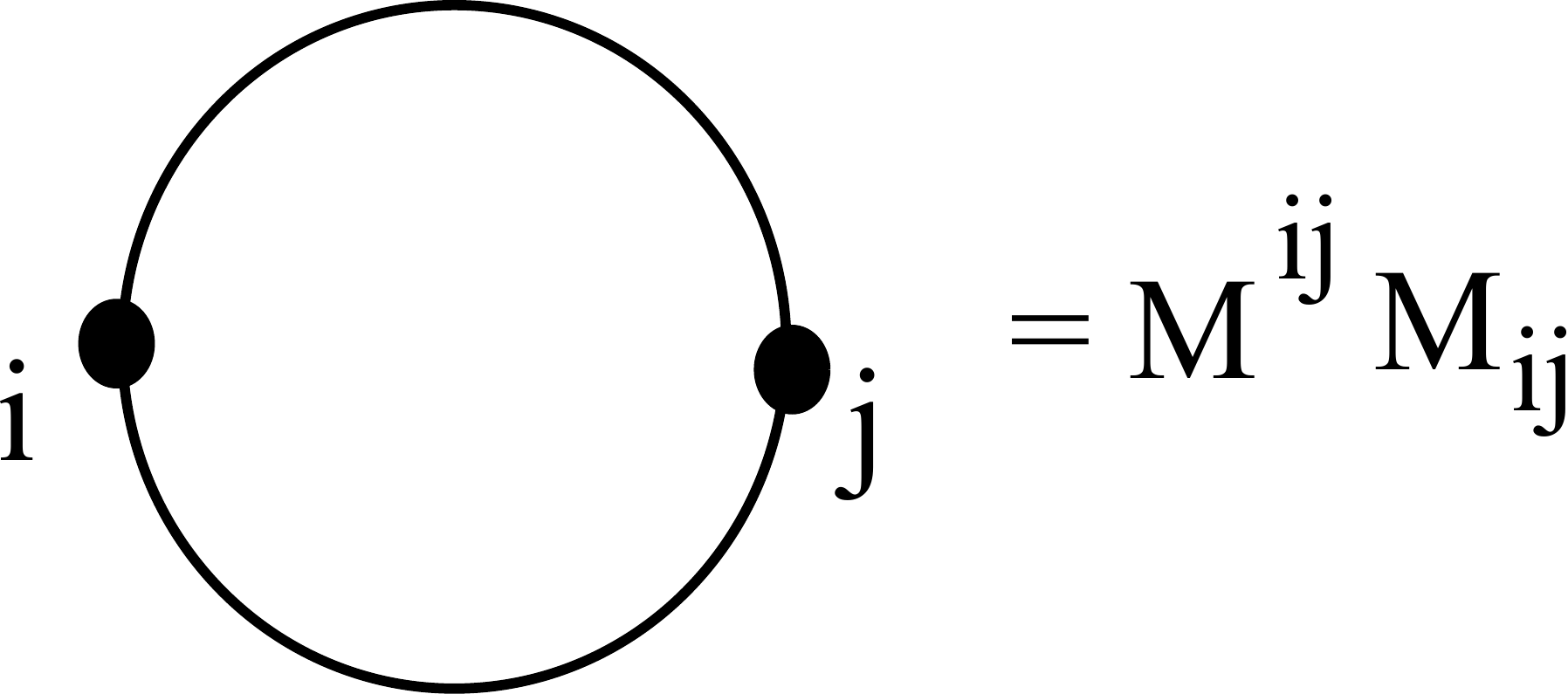}
\caption{The circle value}
\label{fig:circ}
\end{figure}
Open-ended state components have fixed indices $a , b$ at the endpoints and are formed by a number of cups and caps. Let  $\lambda_{b}^a$ denote an open-ended state component. The tangent directions at its endpoints are inherited from $K^{a}_{b}$ and so are both directed upwards . %

 An open-ended state component $\lambda_{b}^a$ contributes to the bracket polynomial with  the partition function  value $<\lambda_{b}^a>$ that is equal to the sum of all products of the entries of the corresponding cup and cap matrices and the sum is taken over all possible labelings of the internal configurations with indices from $\mathcal{I}$. 
\begin{definition}\normalfont
The \textit{bracket partition function} of $K^a_{b}$ is defined as:\\
$$<K_{b}^a> = \sum_{\sigma} <K |\sigma> (-A^2 - A^{-2} )^{n-1} <\lambda_{b}^a>,$$
where $\sigma$ denotes a state, $<K|\sigma>$ denotes the product of crossing weights (that is, the product of $A$'s and $A^{-1}$'s in the expansion of crossings) of $\sigma$, $n$ is the number of circular components and $<\lambda_{b}^a>$ is the partition function assigned to the open-ended state component in $\sigma$. 

 $<K_{ab}>$ and $<K^{ab}>$ are defined in the same way as above, only the notations for the open-ended state components of the Morse knotoid diagrams $K_{ab}$ and $K^{ab}$ differ.  An open-ended state component of a  Morse knotoid diagram is denoted by $\lambda_{ab}$ if its endpoints are directed upwards and downwards, and by $\lambda^{ab}$ if its tangent directions at the endpoints are downwards and upwards, with respect to the bottom to top vertical direction of the plane.
\end{definition}
 \begin{example} \normalfont 
 By using the state sum definition, the reader can easily compute the bracket partition functions $<K_{2}^2>$, $<K_{2}^{1}>$, $<K_{1}^{2}>$ of the Morse knotoid diagram given in Figure \ref{fig:bracket} with the corresponding labeling at its endpoints,  and verify that the induced partition function matrix $[<K_{b}^a>]$ is the following diagonal matrix. 
  $$ \begin{bmatrix}
-A^2 + A^{-2} +1& 0 \\
0 & -A^4 -A^{-2} +A^{-6}
\end{bmatrix}
$$
\end{example}
\begin{proposition}\label{prop:open}
 The bracket partition function matrix $[<\lambda_{b}^a>]$ of an open-ended state component $\lambda_{b}^{a}$, where $a,b \in I$, is one of the following matrices, where $n$ denotes the absolute value of the rotation number of $\lambda_{b}^a$\\
 
$$ M_1= \begin{bmatrix}
(-A^2)^n& 0 \\
0 & (-A^{-2})^n
\end{bmatrix},
\hspace{5mm}
M_2 = \begin{bmatrix}
(-A^{-2})^n& 0 \\
0 & (-A^{2})^n
\end{bmatrix}.
$$

\end{proposition}
\begin{proof}
By Theorem \ref{thm:rota}, $\lambda_{b}^a$ is Morse isotopic to the in-going spiral Morse knotoid diagram whose endpoints are labeled with $a$ and $b$ either in counter-clockwise or clockwise direction. Suppose that $\lambda_{b}^a$ turns inwards in the counter-clock wise direction. Since the rotation number of $\lambda_{b}^a$ is $n$, one encounters $n$ pairs of maxima and minima sequentially while traversing $\lambda_{b}^a$. Then, the formula for the partition function $<\lambda_{b}^a>$ is given by the formula below.


$<\lambda_{b}^a> = \sum_ {\substack{\{i_k\}_{k=1,...,2n-1} \in I\\ a,b fixed}} M^{ai_1} M^{ i_{2} i_{3}}~...~M^{i_{2n-2} i_{2n-1}}M_{bi_{2n-1}}M_{i_{2n-2}i_{2n-3}}~...~M_{i_2 i_1}$\\


$=\sum_{\substack{\{i_k\}_{k=1,...,2n-1} \in I\\ a,b fixed}}  M^{ai_1} M^{ i_{2} i_{3}}~...~M^{i_{2n-2} i_{2n-1}} (M^T)_{i_{2n-1}b} (M^T)_{i_{2n-3} i_{2n-2}}~...~(M^T)_{i_{1}i_{2}}$ \\

$=\sum_{\substack{\{i_k\}_{k=1,...,2n-1} \in I\\ a,b fixed}} (MM^T)^{a}_{i_2} (MM^T)^{i_2}_{i_4}~...~(MM^T)^ {i_{2n-4}}_{i_{2n-2}} (MM^T)^{i_{2n-2}}_b.$

 It is easy to see 
$$[MM^T]= \begin{bmatrix}
-A^2 & 0 \\
0 & -A^{-2}
\end{bmatrix}.
$$
Therefore the matrix determined by  $<\lambda_{b}^a>$ is the following matrix.

$$ \begin{bmatrix}
(-A^2)^n & 0 \\
0 & (-A^{-2})^n
\end{bmatrix}.
$$

When $\lambda_{b}^a$ is Morse isotopic to the in-going spiral Morse knotoid diagram that turns in the clockwise direction, it can be shown that  $<\lambda_{b}^a> $ is assigned to the following matrix that is the $n^{th}$ power of $[M^TM]$. The verification of this is left to the reader.

 
$$\begin{bmatrix}
(-A^ {-2})^n & 0 \\
0 & (-A^{2})^n
\end{bmatrix}.
$$

\end{proof}
\begin{corollary}
The bracket partition function of a Morse knotoid diagram $K$ with upwards directed endpoints, and labels on its endpoints induces a $2 \times 2$ diagonal matrix with entries $<K_{b}^a>$.
\end{corollary}
\begin{proof}
The bracket partition function $<K_{b}^a>$ is given by $<\lambda_{b}^a>$ multiplied with a polynomial coefficient in $A$ and $A^{-1}$. The statement follows from this fact and Proposition \ref{prop:open}.
\end{proof}

\begin{proposition}\label{prop:open2}

The bracket partition function matrices of open-ended state components $\lambda_{ab}$ and $\lambda^{ab}$ that are determined by $<\lambda_{ab}>$ and $<\lambda^{ab}>$, are one of the following matrices,
$$M_1= \begin{bmatrix}
0& (-A^2)^{\frac{n}{2}} \\
 (-A^{-2})^ {\frac{n}{2}} & 0
\end{bmatrix}, 
\hspace{5mm}
M_2 = \begin{bmatrix}
0& (-A^{-2})^{\frac{n}{2}} \\
(-A^{2})^{\frac{n}{2}} & 0 
\end{bmatrix},
$$
where $\frac{n}{2}$ is the absolute rotation number of $\lambda_{ab}$ and of $\lambda^{ab}$ for some $n>0$.
\end{proposition}

\begin{proof}
With the argument in the proof of Proposition \ref{prop:open}, assume first $\lambda_{ab}$ is Morse isotopic to the in-going spiral Morse knotoid diagram that turns in the counter clockwise direction. We first observe that to have a connected component with the endpoints pointing up and down, respectively, the total number of cups and caps forming $\lambda_{ab}$ is necessarily an odd number.

The bracket partition function of $\lambda_{ab}$ is given as follows.\\
 
$<\lambda_{ab}> = \sum_ {\substack{\{i_k\}_{k=1,...,2n-1} \in I\\ a,b fixed}}  M^{i_{1} i_{2}} ~...~M^{i_{n-2}i_{n-1}}M_{i_{n-2}b} M_{i_{n-4}i_{n-1}}~...~M_{i_{1} i_{4}}M_{{a}i_{2}}$\\

$=\sum_ {\substack{\{i_k\}_{k=1,...,2n-1} \in I\\ a,b fixed}}  (M^T)^{i_{2} i_{1}}~ ... ~(M^T)^{i_{n-1}i_{n-2}} M_{i_{n-2} b} M_{i_{n-4}i_{n-1}}~...~ M_{i_{1} i_{4}}M_{{a}i_{2}}$
\\

This sum is the $ab^{th}$ entry of the matrix  $[[(M^TM)^{\frac{n-1}{2}}M^T]^T] $.
Direct computation shows that  $<\lambda_{ab}>$ yields the matrix $M_1$:
$$M_1= \begin{bmatrix}
0& (-A^2)^{\frac{n}{2}} \\
 (-A^{-2})^ {\frac{n}{2}} & 0
\end{bmatrix}.
$$
\\
Suppose now $\lambda_{ab}$ is Morse isotopic to the spiral Morse knotoid diagram that turns in the clockwise direction. It is left to the reader to verify that the matrix $M_2$ that is assigned to $<\lambda_{ab}>$, is the transpose of the matrix $M_1$. That is, 

$$M_2 = \begin{bmatrix}
0& (-A^{-2})^{\frac{n}{2}} \\
 (-A^{2})^ {\frac{n}{2}} & 0
\end{bmatrix}.
$$

If the endpoints of the state component is directed down and up, respectively, it follows from a similar argument above and  $[M_{ab}] = [(M^T)^{ab}]$, where [$M_{ab}]$ and $[(M^{ab}]$ are matrices for the cap and the cup respectively, that  $<\lambda^{ab}>$ yields the transpose of the matrix $M_1$ or $M_2$ depending on the turning of the spiral Morse knotoid diagram that $\lambda^{ab}$ is Morse isotopic. 
\end{proof}
\begin{corollary}
The bracket partition function of a Morse knotoid diagram $K$ with up-down or down-up tangent directions at its endpoints induces a $2\times 2$ non-diagonal matrix with entries $<K_{ab}>$ or $<K^{ab}>$, respectively. 
\end{corollary}
\begin{proof}
It follows from the fact that $<K_{ab}>$ and $<K^{ab}>$) are determined by $<\lambda_{ab}>$,  $<\lambda^{ab}>$, respectively, multiplied with a polynomial coefficient in $A$ and $A^{-1}$ and by Proposition \ref{prop:open2}.
\end{proof}


\subsection{The Binary Bracket Polynomial}\label{sec:binarybracket}
The binary bracket polynomial of virtual links was defined by the second author \cite{Kauf} as a modification of the bracket polynomial. In this section, we study the binary bracket polynomial for Morse multi-knotoids and construct it via a solution to the Yang-Baxter equation. 

Let $K$ be a Morse multi- knotoid diagram. The binary bracket polynomial is based on a coloring of the bracket state components with elements from the set $\{0, 1\}$.  The coloring rule is  as follows. The colors appearing at a smoothing site, that is, on the two pieces of strands of $K$ obtained by smoothing a crossing, must be different. In Figure \ref{fig:binary}, we illustrate possible coloring configurations at the smoothing sites where different colors meet  A dark line at a smoothing site indicates that the two local components at the smoothing site must colored differently.  We call a bracket state of $K$ whose components can be all colored in this way a \textit{properly colored state}. The binary bracket of $K$ is evaluated as the total contribution of all properly colored states of $K$. It can happen that there is no such coloring possible for a state. Such states will have a zero evaluation for the binary bracket polynomial.

\begin{figure}[H]
\centering
\includegraphics[width=.7\textwidth]{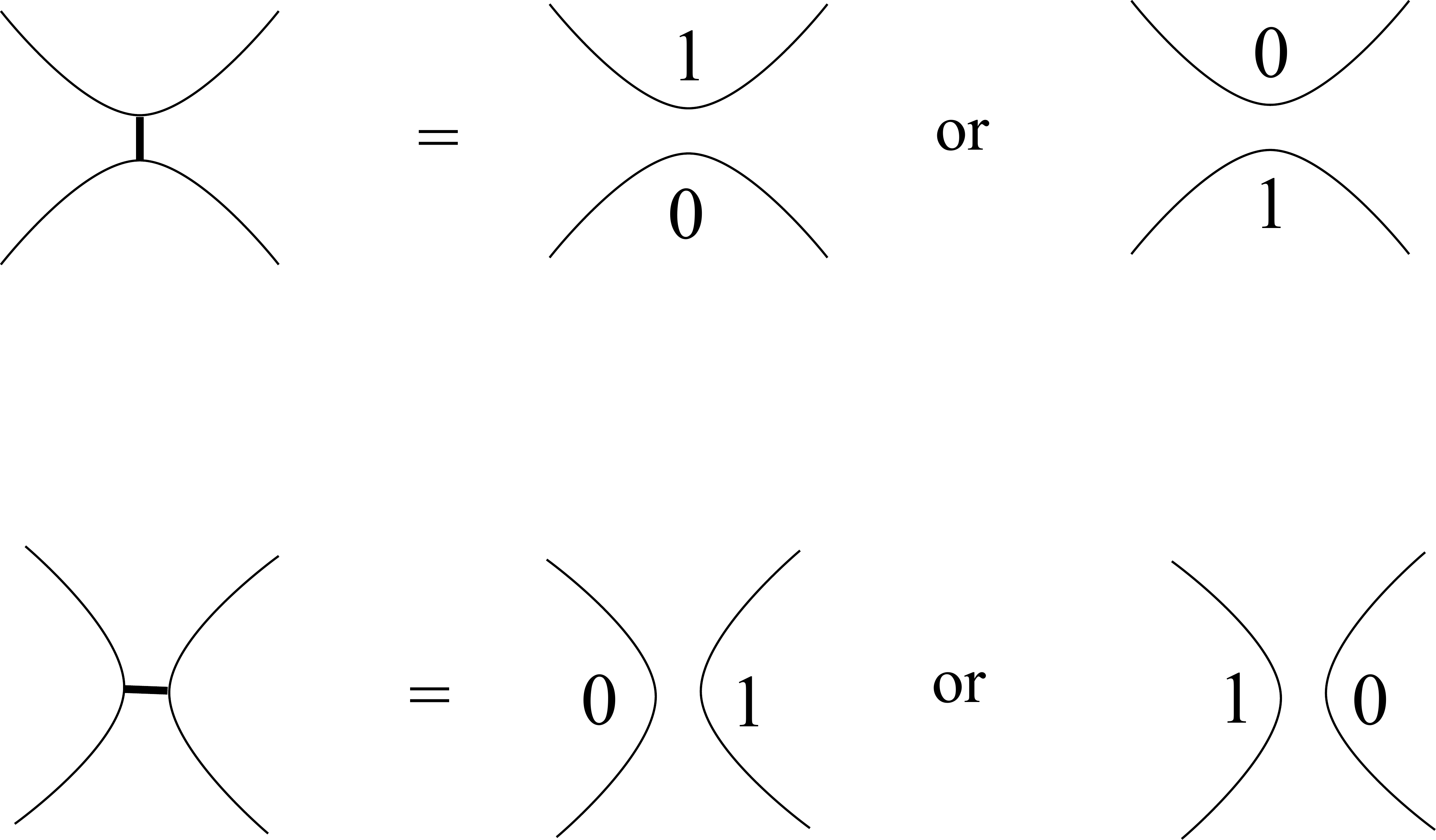}
\caption{The coloring at smoothing sites}
\label{fig:binary}
\end{figure}
The binary bracket polynomial of a Morse multi-knotoid diagram can be computed recursively by the relations given in Figure \ref{fig:binrel}. It is clear that the trivial knot diagram has two properly colored states, colored with $0$ and $1$. Thus, the trivial knot diagram gets the value $2$ with the binary bracket polynomial. Likewise, the trivial knotoid diagram has two properly colored states, but we assume only one of the the colored states for the trivial knotoid diagram. Let this state be the state colored with $0$. With this assumption, the trivial knotoid diagram is assigned to $1$ with the binary bracket polynomial, see the third relation. The second relation shows that each disjoint unknot component multiplies the binary bracket polynomial of a Morse multi-knotoid diagram $K$ by two. Notice that if the unknot component is connected to $K$ with a dark band  then the binary bracket polynomial of $K$ remains the same. In the first relation, we see that the coefficient contributions coming from smoothing sites of a crossing are the same with the usual bracket polynomial case.
 \begin{figure}[H]
\centering
\includegraphics[width=.7\textwidth]{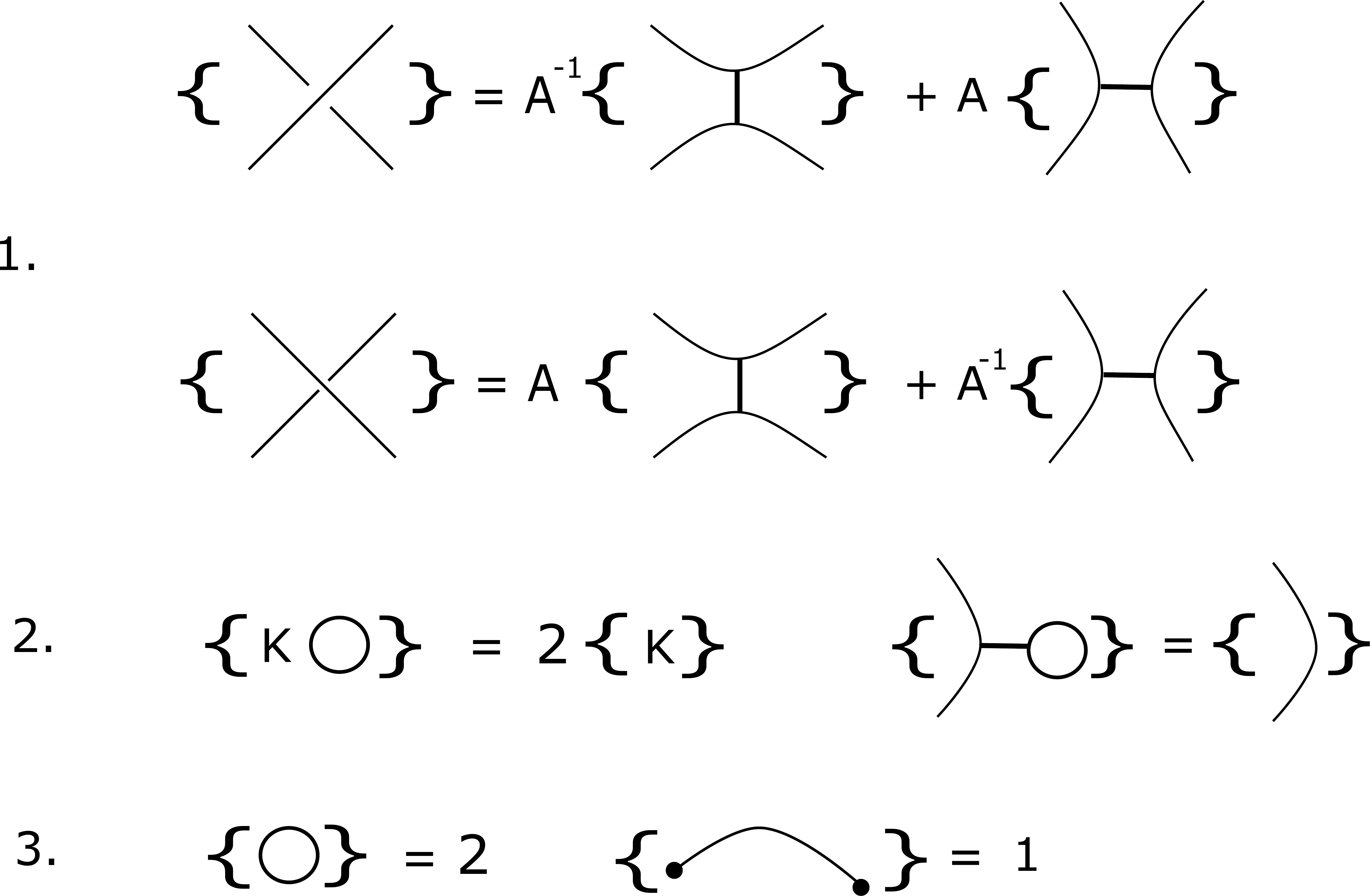}
\caption{The binary bracket relations}
\label{fig:binrel}
\end{figure}
The closed summation formula for the binary bracket polynomial of a Morse multi-knotoid diagram is as follows.
\begin{definition} \normalfont
Let $K$ be a Morse multi-knotoid diagram. The \textit{binary bracket polynomial} of $K$ is defined as,
$$\{K\}(A) = \sum_{\substack{S \in \text{Properly colored states}}}  < K | S > ,$$ 
where $< K | S>$ is the product of the contributions of the smoothing sites in a properly colored state $S$.
\end{definition}

Figure \ref{fig:com} shows the whole set of bracket states of the given Morse knotoid diagram $K$. The first three bracket states of $K$ do not contribute to the binary bracket polynomial since they do not admit a proper coloring. As the figure suggests, different colors at the smoothing sites of these states would yield an incompatible coloring on the open-ended segment components. The last state admits only one coloring of its components such that the open-ended state component is colored with $0$ and the closed state component with $1$, and these components are connected with a band. By the second relation, the contribution of this state is as the value of a single trivial knotoid diagram multiplied with the smoothing site coefficient that is $A^{-2}$. Therefore $\{K\} = A^{-2}$.  
\begin{figure}[H]
\centering
\includegraphics[width=1\textwidth]{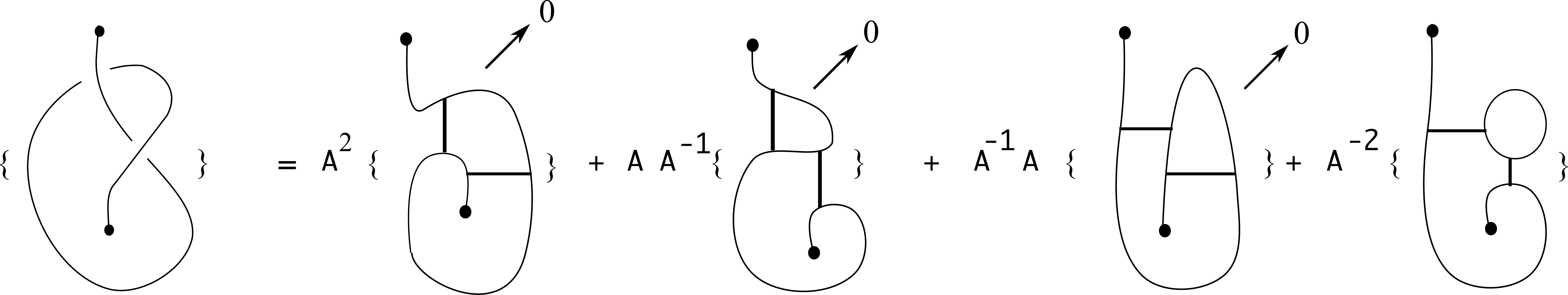}
\caption{Computing the binary bracket of a knotoid}
\label{fig:com}
\end{figure}

\begin{proposition}
The binary bracket polynomial is a Morse isotopy invariant.
\end{proposition}
\begin{proof}
The verification that the binary bracket polynomial remains invariant under the Morse isotopy moves follows the same as in the case of virtual knots and links. The reader is referred to \cite{Kauf} for details of the verification.
\end{proof}
\begin{proposition}\label{prop:invariant}
The binary bracket polynomial becomes invariant under the knotoid isotopy when it is normalized with the term  $A^{-w(K)}$, where $w(K)$ denotes the writhe of a planar multi-knotoid diagram $K$.
\end{proposition}
\begin{proof}
As shown in Figure \ref{fig:change}, a Reidemeister I move multiplies the binary bracket polynomial with $A$ or $A^{-1}$ depending on the type of the curl the move adds. Then, the invariance can be provided by multiplying the binary bracket polynomial with the term  $A^{-w(K)}$.
\begin{figure}[H]
\centering
\includegraphics[width=.7\textwidth]{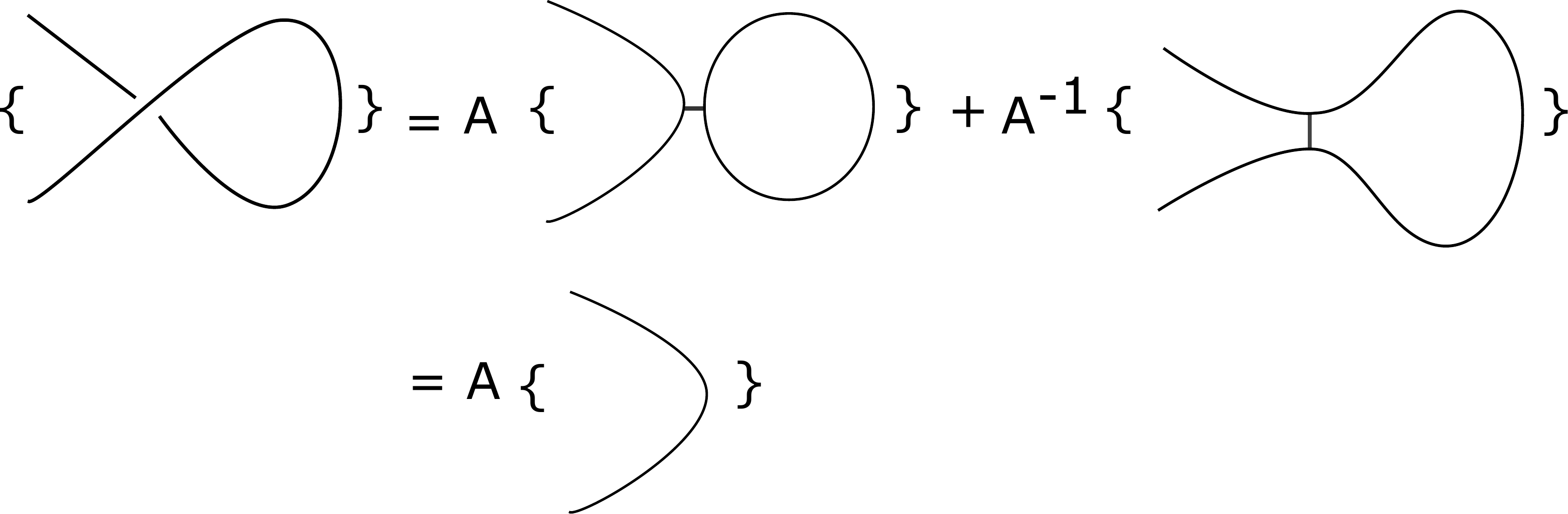}
\caption{The change under a Reidemeister I move}
\label{fig:change}
\end{figure}
\end{proof}
\begin{corollary}\normalfont
Since any planar knotoid admits a unique standard Morse diagram, the normalized binary bracket polynomial is an invariant of knotoids in $\mathbb{R}^2$.
\end{corollary}
\subsubsection{A closer look}
It is understood that the coloring condition restricts the collection of bracket states of a knotoid diagram $K$  to a small collection. By a closer look at smoothing sites of $K$, we observe that the coloring condition implies an alternating coloring on $K$.  Let a strand that connects two vertices of the underlying flat diagram of $K$ be named as an \textit{edge}. The coloring condition at a smoothing site implies that any two edges at the corresponding vertex that are not adjacent with respect to a cyclic order receive different colors. See Figure \ref{fig:binarycoloring} for an illustration. This requirement is satisfied when the edges of the underlying flat diagram of $K$ is colored with $0$ or $1$ in such a way that the colors on the edges alternate as we travel around the diagram with the orientation from the leg to the head of $K$. There exists exactly two such colorings for the flat diagram of $K$ depending on the labeling of the first edge that is adjacent to the tail of $K$ either with $0$ or $1$. Since each vertex except the endpoints is visited twice, any flat knotoid diagram can be colored in this way. When we assume to color  the initial edge incident to the leg with $0$, $K$ has only one properly colored state, since there is exactly one way to smooth each crossing of $K$ equipped with such coloring.

\begin{figure}[H]
\centering
\includegraphics[width=.5\textwidth]{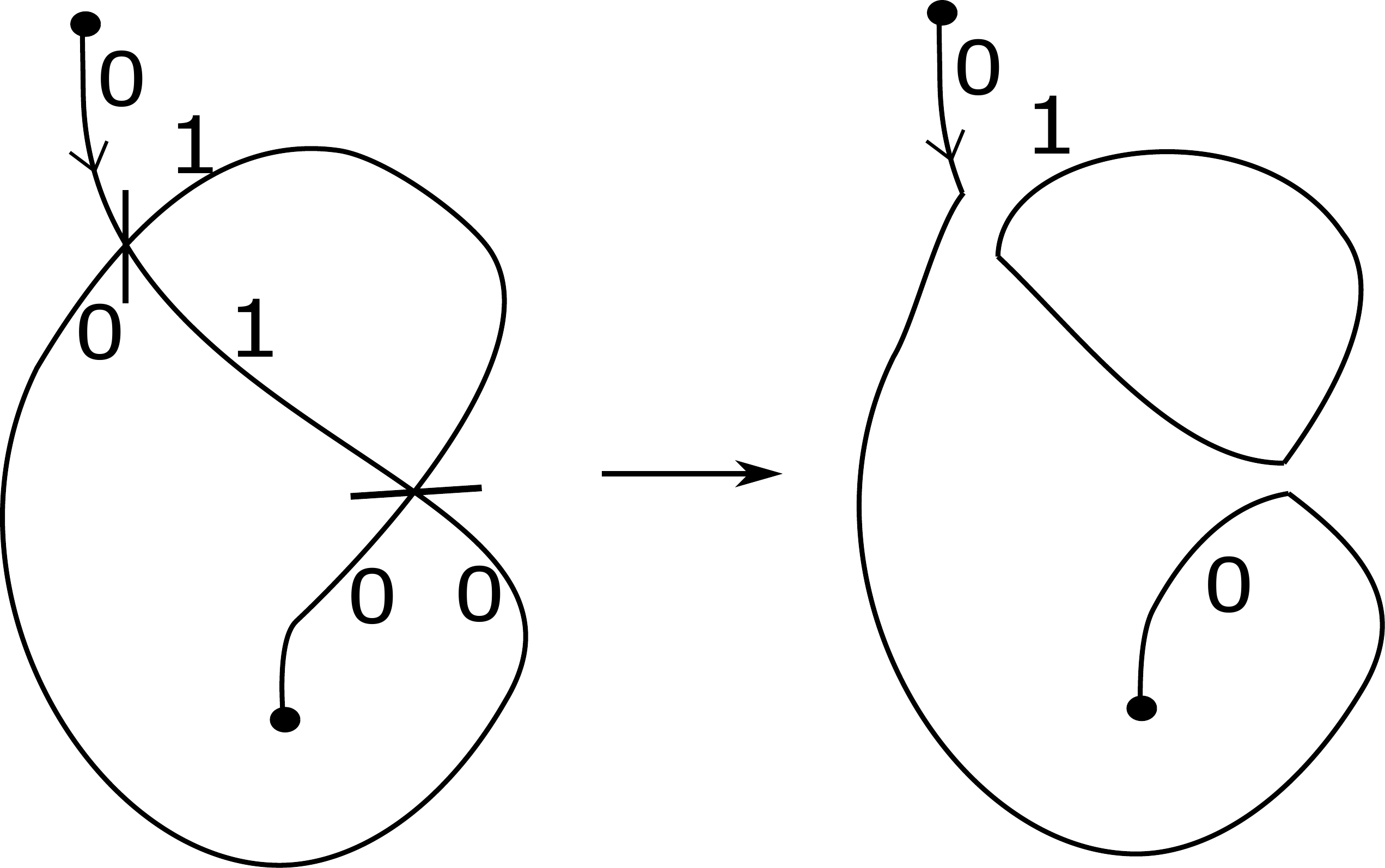}
\caption{}
\label{fig:binarycoloring}
\end{figure}

\begin{proposition}\label{prop:binaryprop}
If $\kappa$ is a Morse knotoid that admits a diagram whose endpoints lie in the same local region of $\mathbb{R}^2$, then $\{\kappa\} =  A ^{w(\kappa)}$, where $w(K)$ is the writhe of $\kappa$.
\end{proposition}
\begin{proof}
Let $K$ denote a Morse knotoid diagram of $\kappa$ satisfying the given condition on the endpoints. 

From the above discussion, we first note that $K$ admits a unique binary coloring with the assumption on its initial edge incident to its tail colored with $0$, and as a result, there is only one properly colored state of $K$.  
 
Since $K$ has its endpoints in the same planar region, any crossing of $K$ has even parity, so is an even crossing \cite{GK1}. This means that there is an even number labels between the two labels representing a crossing of $K$ in the Gauss code of $K$. This is equivalent to say that there is an even number of intersections between any loop based at a crossing of  of $K$ (consider the crossing as flat) and the rest of the diagram. (Note that we do not count the intersection at the base point of the loop.)

 Figure \ref{fig:crossing parity} illustrates an even crossing.
\begin{figure}[H]
\centering
\includegraphics[width=.25\textwidth]{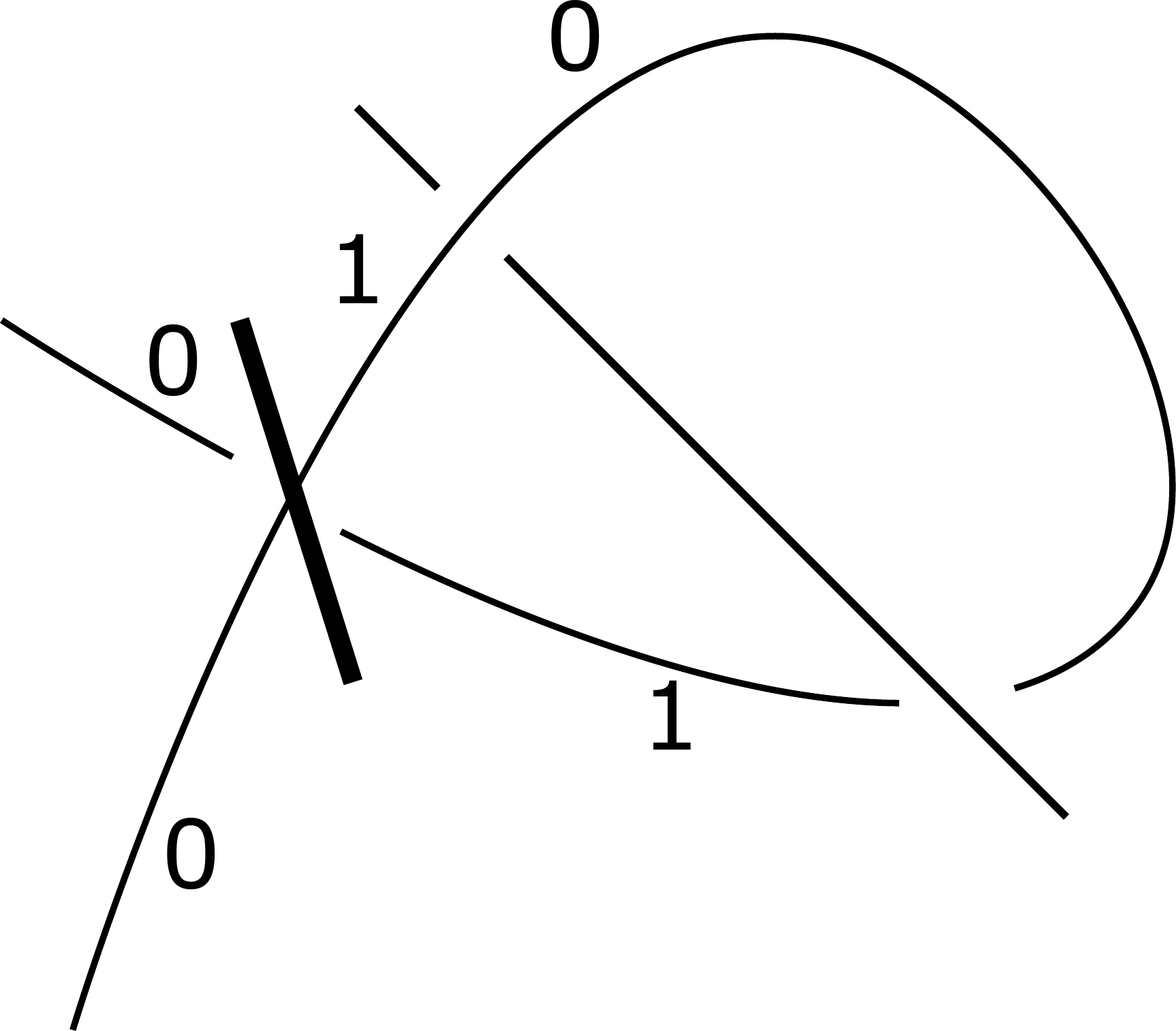}
\caption{An even crossing}
\label{fig:crossing parity}
\end{figure}
Assume now that $K$ is given an orientation from its tail to its head. Let $c$ denote a crossing of $K$ and $2n$, $n \in \mathbb{N}$ be the number of intersections of the loop at (flat) $c$ with the rest of the diagram. Each intersection permutes the colors appearing on the loop at $c$ so that its initial and the last edge incident to the flat $c$ receive the same color as follows. There are in total $2n+2$ edges on the loop, except from the initial edge going inwards to the flat crossing $c$. Let $\lambda \in \{0,1\}$ be the color on the initial edge. Then the ${k}^{th}$ edge is colored with $\lambda+k$.  Therefore, the last edge incident to the crossing that is going outwards from the crossing, recieves the color $\lambda + 2n +2 \equiv \lambda \pmod 2$.

To obtain the unique properly colored state of $K$ all crossings of $K$ are required to be smoothed so that the initial and the last edges incident to crossings  remain on the same side of the smoothing sites. This corresponds to smoothing crossings of $K$ agreeably with the orientation on $K$ which results in each crossing $c$ of $K$ contributing to the binary bracket polynomial with the value $A^{sign(c)}$, see Figure \ref{fig:signn}. Therefore, the total contribution from the oriented smoothing the crossings of $K$ is $A^{w(K)}$ where $w(K)$ is the writhe of $K$. The writhe is clearly a Morse isotopy invariant so $w(K)=w(\kappa)$.  The statement follows.
\begin{figure}[H]
\centering
\includegraphics[width=.4\textwidth]{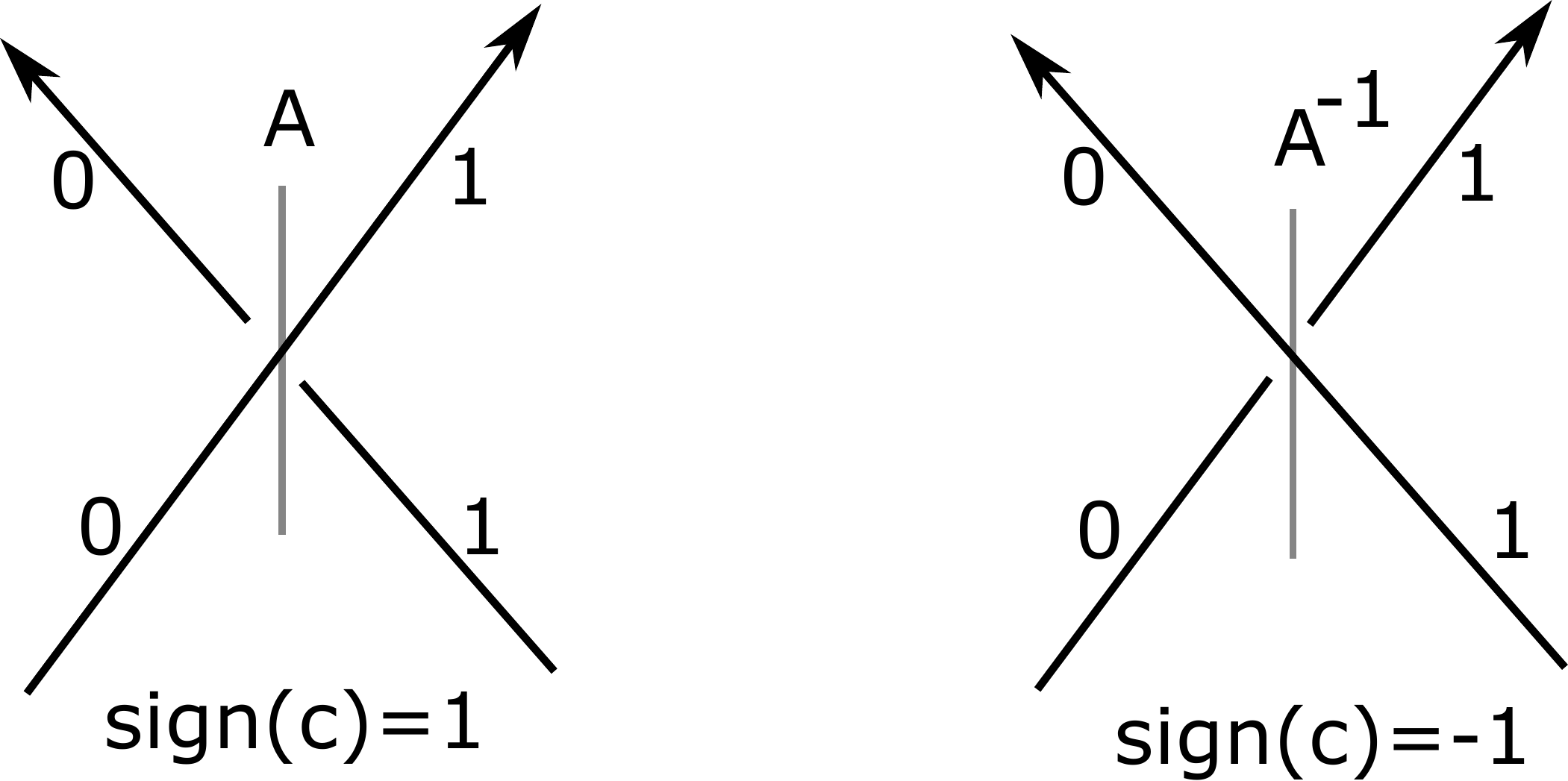}
\caption{Oriented smoothing of crossings}
\label{fig:signn}
\end{figure}
\end{proof}

By Proposition \ref{prop:binaryprop}, we have the following corollary. 

\begin{corollary}
The normalized binary bracket polynomial is trivial for knot-type knotoids.
\end{corollary}

In a knot-type knotoid diagram any crossing is even with respect to to the Gaussian parity but in a proper knotoid diagram, there is at least one odd crossing (a crossing admitting odd parity with respect to the Gaussian paritty) cite{GK1}. For a proper knotoid diagram $K$, we can define \textit{odd writhe} as the sum of the signs of the odd crossings in $K$. The odd writhe is a knotoid invariant \cite{GK1}. We can also define the odd writhe for Morse knotoids directly.

\begin{proposition}
If $\kappa$ is a Morse knotoid with any of its representative diagrams having its endpoints in different local regions determined by the diagram in $\mathbb{R}^2$, then $\{ \kappa \} = A^{-2J(\kappa)+ w(\kappa)}$, where $J(\kappa)$ is the odd writhe of $\kappa$ and $w(\kappa)$ is the writhe of $\kappa$.
\end{proposition}
\begin{proof}
Let $K$ be a representative diagrams of $\kappa$. By the former discussions we know that $K$ admits a unique properly colored state, and since it has its endpoints in different planar regions,  $K$ has at least one odd crossing. That is, there is at least one crossing $c$ of $K$ such that there is an odd number of labels between the two incidences of the label representing $c$ in the Gauss code of $K$. This requires that the loop at $c$ encloses one of the endpoints of $K$. See Figure \ref{fig:odd}.
\begin{figure}
\centering
\includegraphics[width=.25\textwidth]{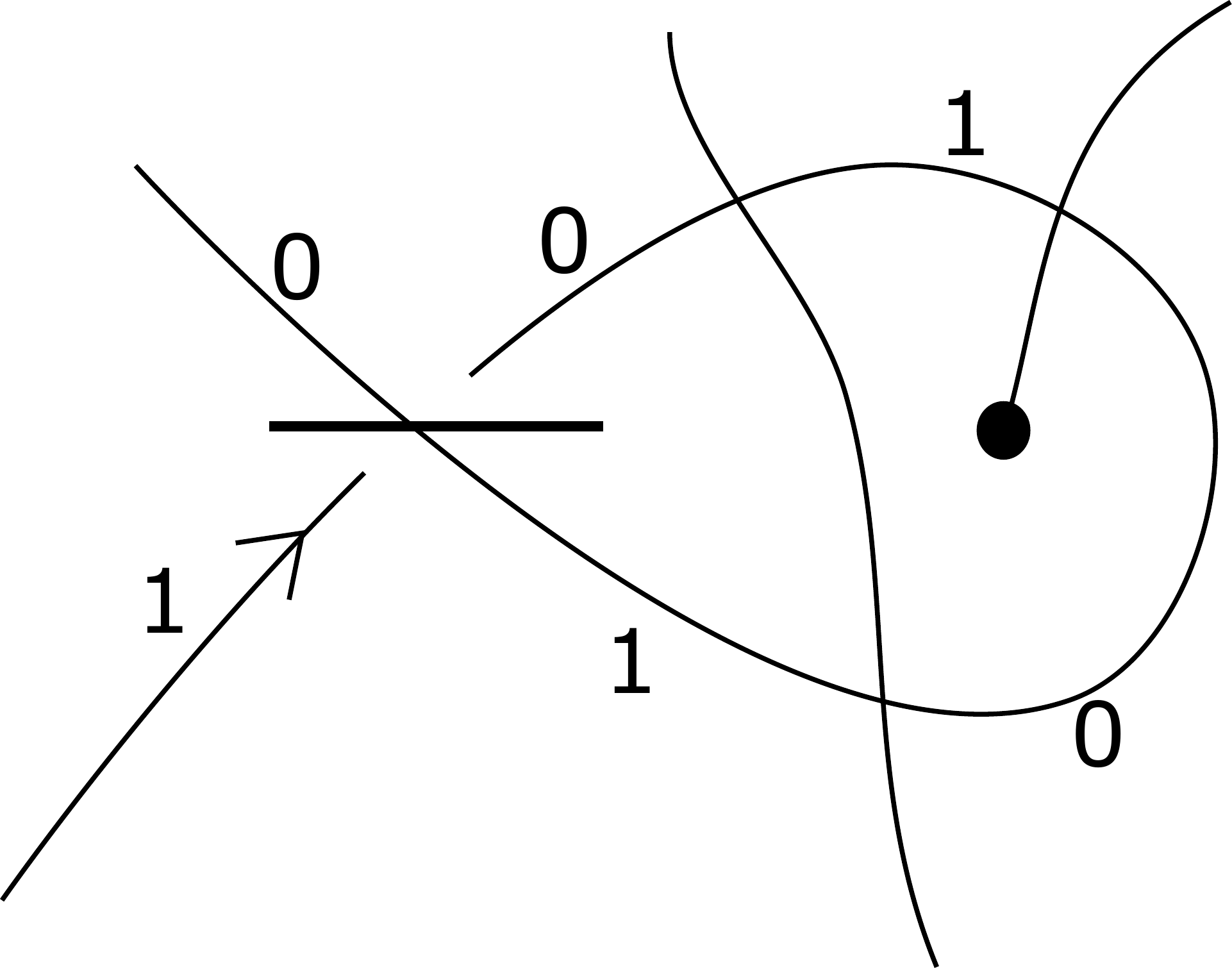}
\caption{An odd crossing}
\label{fig:odd}
\end{figure}
Assume $K$ is oriented from its tail to its head. Let the loop formed by an (flat) odd crossing has $2n+1$, $n \in \mathbb{N}$ intersections with the rest of the diagram.  Let $\lambda \in \{0,1\}$ be the color of the initial edge incident to the flat crossing that goes inwards into the crossing. Then the last edge incident to the crossing that goes outwards, is colored with $\lambda+ 2n + 3 \equiv \lambda + 1 \pmod 2$. That is, the initial edge incident to the vertex of the crossing receives a different color than the last edge.

The only way to smooth an odd crossing according to the coloring condition is to smooth it so that the edges that lie on same one side of the crossing connect with the edges that lie on the other side of the crossing. This corresponds to the $A-type$ smoothing for a negative crossing, as depicted in Figure \ref{fig:odd}. Thus, the contribution of smoothing of a negative crossing is $A$ that is  $A^{-sign(c)}$. If the odd crossing is a positive crossing, then the smoothing would correspond to the $B$- type smoothing and the contribution to the binary bracket polynomial would be $A^{-sign(c)} = A^{-1}$. 

Let $E(K)$ denote the total sum of the signs of even crossings of $K$. It is clear that
$$w(K) = J(K) + E(K),$$  
where $w(K)$ is the writhe of $K$.

From Proposition \ref{prop:binaryprop} and the observation above, it follows that the total contribution made from smoothing all crossings of $K$ in order to obtain the properly colored state of $K$ is equal to $A^{-J(K) + E(K)}$. Substitute $E(K) = w(K) - J(K)$, we find that the contribution of the properly colored state is  $A^{-2J(K)+ w(K)}$.
\end{proof}
\begin{corollary}
The normalized binary bracket polynomial of a proper knotoid $\kappa$ is $A^{-2J(\kappa)}$, where $J(\kappa)$ is the odd writhe of $\kappa$.
\end{corollary}

If $K$ is a multi-knotoid diagram then $K$ does not necessarily admit such binary coloring and so may have no properly colored state at all. See the flat multi-knotoid diagram given on the left hand side of  Figure \ref{fig:binarycoloringii} where a coloring on the knot component results in a contradiction ($0 =1$). We will examine uncolorable multi-knotoid diagrams below in more detail. If $K$ is colorable multi-knotoid diagram then each of its knot components admits two colorings by swapping $0$'s to $1$'s  or vice versa on its edges. This implies that $K$ has $2^{m}$ properly colored states, where $m$ is the number of knot components of $K$. It is not hard to see that a multi-knotoid diagram whose flat diagram is seen on the right hand side of Figure \ref{fig:binarycoloringii} has $2$ properly colored state obtained by swapping the given colors on the unknot component. 

\begin{figure}[H]
\centering
\includegraphics[width=.7\textwidth]{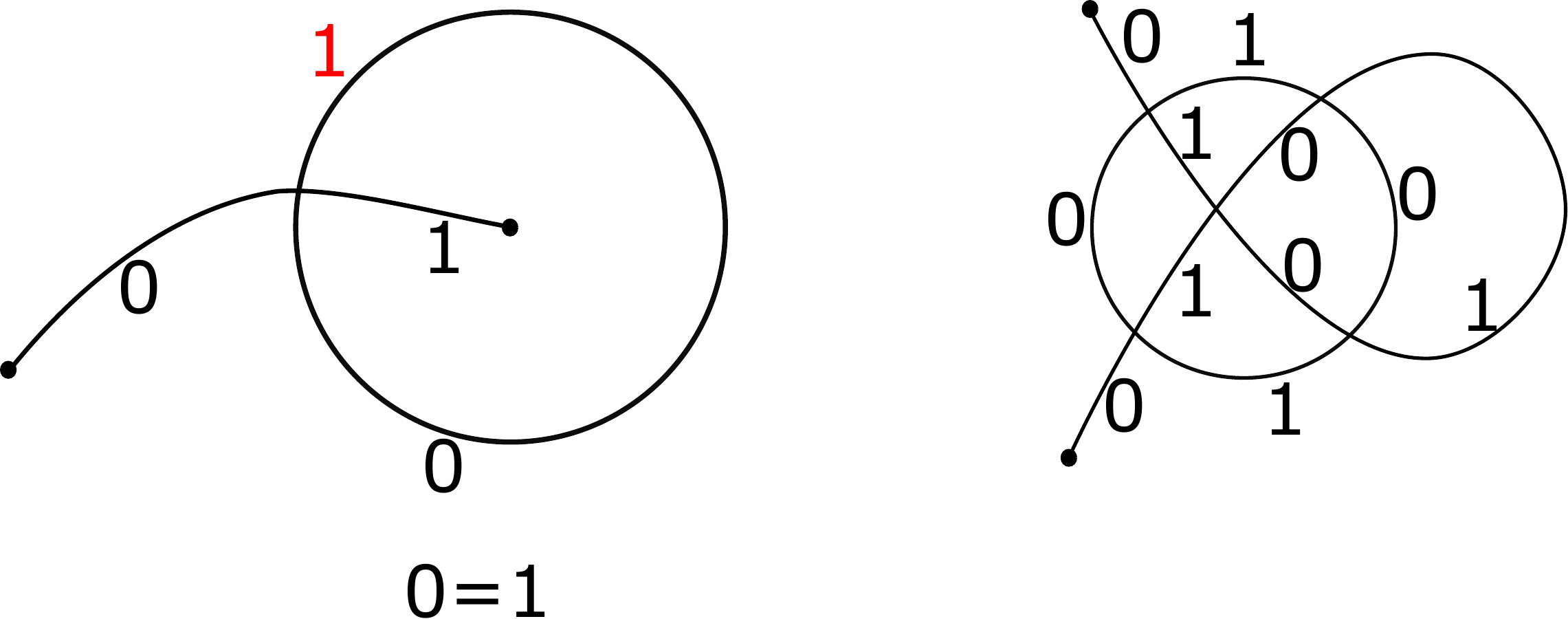}
\caption{A uncolorable and a colorable flat multi-knotoid diagram}
\label{fig:binarycoloringii}
\end{figure}

\begin{proposition}
Let $K$ be a multi-knotoid diagram. If the number of crossings shared by the knotoid component of $K$ and one of the knot components is odd, then $K$ cannot be colored properly.
\end{proposition}
\begin{proof}
First observe that any two knot components of $K$ share an even number of crossings since any strand entering `inside'  a region enclosed by a knot diagram must go out of the region, as a result of the Jordan curve theorem. Thus, an odd number of crossing can be shared only between the knotoid component and a knot component of $K$. Let $\mu$ be a knot component of $K$ sharing $2n+1$, $n \in \mathbb{N}$ crossings with the knotoid component of $K$. Let $\lambda \in \{0,1\}$ be the color on an edge of $\mu$.  Each crossing alternates the color $\lambda$ by $1$ as we traverse the component and so the edge that is colored initially with $\lambda$ is required to receive the color $\lambda + 2n +1 \equiv \lambda+1 \pmod 2$ when we are back to the same edge. This results in a not well-defined coloring of the knot component because one of the edges on it receives two different colors. We exemplify this in Figure \ref{fig:binarycoloringii}. 
\end{proof}

In Figure \ref{fig:multibin}, we illustrate the two proper colorings of a multi-knotoid diagram whose underlying flat diagram is given in \ref{fig:binarycoloringii}. We also indicate the smoothing types to   obtain its properly colored states. The reader can verify easily that the binary bracket polynomial of the multi-knotoid diagram is $A^{3} + A^{-5}$. and its writhe is $-5$. Therefore,  the normalized binary bracket polynomial is $A^8 +1$, and this shows that the multi-knotoid is not isotopic to the multi-knotoid diagram with two split trivial components with the normalized binary bracket polynomial equal to $2$. 

In Figure \ref{fig:white}, we also illustrate the two proper colorings of a multi-knotoid diagram. The reader can verify that the linking number of the multi-knotoid, that is the half of the total sum of the signs of the shared crossings, equals $0$. But we find that its normalized binary bracket polynomial is $A (A^{3} + A^{-5})= A^4 + A^{-4}$ which shows that the nontriviality of this multi-knotoid.  
\begin{figure}[H]
\centering
\includegraphics[width=.75\textwidth]{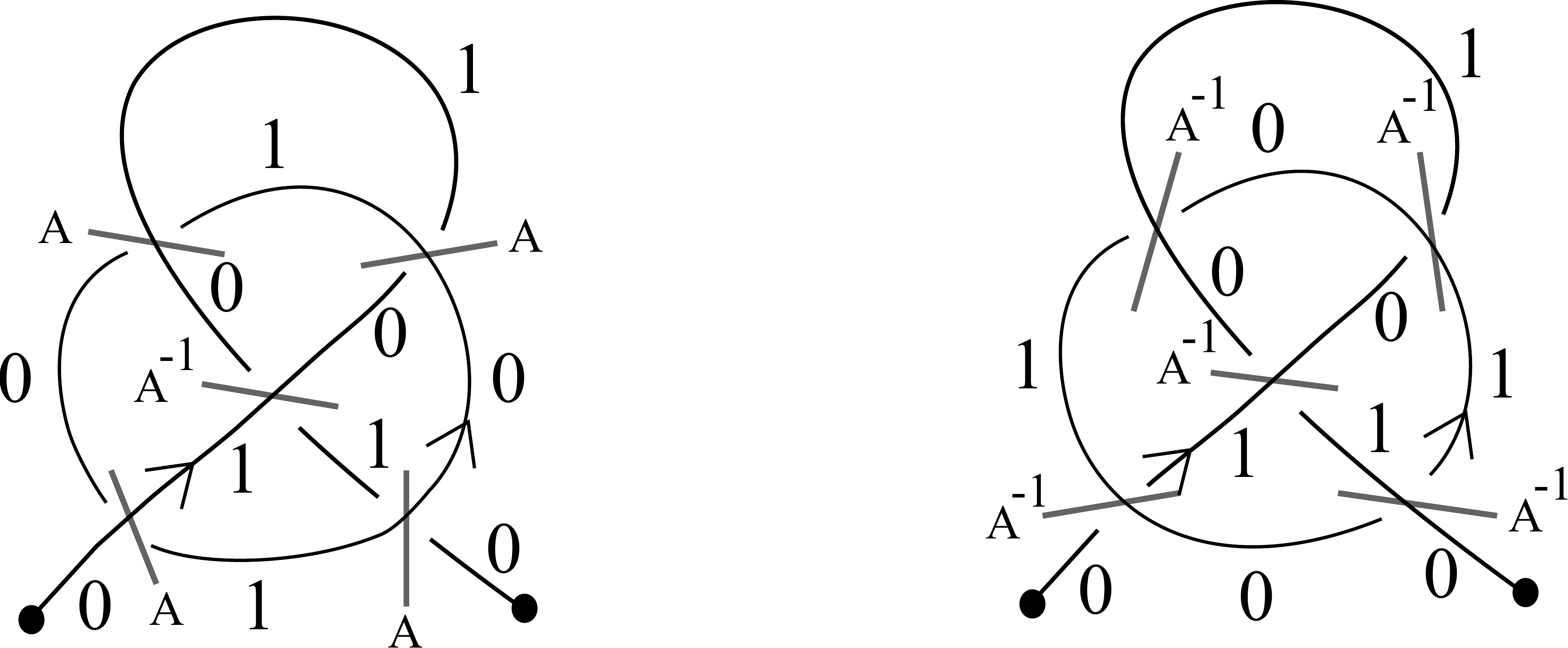}
\caption{}
\label{fig:multibin}
\end{figure}

\begin{figure}[H]
\centering
\includegraphics[width=.75\textwidth]{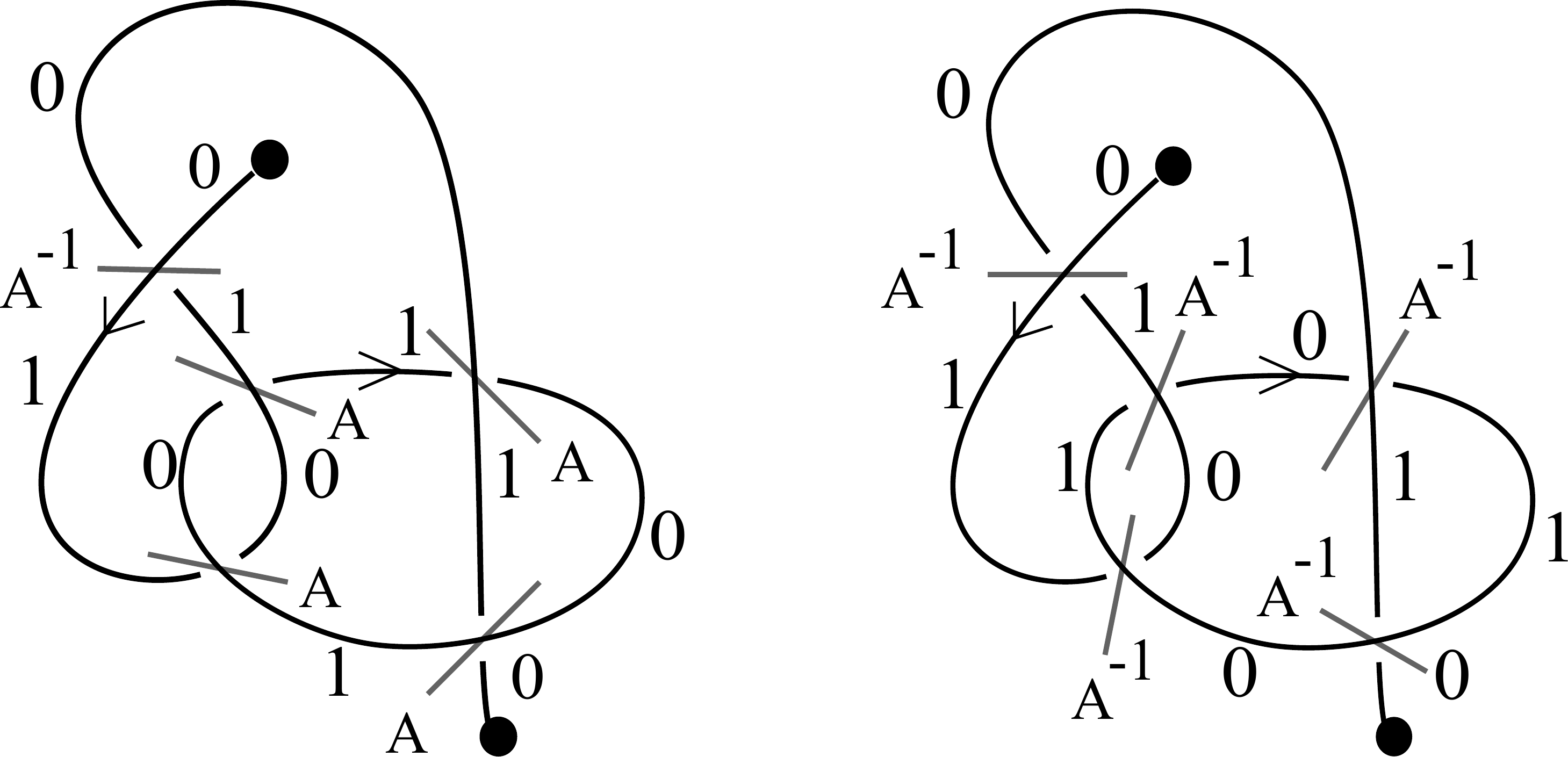}
\caption{}
\label{fig:white}
\end{figure}


\subsubsection{A quantum model for the binary bracket polynomial}


A quantum model for the binary bracket of virtual knots and links is given in \cite{Kauf}. This model can be applied as a quantum model for the binary bracket polynomial for Morse knotoids and multi-knotoids by assigning specific matrices to cups, caps and crossings of diagrams.

We assign the $2\times 2$ identity matrix to cups and caps, and the following $R, R^{-1}$ matrices to the right-handed and left-handed crossings of a knotoid diagram. 
\vspace{1.5cm}
\\
\begin{center} $R = \begin{bmatrix}
0 & 0 & 0 & A^{-1}\\
0 & A  & 0 &0 \\
0 & 0& A& 0\\
A^{-1}  & 0 & 0 & 0
\end{bmatrix},
 \hspace{1cm}
R^{-1} = \begin{bmatrix}
0 & 0 & 0 & A\\
0 & A^{-1}  & 0 &0 \\
0 & 0& A^{-1} & 0\\
A  & 0 & 0 & 0
\end{bmatrix}.
$
\end{center}
\vspace{1cm}

Notice that the $R$ matrix given can be derived from the state expansion of the binary bracket polynomial at the right-handed crossing, see Figure \ref{fig:R} that illustrates the entries $R_{01}^{01}$ and $R_{11}^{00}$. The $R$ matrix is unitary when $A$ is on the unit circle in the complex plane  \cite{Dye}, and, in principle, can be used as a quantum gate for the design of a topological quantum computer. This indicates the possibility for applying quantum invariants of knotoids in quantum computing, a possibility that we shall pursue in subsequent work.
\begin{figure}[H]
\centering
\includegraphics[width=.75\textwidth]{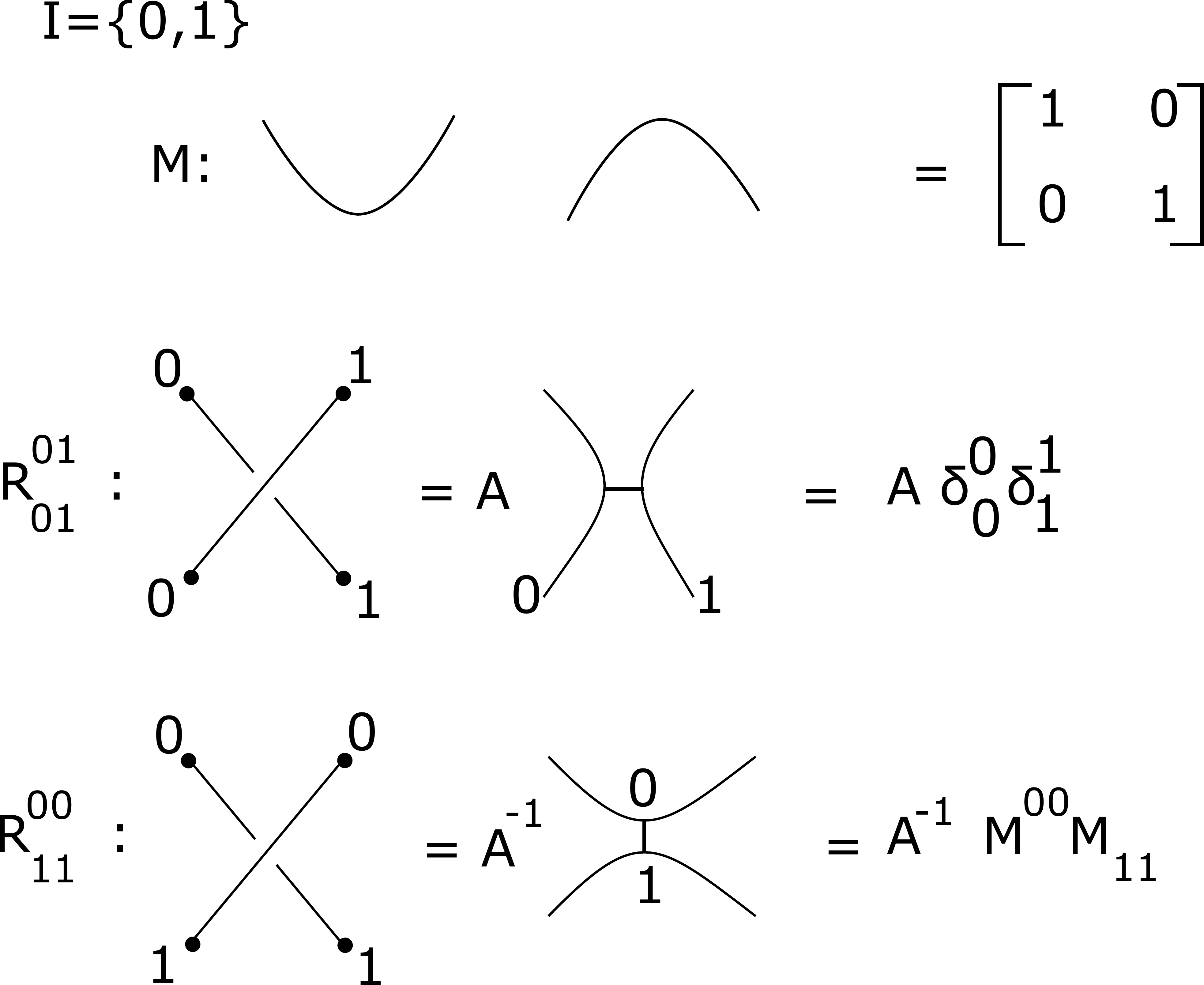}
\caption{The $R$-matrix entries}
\label{fig:R}
\end{figure}

With these matrix assignments, the binary bracket polynomial can be given as a partition function where a single unknot component gets the value $\delta_{00} + \delta_{11} = 2$ and the trivial knotoid diagram gets the value $\delta_{0}^{0} = 1$. 

\begin{figure}[H]
\centering
\includegraphics[width=.5\textwidth]{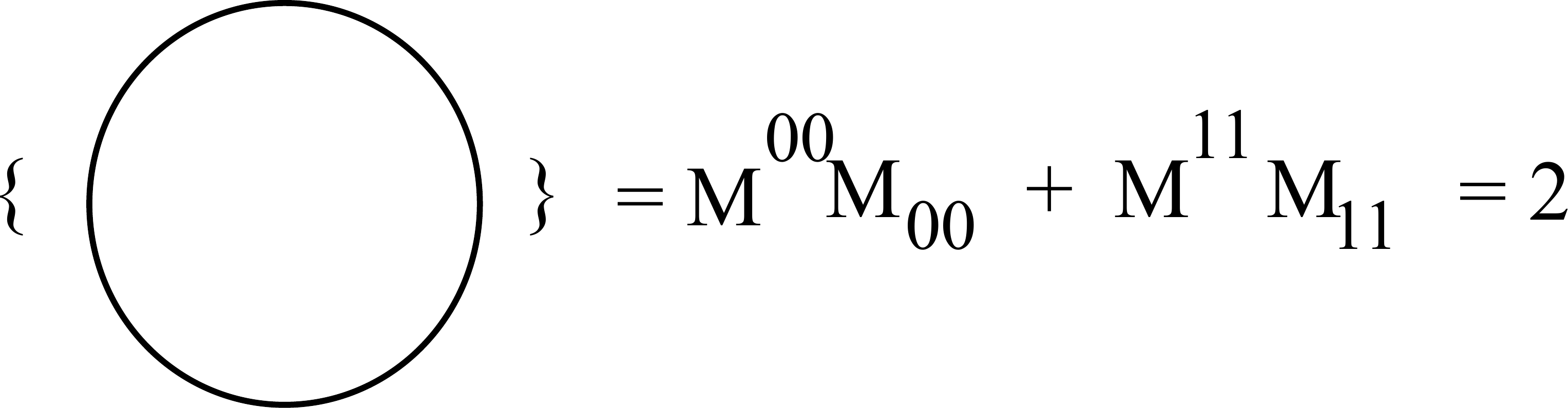}
\caption{}
\label{fig:bincircle}
\end{figure}

\begin{remark}\normalfont

Since cups and caps are assigned to the identity matrix, this model for the binary bracket polynomial remains invariant under the rotations of the arcs. Such model where the cups and caps are assigned to the identity matrix and the crossings are assigned to the $R$ matrices given by the crossing expansions
can be given also for the bracket polynomial \cite{Ka5}.  
\end{remark}


\section{Oriented Quantum Invariants}\label{sec:orientedquantum}
\subsection{General Schema }
In this section we present a general schema for oriented quantum models for invariants of oriented Morse knotoids. For this, we introduce \textit{right} and \textit{left oriented cups} and \textit{caps} that are assigned to multiples of the Kronecker delta, as in \ref{fig:oriented}, where $\mu$ is a constant and $\delta_{ab}$ is the Kronocker delta for some $a,b$ in an index set.  The topological invariance under the min-max move demands that the matrices assigned to a pair of right oriented or left oriented cup and cap to be inverses to each other.  


\begin{figure}[H]
\centering
\includegraphics[width=.8\textwidth]{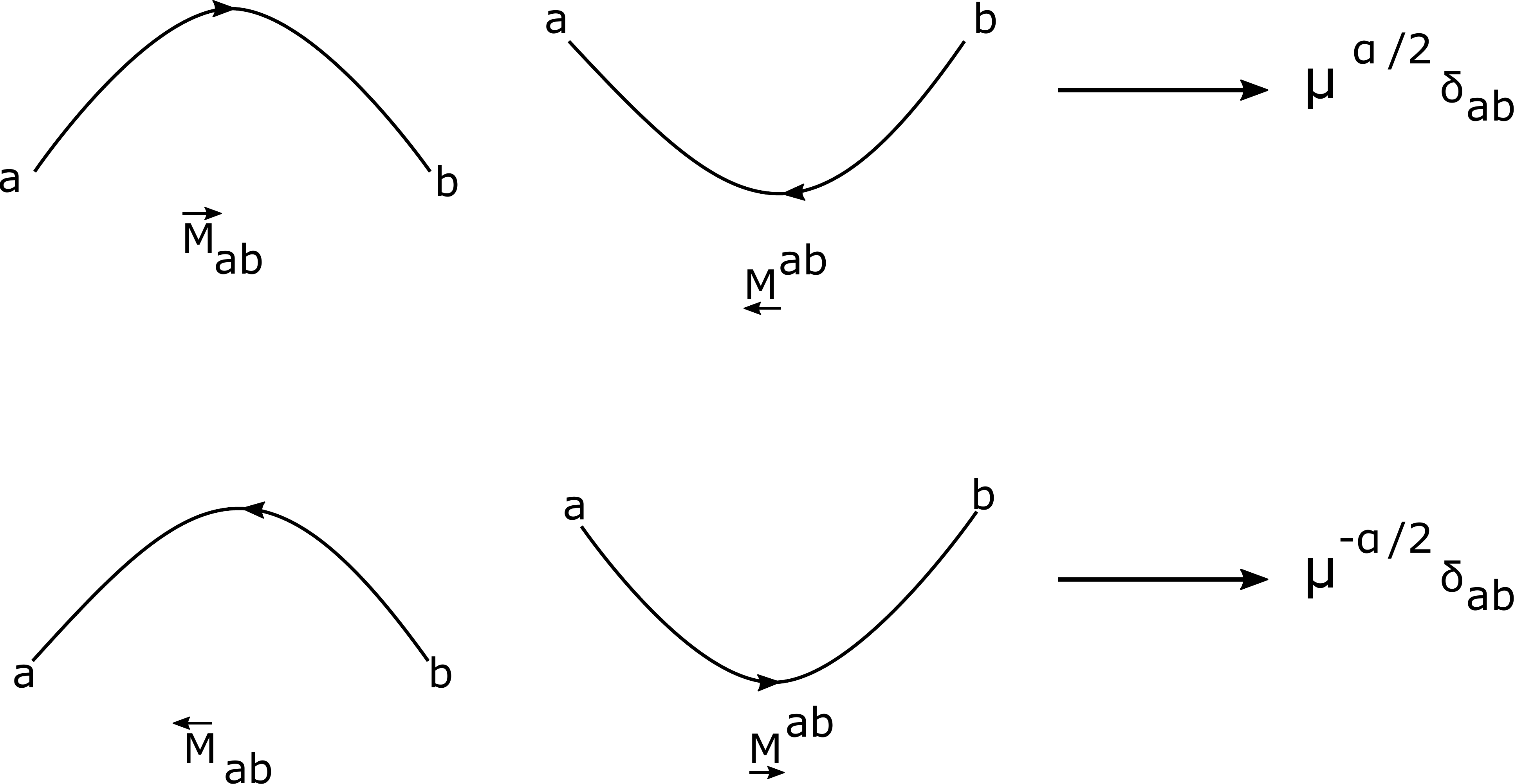}
\caption{Oriented cups and caps}
\label{fig:oriented}
\end{figure}

Crossings may appear as endowed with two types of local orientations, \textit{mixed} and \textit{parallel} type. In a paralel oriented crossing, the arcs at the crossing are oriented in the same way directing both up or down with respect to the bottom to top direction of the plane. In a mixed oriented crossing, the arrows on arcs of the crossing point to different directions. A mixed oriented crossing can always be converted to a parallel oriented crossing by regular isotopy as we show in Figure \ref{fig:mixed}. 

\begin{figure}[H]
\centering
\includegraphics[width=.92\textwidth]{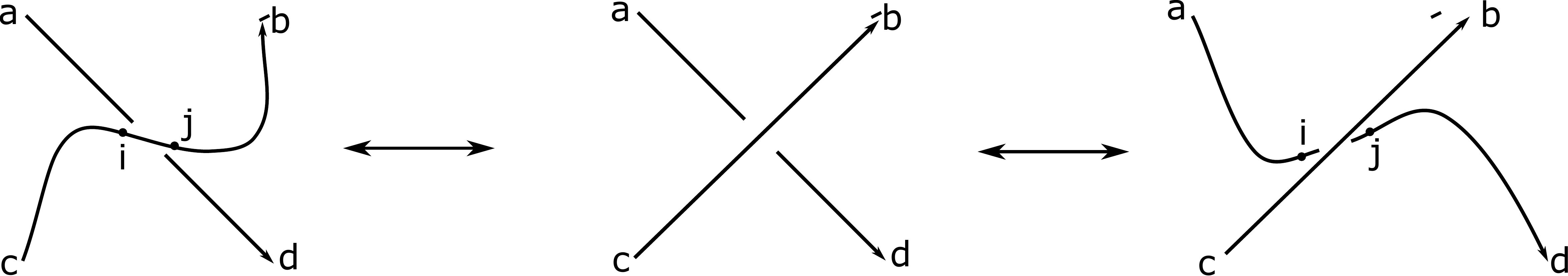}
\caption{Twist at a mixed negative crossing} 
\label{fig:mixed}
\end{figure}

 We assign $R, \overline{R}$ matrices to positive and negative crossings, respectively, that are both solutions of the Yang-Baxter equation so that oriented Morse isotopy type III moves are satisfied. To satisfy the parallel type Morse isotopy type II move (that is, the strands involved in the move are oriented in the same way), we assume $\overline{R}= R^{-1}$. 
 We also impose the following identity on cups, caps and $R, R^{-1}$ matrices induced by the anti-parallel Reidemeister II moves.
\begin{equation}\label{eqn:two}
\overrightarrow{M}_{jb}R^{tj}_{ia}\overrightarrow{M}^{si}\overleftarrow{M}^{sk} \overline{R}_{ck}^{lt} \overleftarrow{M}_{ld} = \delta^{a}_{c}\delta^{b}_{d},
\end{equation}

 where in an entry $R^{cd}_{ab}$, $a,b \in \mathcal{I}$ hold for the indices on the arcs going inward to the corresponding crossing and $c, d \in \mathcal{I}$ hold for the indices on the arcs emanating from the crossing, following the orientation.

Furthermore, the conversion of a mixed negative crossing to a parallel crossing enforces the following identity involving the cups and caps matrices and the $R^{-1}$ matrix assigned to negative crossing. Note that we have the variations of this identity involving a positive crossing so the $R$ matrix and we leave it to the reader to investigate these identities.
$$ \overrightarrow{M}_{ci} \overline{R}^{dj}_{ia}\underrightarrow{M}^{jb} = \underrightarrow{M}^{ai} \overline{R}^{jb}_{ci} \overrightarrow{M}_{jd}.$$

When we insert the values for cups and caps into this equation we find,
\begin{equation}\label{eqn:one}
\mu^{\frac{c}{2}} \overline{R}^{db}_{ca} \mu^{\frac{-b}{2}} =\mu^{\frac{-a}{2}}\overline{R}^{db}_{ca}\mu^{\frac{d}{2}}.
\end{equation}

\begin{definition}
An $R$ matrix solution of the Yang-Baxter equation is called \textit{spin-preserving} if $R^{ab}_
{cd}$ is zero whenever $a+b \neq c+d$. 
\end{definition}
We can deduce from the Equation \ref{eqn:one} that any $R$ solution of the Yang-Baxter equation matrix satisfying the conversion identity is spin-preserving. We also conclude that
 a partition function obtained by the products of the matrices $M$ and $R$ satisfying the above equations will give us an invariant  of a Morse knotoid diagram with fixed indices on its endpoints.

\subsection{The Alexander Polynomial for knotoids via a quantum state sum model}

In this section we construct the Alexander-Conway polynomial for Morse multi-knotoids. We do this by adapting the state sum model given in \cite{Ka5} that yields a solution of the Yang-Baxter equation. The importance of having the Yang-Baxter state sum model of the Alexander polynomial for Morse multi-knotoids lies in the fact that not every ascending or decending multi-knotoid diagram is Morse isotopic to the trivial knotoid diagram. This is equivalent to say that there exist multi-knotoid diagrams that admit no unknotting sequence. See Figure \ref{fig:obstructions} for some basic examples. This fact obstructs the Conway type skein identity (see the next section)  to be applied in a recursive computation of the Alexander polynomial since one encounters a multi- knotoid diagram that admits no unknotting sequence in a step of the recursion. 
\begin{figure}[H]
\centering
\includegraphics[width=.7\textwidth]{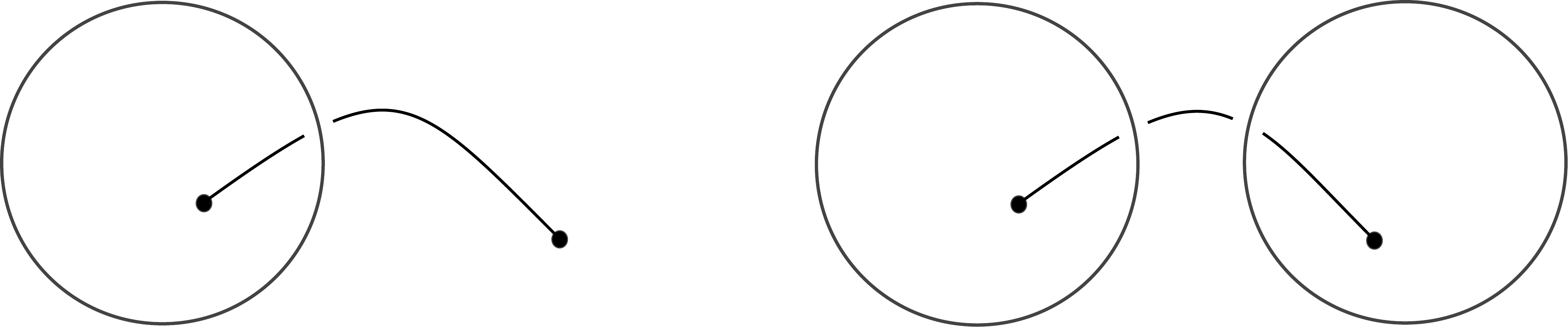}
\caption{Some of the irreducible bits avoiding the recursive computation}
\label{fig:obstructions}
\end{figure}

\subsubsection{A small review of the Alexander-Conway Polynomial}
Recall that the one-variable Alexander polynomial $\nabla_{K}(z)=\nabla_{K}$ of an oriented classical link $K$ is the unique polynomial determined by the following three properties \cite{Conway}:\\
\begin{enumerate}
\item $\nabla_{K} = \nabla_{K'}$ if $K$ is ambient isotopic to $K'$. 
\item $\nabla_K = 1$ if $K$ is the unknot.
\item $\nabla_{L_+} -\nabla_{L_-} = z\nabla_{L_0}$ where $L_{+}, L_{-} , L_{0}$ are links differing from each other at only one crossing, as shown in Figure \ref{fig:skein}.
\end{enumerate}
\begin{figure}[H]
\centering
\includegraphics[width=.6\textwidth]{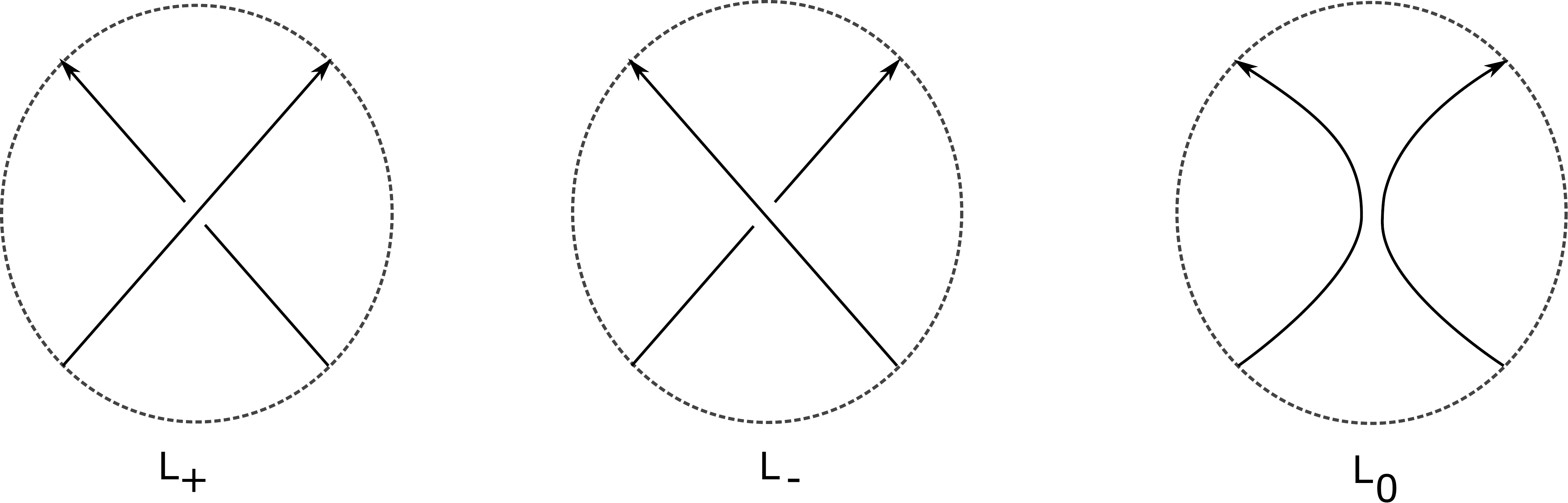}
\caption{The links in the skein relation of Alexander polynomial}
\label{fig:skein}
\end{figure}
The Conway skein identity (the third property given in the above list) causes that any split knot/link has vanishing Alexander polynomial \cite{Ka5}. Precisely, any oriented split link $L$ can be represented abstractly as the diagram $L_0$ given in Figure \ref{fig:split}. The links $L{+}$ and $L_{-}$ related to $L_0$ in the skein identity are clealry ambient isotopic. By the Property $1$, and the skein identity we have that $\nabla_{L_{0}}= \nabla_{L} = 0$.
\begin{figure}[H]
\centering
\includegraphics[width=.5\textwidth]{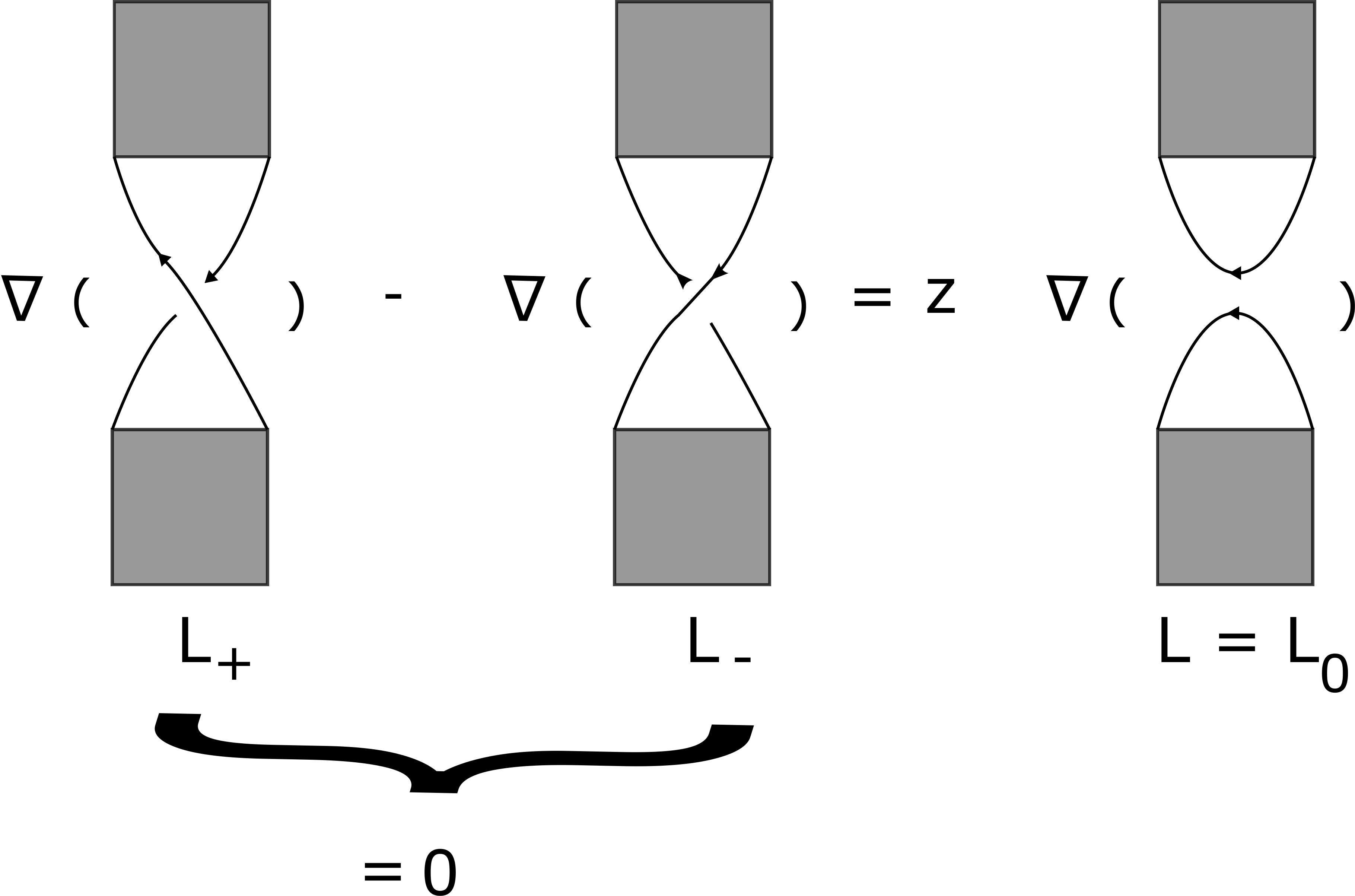}
\caption{The Alexander polynomial of split links vanishes.}
\label{fig:split}
\end{figure}

A state sum model  given for the Alexander polynomial that is of the form  $\sum_{\sigma} <K | \sigma> \delta^{\sigma}$, where $\delta$ is the value for state components, satisfies the following  split property. 
$$\nabla_{(O K)} = \nabla_{O}~\nabla_{K}.$$
Then with the observation above, we deduce that any state sum model for the Alexander polynomial vanishes for any oriented classical links. We can get over this obstruction by utilizing $(1,1)$-tangles for the construction of the Alexander polynomial for oriented classical links \cite{Ka5}. In the state sum model of the Alexander polynomial for $(1,1)$-tangles, the value of the trivial $(1,1)$-tangle is assumed to be $1$ and the value of the unknot is $0$. See Figure \ref{fig:alexx}. 
 \begin{figure}[H]
\centering
\includegraphics[width=.55\textwidth]{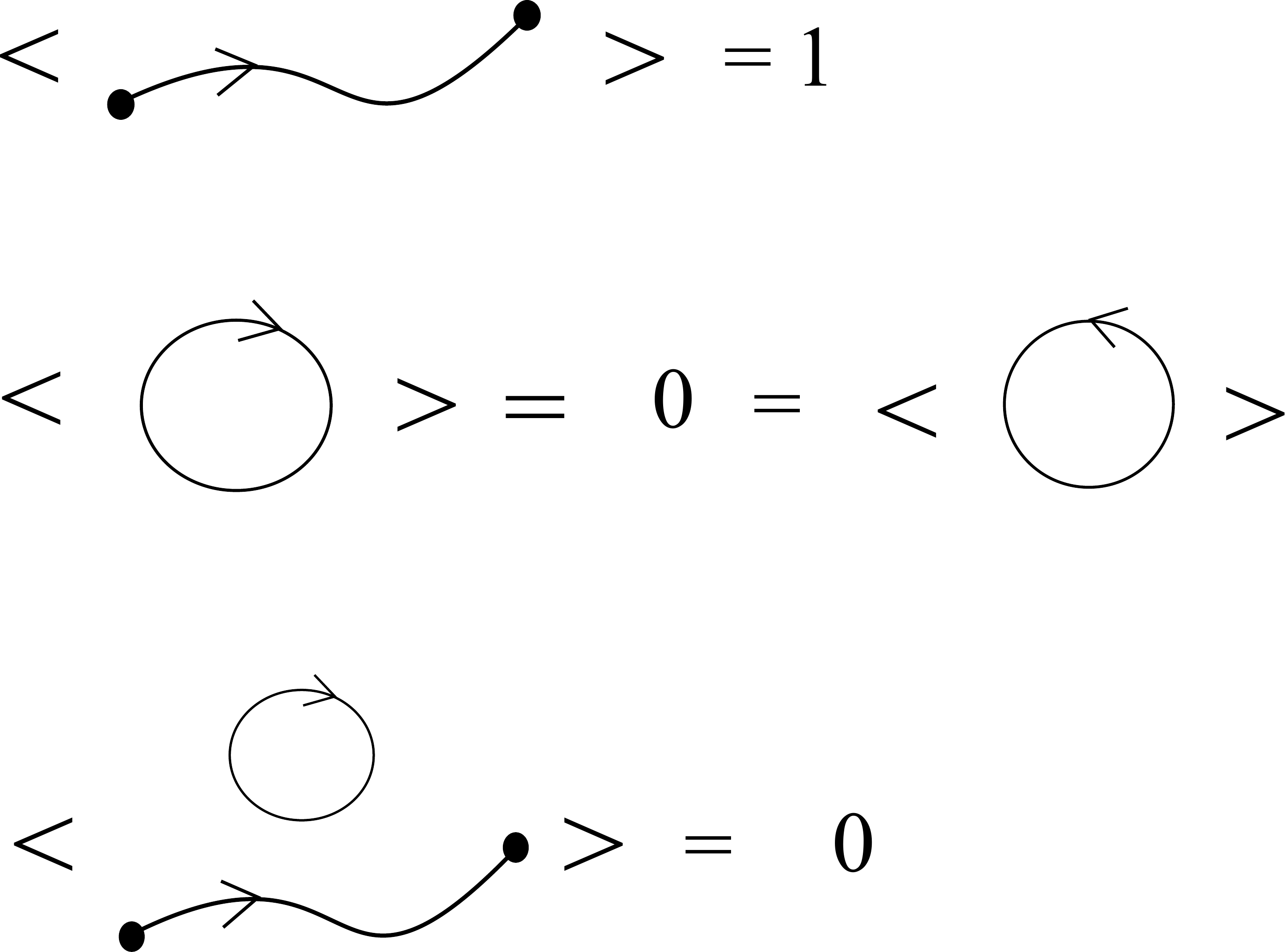}
\caption{The value of the trivial components}
\label{fig:alexx}
\end{figure}
The states of this model are obtained by expanding each crossing of an oriented $(1,1)$-tangle, as given in Figure \ref{fig:expansion}. The strands at the smoothing sites are labeled with indices from the index set $\mathcal{I}~=~\{0,1\}$. 
\begin{figure}[H]
\centering
\includegraphics[width=.8\textwidth]{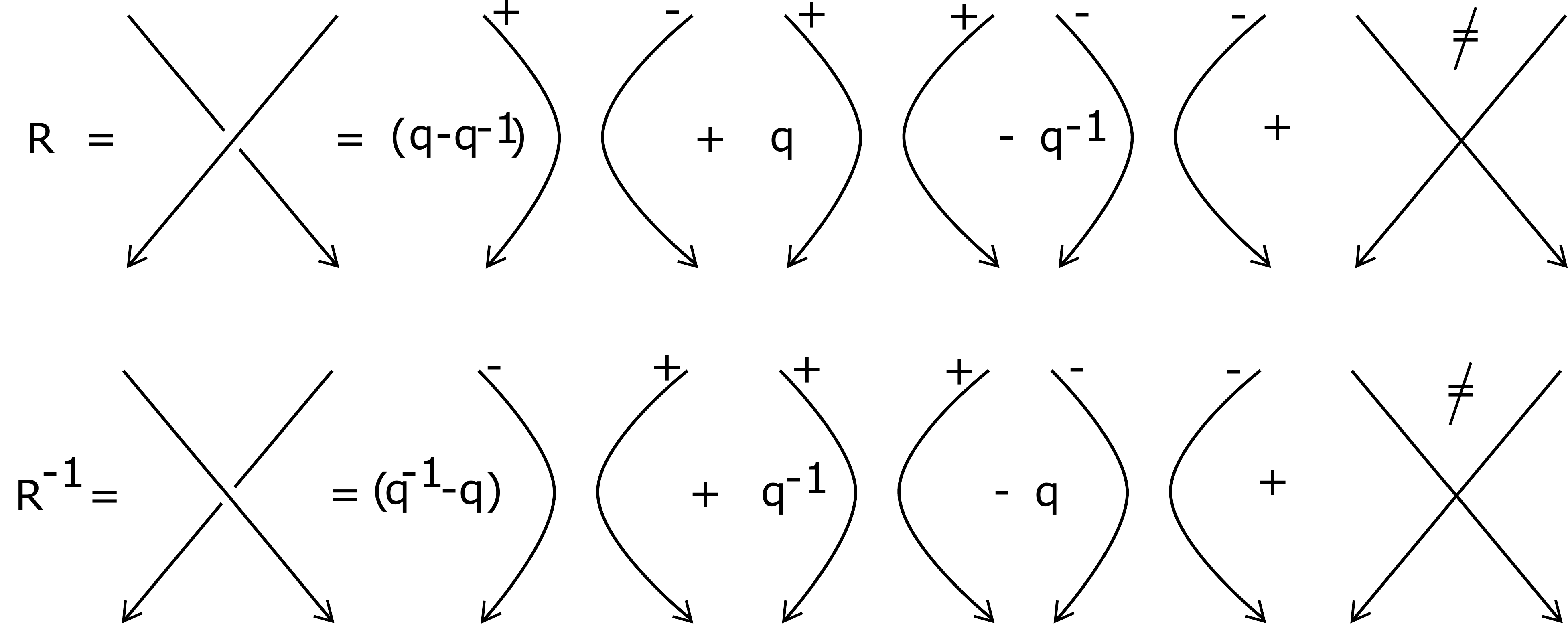}
\caption{The expansions at crossings}
\label{fig:expansion}
\end{figure}

This expansion utilizes the following $R$ and $R^{-1}$ matrices for the positive and the negative crossing, respectively.
$$R= \begin{bmatrix}
q & 0 &0 &0 \\
 0 & q-q^{-1} & 1& 0 \\
 0&1& 0 &0\\
 0&0&0&-q^{-1}\\
\end{bmatrix}, 
\hspace{.7cm}
R^{-1} = \begin{bmatrix}
q^{-1} & 0 &0 &0 \\
 0 & 0 & 1& 0 \\
 0&1&q^{-1} - q&0\\
 0&0&0&-q\\
\end{bmatrix}.$$

It is verified in \cite{Ka5} that the $R$ matrix (and its inverse $R^{-1}$) is a spin-preserving solution of the Yang-Baxter equation. 

Now notice that an oriented unknot has two states with labels $+$ and $-$, shown in Figure \ref{fig:alexcups}. Each state of the unknot is evaluated as $i^{-1}$ and $i$, respectively, that makes the total sum equals $0$. Accordingly, $i^{rot(\sigma)label(\sigma)}$ is assigned to a signed oriented circular state component $\sigma$.  Similarly, a trivial $(1,1)$ tangle has two states with labels $+$ and $-$ but its rotation number is counted to be $0$.  Therefore each of its states is assigned to $\frac{1}{2}$, making the sum equal to $1$.

\begin{figure}[H]
\centering
\includegraphics[width=.6\textwidth]{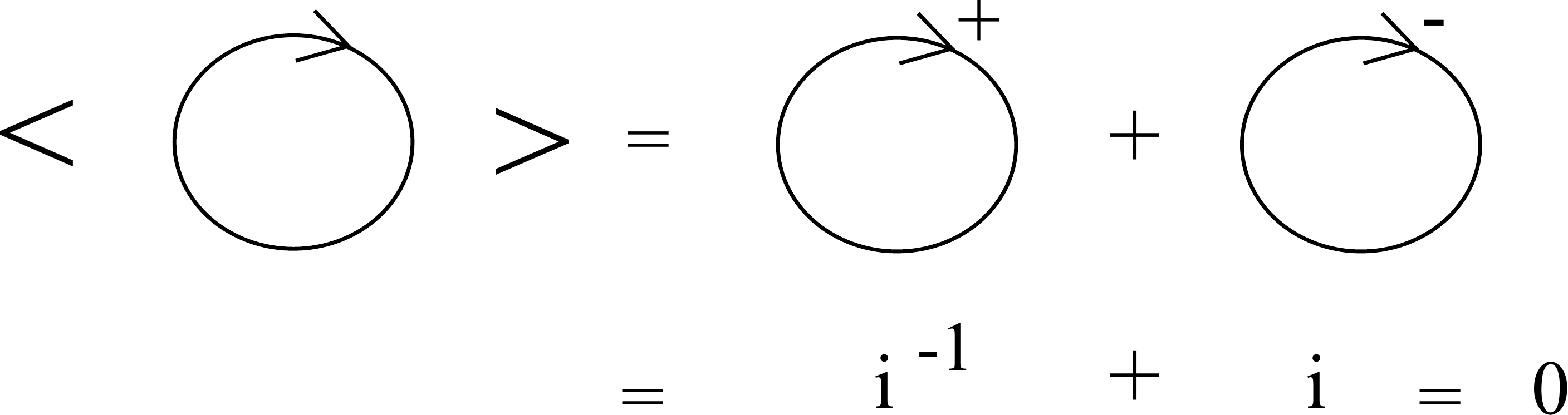}
\caption{The state sum evaluation of an oriented unknot.}
\label{fig:alexcups}
\end{figure}

It is then straightforward to see the matrices assigned to an oriented cup and a cap with a $+$, $-$ label, are the following matrices. See also Figure \ref{fig:alexorient} for an illustration.
  \begin{figure}[H]
\centering
\includegraphics[width=0.5\textwidth]{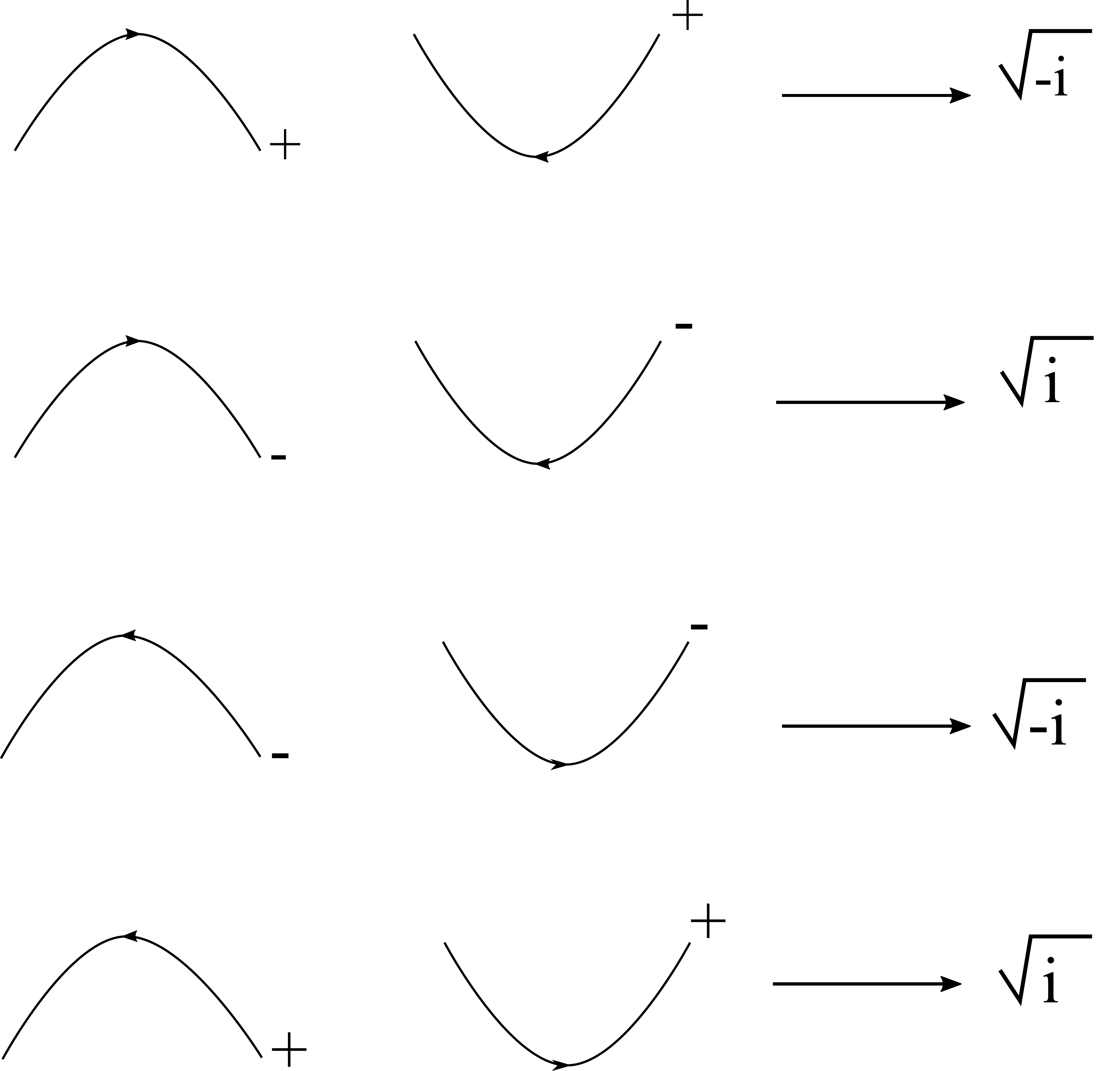}
\caption{Cups and caps assignment}
\label{fig:alexorient}
\end{figure}

$$ \overrightarrow{M_{ab}}
 = \underleftarrow{M^{ab}} = \begin{bmatrix}
\sqrt{-i} & 0 \\
 0 & \sqrt{i}  \\
\end{bmatrix},$$
$$\overleftarrow{M_{ab}} = \underrightarrow{M^{ab}} = \begin{bmatrix}
\sqrt{i} & 0 \\
 0 & \sqrt{-i}  \\
\end{bmatrix}.$$

The closed state sum formula for the Alexander polynomial of an oriented $(1,1)$-tangle is given as follows.

\begin{definition} \normalfont \cite{Ka5}
Let $K$ be an oriented tangle diagram representing an oriented $(1,1)$- tangle. The state sum polynomial of $K$, $\nabla(K)$ is defined as,
 $$\nabla_{K} = (iq^{-1})^{-rot(K)} \sum _{\sigma} <K | \sigma>\frac{1}{2}( i ^{||\sigma||}),$$
 where $rot(K)$ is the rotation number of $K$, that is the total sum of the half of the signs of the oriented cups and caps forming $K$,  $< K | \sigma>$ is the product of the coefficients at the smoothing sites of the state $\sigma$,  and the norm of $\sigma$, $||\sigma||$ is the sum of indices on circular closed components of the state $\sigma$, each multiplied by the rotation number of the corresponding state component.
 \end{definition}

\begin{proposition}\cite{Ka5}
The given state sum model satisfies the Conway skein indentity.
\end{proposition}

Figure \ref{fig:identity} verifies that the state model satisfies the Conway skein identity of the Alexander  polynomial. 
\begin{figure}[H]
\centering
\includegraphics[width=1\textwidth]{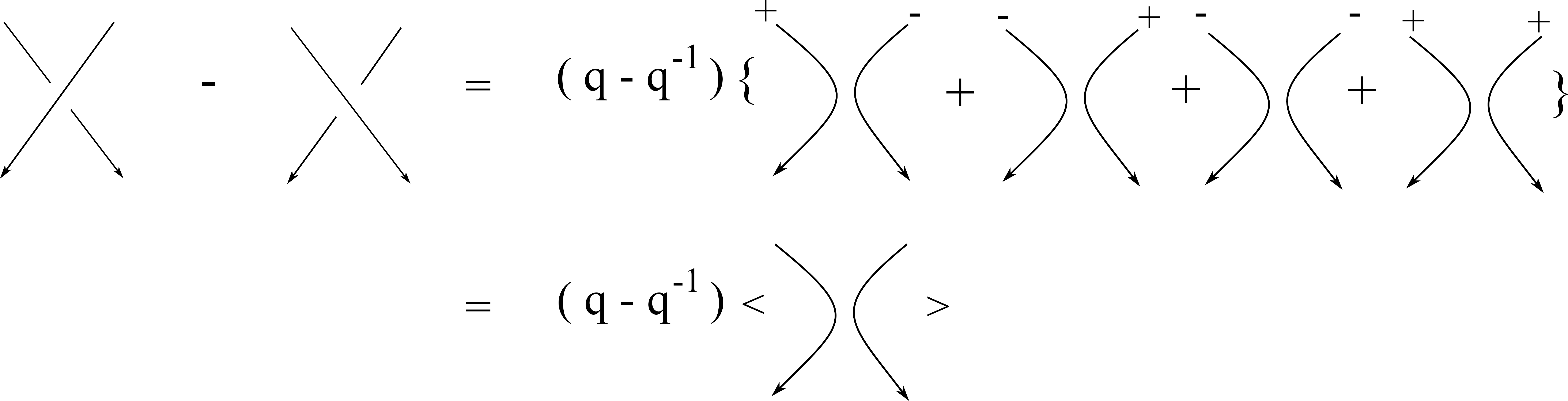}
\caption{The skein relation of the state model}
\label{fig:identity}
\end{figure}

\begin{proposition}\cite{Ka5}
The state sum polynomial becomes an ambient isotopy invariant of a classical knot/link $K$ when it is normalized with $(iq^{-1})^{-rot(K)}$. The trace of $(iq^{-1})^{-rot(K)}\nabla_{K}$ is the Alexander polynomial of $K$. 
\end{proposition}

\subsubsection{The Alexander polynomial for Morse knotoids}
The $R$, $R^{-1}$ and the $M$ matrices assigned to oriented cups and caps of the quantum state model given in the former section, can be utilized for defining the  Alexander polynomial  Morse knotoid diagrams. 
Let $K^{a}_{b}$ denote an oriented Morse knotoid diagram with its endpoints labeled with $a, b \in \{-, +\}$. 

As in the $(1,1)$-tangle case,  we obtain the states of an oriented Morse knotoid diagram $K^{a}_{b}$ by the crossing expansions of the model. Each state of $K^{a}_{b}$ contains exactly one open-ended component containing the endpoints of $K$ and a number of circular components. Each state component (open-ended or closed) is oriented and labeled with either $+$ or $-$, and is composed of a number of cups and caps. We will assume that each state component contributes to the Alexander polynomial as a product of the corresponding values in the oriented cups and caps matrices. See Figure \ref{fig:alexorient}, and note that  a single $+$ or $-$ means that both ends have the same labels. As a result, the value of an open-ended state component $\lambda$ is nontrivial only if  the endpoints admit the same labeling; either $+$ or $-$. Thus, the value of an open-ended state component $\lambda$ is equal to $i^ {label (\lambda)rot(\lambda)}$. 
\begin{figure}[H]
\centering
\includegraphics[width=0.35\textwidth]{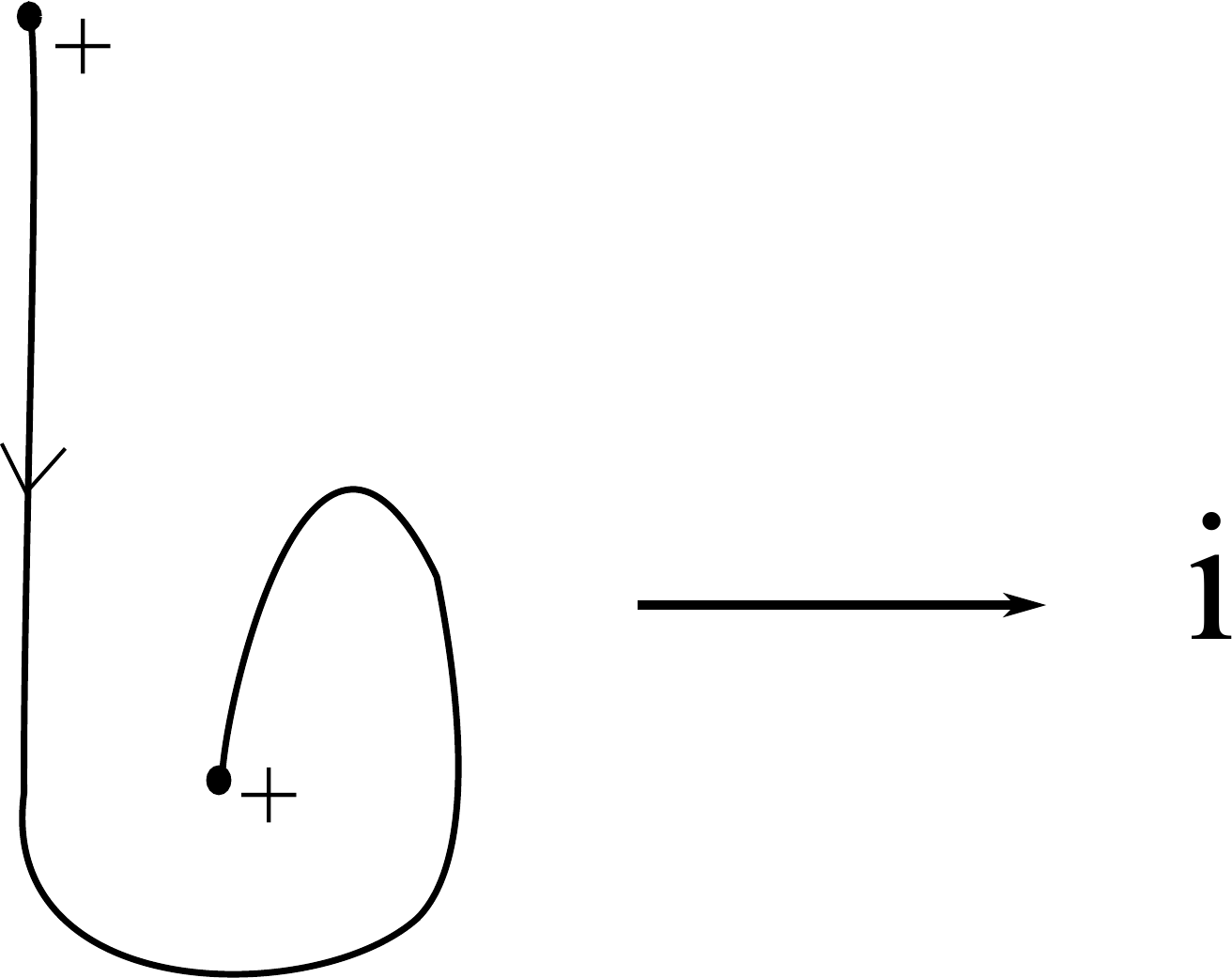}
\caption{The evaluation of an open-ended state component with $+$ label and rotation number $1$}
\label{fig:alexorient}
\end{figure}


\begin{definition}\normalfont
The state sum $\tilde{\nabla}$ of a Morse knotoid $K$ is defined as:
$$\tilde{\nabla}_
{K}= \sum _{\sigma} <K | \sigma> i ^{||\sigma||},$$
where  $rot(K)$, $< K| \sigma>$ are defined similarly as above, and $|| \sigma ||$ is defined to be the sum of labels of loop components and the long segment component in the state $\sigma$ multiplied by the rotation number assigned to the components. 
\end{definition}


\begin{proposition}
The matrix determined by the state sum $\tilde{\nabla}$ is a Morse isotopy invariant.
\end{proposition}
\begin{proof}
Showing that the state sum model is invariant under the oriented type II and type III Morse isotopy moves goes similarly with showing that the polynomial is invariant under the classical Reidemeister moves of classical knots and links. The reader is referred to \cite{Ka5} for illustrations of invariance for the classical case.
\end{proof}

Let us examine the polynomial under the Reidemeister type I knotoid isotopy moves. As shown in the following figures, the state sum polynomial changes by either $q^{-1}i$ or $-qi$ according to the rotation of the curl created by the type I moves. Therefore the state sum multiplied by the factor $iq^{-rot(K)}$ becomes invariant under all the oriented knotoid isotopy moves. Since any knotoid in $\mathbb{R}^2$ can be represented by its standard Morse diagram, the normalized polynomial is an invariant of planar knotoids.

\begin{figure}[H]
\centering
\includegraphics[width=1\textwidth]{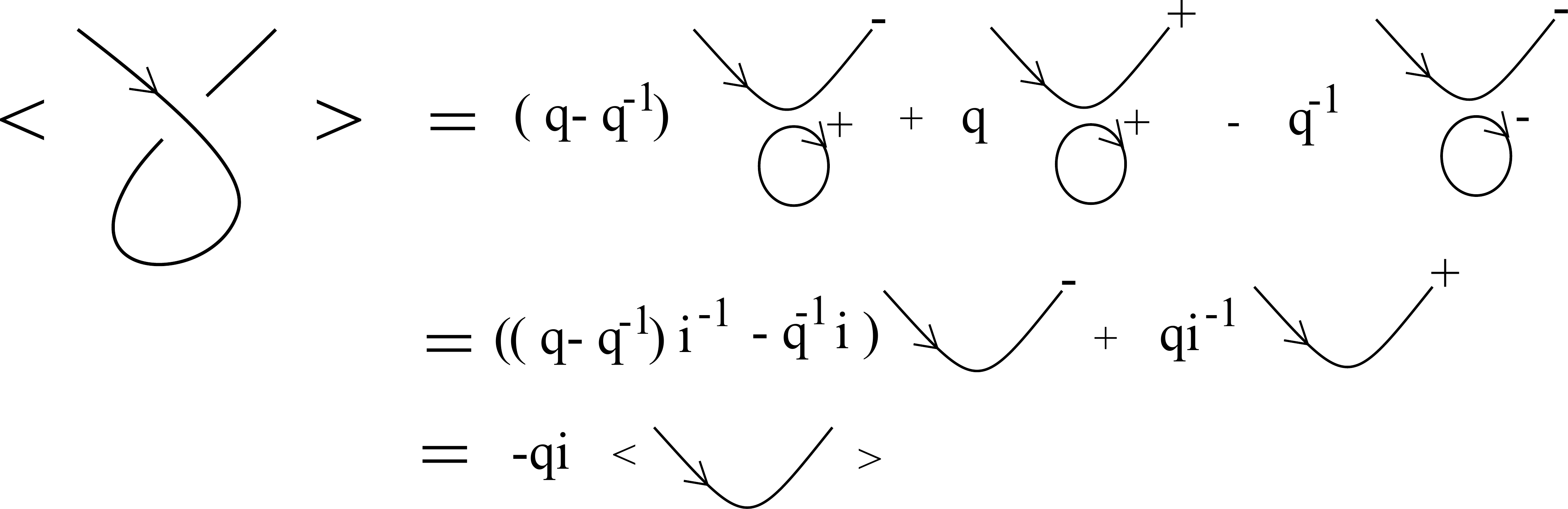}
\caption{The change of $\tilde{\nabla}$ under an oriented type I move}
\label{}
\end{figure}
 As done in Figure \ref{fig:identity}, we can verify the normalization of this state sum polynomial satisfies the Conway skein identity. Since any knotoid $\kappa$ in $\mathbb{R}^2$ has a unique standard Morse knotoid representation $K$, we can define the state sum polynomial for $\kappa$, so that  $\tilde{\nabla}_{\kappa}=  \tilde{\nabla}_{K}$. 
\begin{corollary}
Let $K$ be a knotoid in $\mathbb{R}^2$. The normalization of $\tilde{\nabla}(K)$, 
$\nabla_{K} = (iq)^{-rot(K)} \tilde{\nabla}_{K}$ and its trace, $tr(\nabla_{K})$ are both invariants of planar knotoids.  
\end{corollary}

\begin{definition}\normalfont
Let $K$ be a knotoid in $\mathbb{R}^2$. The half of the trace of the matrix $\nabla_{K} = iq^{-rot(K)} \tilde{\nabla}_{K}$ is the Alexander polynomial of $K$. 

\end{definition}
\begin{example}\label{ex:multi}\normalfont
The multi-knotoid diagram given in Figure \ref{fig:compirreducible} admits no unknotting sequence. This obstacles the computation of the Alexander polynomial by the skein relation.  Nevertheless, the quantum state sum model provides us a way to compute its Alexander polynomial.  
We see from the figure that the knotoid component  of $K$ contributes to the rotation number trivially and thus the total rotation number of $K$ is $-1$ contributed by the knot component of $K$. The first two states have rotation number $1$, and the last two states have rotation number $-1$. 

The the state sum matrix is given as follows.

$$\tilde{\nabla}_K = \begin{bmatrix}
(q-1)i  &0 \\
 0&(q^{-1}+1)i \\
\end{bmatrix}.$$

By normalizing the matrix $\tilde{\nabla}_K$ and taking the half of its trace we find the Alexander polynomial of $K$ as $\frac{1}{2}( iq (qi - q^{-1} i ^{-1}+ i^{-1} + i) )= \frac{1}{2}(-q^2 - 1)$. 
\end{example}\normalfont
\begin{figure}[H]
\centering
\includegraphics[width=1\textwidth]{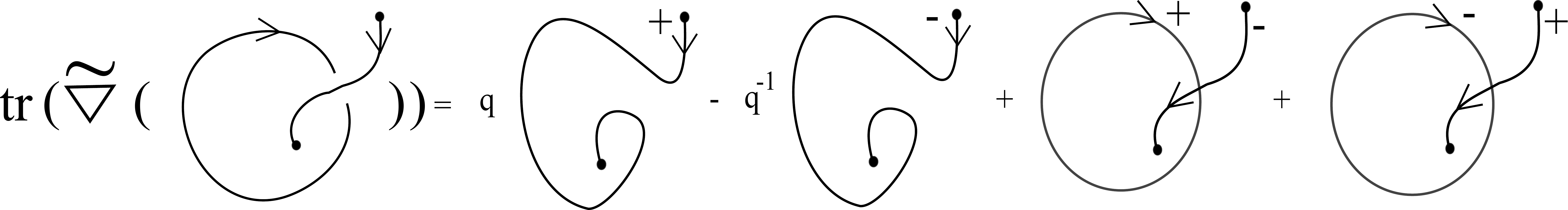}
\caption{The Alexander polynomial computation of a multi-knotoid}
\label{fig:compirreducible}
\end{figure}

\begin{example}\normalfont
In Figure \ref{fig:compirreduciblee}, we see a knotoid $K$ and its states. The states given on the top line are the states of $K$  whose open ended component is labeled $+,+$ at its endpoints, and the two states on the bottom line are the states of $K$  whose open-ended components are labeled with $-,-$ at its endpoints. Since the rotation number of its components is $-1$, the evaluation of the states on the top is given as $q^2 i^{-1}$ and $(q-q^{-1}) i$, respectively and the evaluation of the states on the bottom is given as $-q^{-2}i$ and $(q^{-1}-q) i^{-1}$, respectively. The state sum matrix is given as follows.

$$\tilde{\nabla}_K = \begin{bmatrix}
(-q^2+(q-q^{-1}))i  &0 \\
 0&( -q^{-2}+(q-q^{-1}))i \\
\end{bmatrix}.$$
The rotation number of $K$ is $-1$. Then, by normalizing $\tilde{\nabla}_K$ by $iq$, we find the Alexander polynomial of $K$ is $\frac{1}{2}(q^3+q^{-1} -2q^{2}+2)$. 

Notice that we can also apply the Conway skein identity to compute the Alexander polynomial of $K$ since we now know the Alexander polynomial of the multi-knotoid given in Example \ref{ex:multi}.

\end{example}
\begin{figure}[H]
\centering
\includegraphics[width=.5\textwidth]{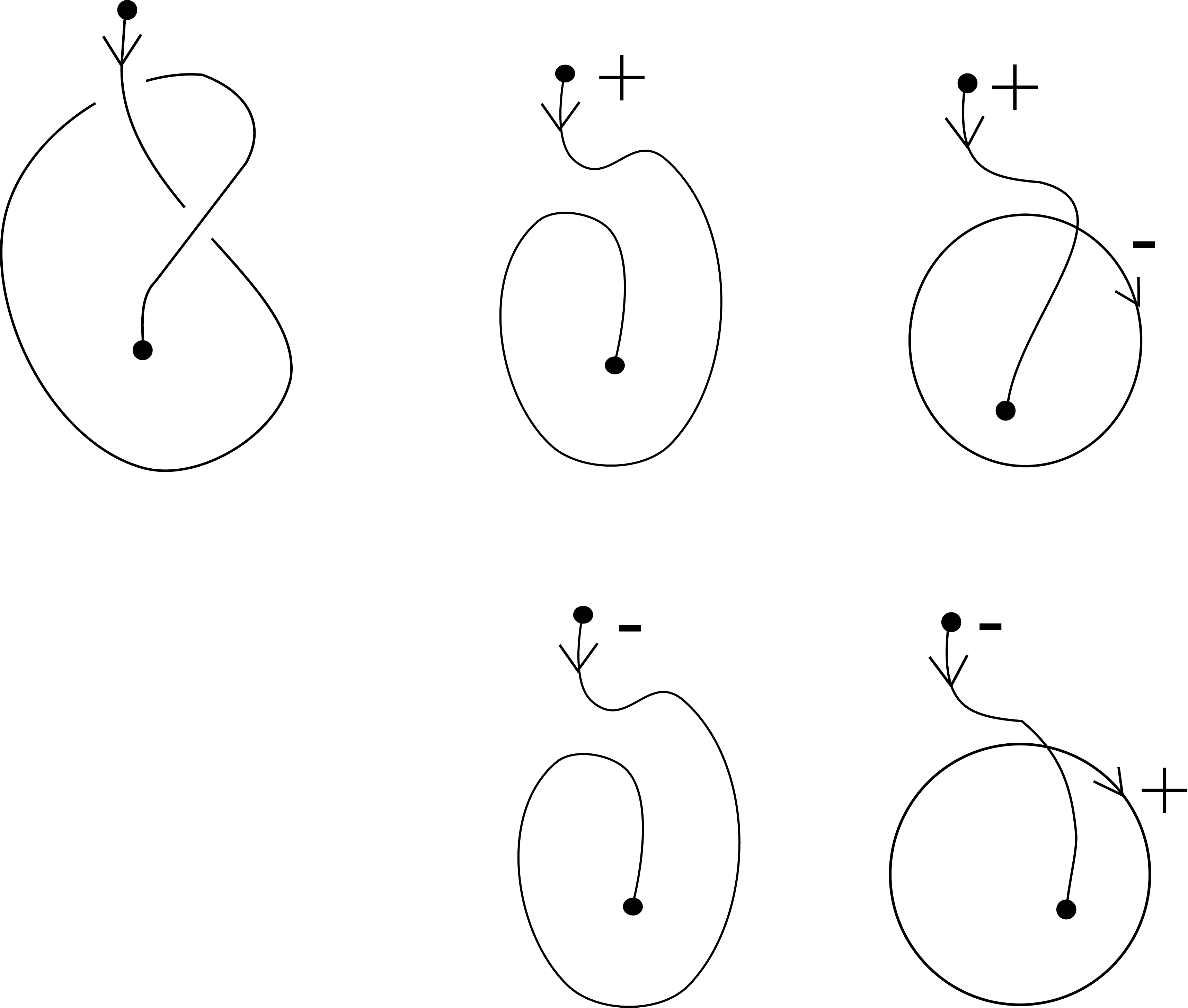}
\caption{A knotoid and its states.}
\label{fig:compirreduciblee}
\end{figure}



\subsection{A generalization of the Alexander polynomial}\normalfont

A generalization of the Alexander polynomial, namely the Sawollek polynomial was defined for virtual knots and links \cite{KaRad,Saw} as the determinant of the $2 \times 2$ matrix representing the crossing relations of a given Alexander biquandle coloring of a virtual knot/link diagram. 

A \textit{biquandle} $X$  is a set endowed with four binary operations satisfying a number of axioms that are motivated by the Reidemeister moves when the elements of $X$ is considered to be associated to the edges of a knot or link diagram (or a knotoid/multi-knotoid diagram). Satisfying the oriented Reidemeister III moves, a biquandle can be considered as a solution to the Yang-Baxter equation.  In Figure \ref{fig:alexbiqrel}, we present the four operations of a specific biquandle called the \textit{Alexander biquandle}, defined at a positive and a negative crossing, respectively. The reader is referred to \cite{FK, NG} for more on biquandles and the induced invariants of classical, virtual knots and knotoids. 

\begin{figure}[H]
\centering
\includegraphics[width=.75\textwidth]{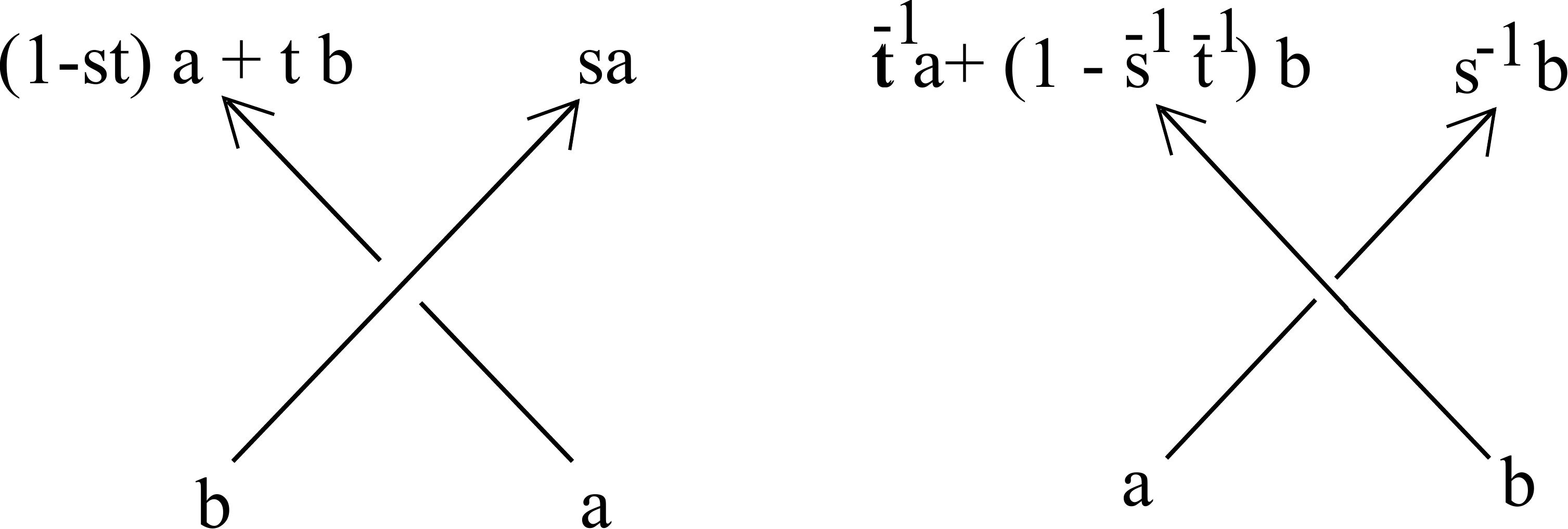}
\caption{The Alexander biquandle relations at crossings}
\label{fig:alexbiqrel}
\end{figure}

 We now adapt the quantum state sum model given in \cite{KaRad} to construct the Sawollek polynomial for Morse knotoids. This model is based on the \textit{generalized Burau matrix} $B$ that represents the \textit{Alexander biquandle} operation rule at a positive crossing. Precisely, we consider a positive crossing operation as a linear transformation $T: V \rightarrow V$ on a two dimensional module $V$ with basis $\{e_1,e_2\}$ defined by $T(e_1)= (1-st) e_1 +t e_2$, $T(e_2) = se_1$.  Thus the matrix $B$ is given as follows.
 $$B=\begin{bmatrix}
1-st & s \\
 t & 0 \\
\end{bmatrix}$$
It is clear that the inverse of $B$ is induced by the negative crossing operations and is given as follows.
 $$B^{-1}=\begin{bmatrix}
0 & t^{-1} \\
s^{-1} & 1-s^{-1}t^{-1}  \\
\end{bmatrix}.$$
We extend the linear transformation $T$ to the exterior algebra of $V$, $\Lambda^{*} V$ with basis $\{1, e_1, e_2, e_1\wedge e_2\}$. The extension $T^{*} : \Lambda^{*} V \rightarrow \Lambda^{*} V$ is determined by the rules,
$$T^*(1) = 1,$$
$$T^*(e_1) = (1-st)e_1 + te_2,$$
$$T^*(e_2) = se_1,$$
$$T^*(e_1 \wedge e_2) = Det(B) e_1 \wedge e_2 = -st e_1 \wedge e_2.$$
Thus the linear transformation  $T^*$ and its inverse $T^*{-1}$ are represented by the following matrices. The matrices $R, R^{-1}$ are solutions of the Yang-Baxter equation \cite{KaRad}.
 $$R=\begin{bmatrix}
1& 0&0 &0 \\
0 & 1-st &s&0  \\
0&t&0&0\\
0&0&0&-st\\
\end{bmatrix},
\hspace{.6cm}
R^{-1}=\begin{bmatrix}
1& 0&0 &0 \\
0 & 0 &t^{-1}&0  \\
0&s^{-1}&1-s^{-1}t^{-1}&0\\
0&0&0&-s^{-1}t^{-1}\\
\end{bmatrix}.$$

The matrices $R$ and $R^{-1}$ induce a state sum model for the generalized Alexander polynomial of oriented Morse knotoids in such a way that the entries of the matrix $R$ contribute to the state expansion at a positive crossing as local vertex weights, and the entries of the matrix $R^{-1}$ contribute to the state expansion at a negative crossing as local vertex weights, as shown in Figure \ref{fig:genexpansion}. It is convenient to do the following changes of variables, 

$$s \leftrightarrow \sigma^2 , \hspace{2cm} t \leftrightarrow \tau^{-2},$$
$$R \leftrightarrow \sigma^{-1}\tau R ,\hspace{2cm}  R^{-1} \leftrightarrow \sigma\tau^{-1} R^{-1},$$
$$z \leftrightarrow \sigma^{-1}\tau -  \sigma\tau^{-1}.$$
Thus we obtain the matrices:

 $$R=\begin{bmatrix}
\sigma^{-1}\tau& 0&0 &0 \\
0 & z&\sigma\tau^{-1}&0  \\
0&\sigma^{-1}\tau^{-1}&0&0\\
0&0&0&-\sigma\tau^{-1}\\
\end{bmatrix},
\hspace{.6cm}
R^{-1}=\begin{bmatrix}
\sigma \tau^{-1}& 0&0 &0 \\
0 & 0 &\sigma \tau &0  \\
0&\sigma^{-1}\tau^{-1}&-z&0\\
0&0&0&-\sigma^{-1}\tau\\
\end{bmatrix}.$$
 The expansion is applied at each crossing of an oriented Morse knotoid diagram. This results in collection of the  states, each containing a number of circular components and an open-ended segment component, each with an orientation and a $+$ or $-$ label on them.
\begin{figure}[H]
\centering
\includegraphics[scale=.25]{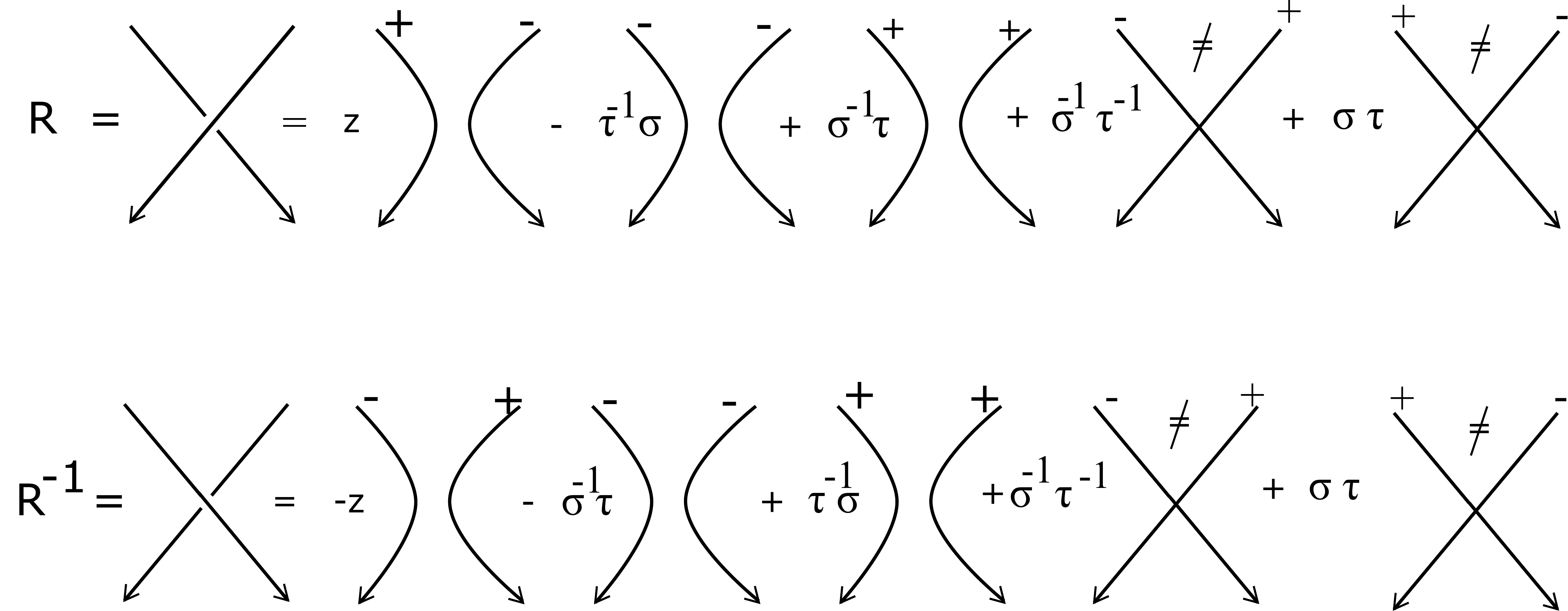}
\caption{}
\label{fig:genexpansion}
\end{figure}
Plus or minus labeled cups and caps of the state components of a Morse knotoid diagram are evaluated the same with the labeled cups and caps of the Alexander polynomial state model, shown in Figure \ref{fig:alexorient}. That is, each state component $s$ is evaluated as $i^ {rot(s)label(s)}$. 
\begin{definition}\normalfont
We define the state sum polynomial of a Morse knotoid diagram $K$ corresponding to this state expansion as follows:\\
$$Z(K)(\sigma, \tau)= \frac{1}{2} \sum_{s} <K|s> i^{||s||},$$
where the summation runs over all states obtained by smoothing the crossings of $K$ and labeling the components with $+$ and $-$,  $<K | s>$ is the product of the local vertex weights in state $s$, and $||s||$ is the sum of the product of labels of the state components with their rotation number.
\end{definition}
Note that $Z(K)$ here denotes the state sum polynomials $Z(K^{+}_{+})$ and $Z(K^{-}_{-})$, where $K^{+}_{+}, K^{-}_{-}$ denote the Morse knotoid diagram $K$ with endpoints labeled $+,+$ and $-,-$, respectively. Thus $Z(K)$ yields a $2\times2$ diagonal matrix with entries these polynomials.

It is not hard to check the invariance of $Z(K)$ under the Morse isotopy type II and type III moves, see \cite{Kauf}  where the verification is done for virtual knots and links. Therefore, the matrix determined by $Z(K)$ is a Morse isotopy invariant.
\begin{proposition}
The  polynomial $Z(K)$ is a $2$-variable generalization of the Alexander polynomial of Morse knotoids.
\end{proposition}
\begin{proof}
Substituting $s=1$ in the $R$ matrix of the state sum $Z(K)$ gives the $R$ matrix of the Alexander polynomial up to a change of basis. See \cite{Ka5} for details.
\end{proof}
 
In Figure \ref{fig:genalexxx}, we show that $Z(K)$ is not invariant under a type I Morse isotopy move that adds a negative curl to a vertical strand. In fact, by considering the other variations of type I moves, the reader can verify that the state sum polynomial $Z(K)$ changes by $(i\sigma^{-1}\tau)^{-rot(K)}$ under a type I Morse isotopy move. Then, $Z(K)$ can be normalized by the factor $(i\sigma\tau^{-1})^{rot(K)}$ and determines a $2\times2$ matrix invariant for knotoids in $\mathbb{R}^2$. 

\begin{figure}[H]
\centering
\includegraphics[width=.85\textwidth]{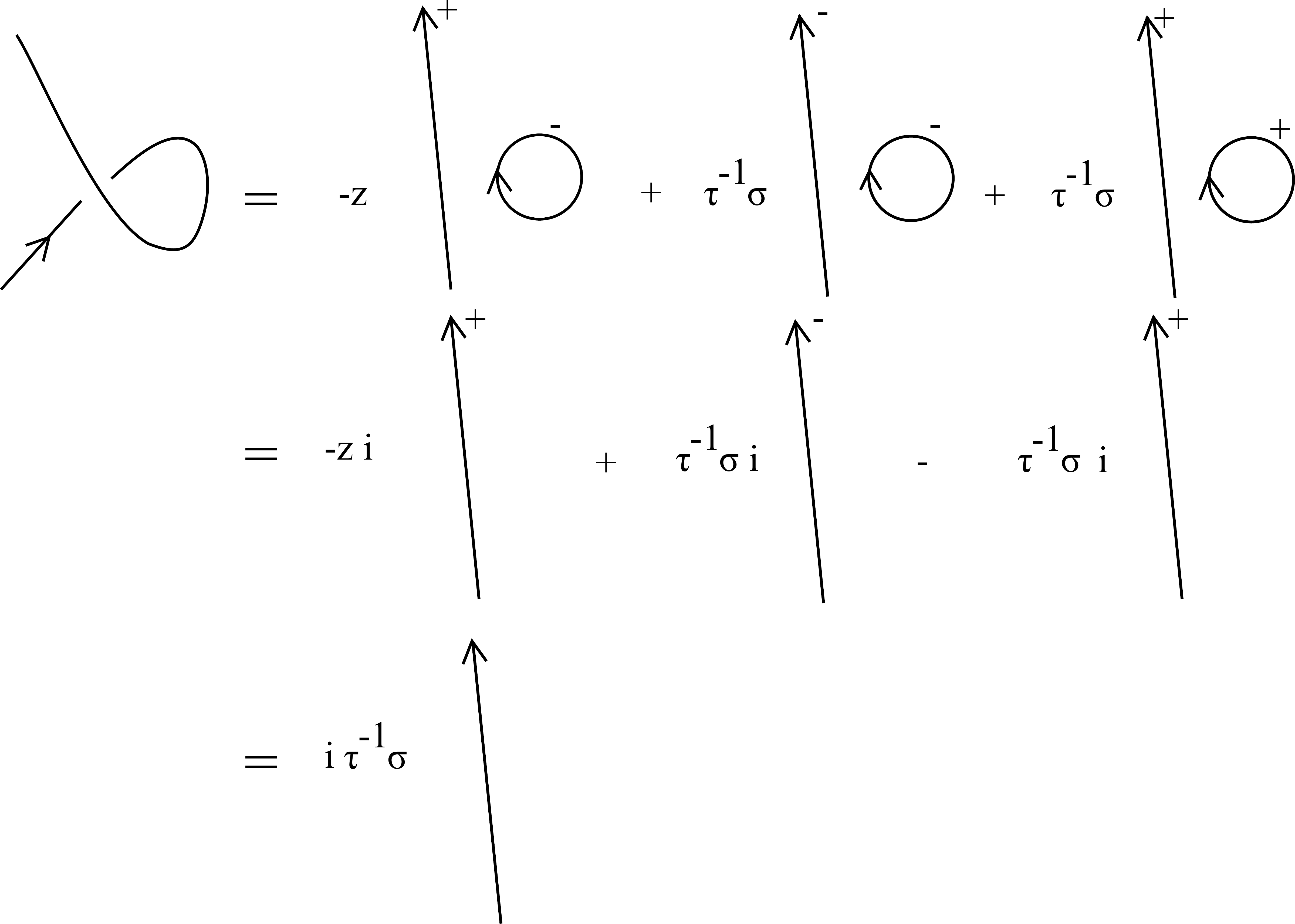}
\caption{The change under type I Morse isotopy move}
\label{fig:genalexxx}
\end{figure}

\begin{definition}\normalfont
The trace of the matrix given by $W(K) = (i\sigma\tau^{-1})^{rot(K)} Z(K)$ is called the \textit{Sawollek polynomial } of Morse knotoids in  $\mathbb{R}^2$.
\end{definition}



\begin{proposition}
The Sawollek polynomial is an invariant of knotoids in $\mathbb{R}^2$. 
\end{proposition}
\begin{proof}
Every knotoid in $\mathbb{R}^2$ has a unique standard Morse isotopy representation. By the invariance discussion above, the trace of the matrix given by $W(K)$ is invariant under the knotoid isotopy moves.  Then the statement follows.
 \end{proof}
\begin{lemma}\label{lem:virtclos}
Let $K$ be a Morse knotoid diagram in $\mathbb{R}^2$ that is of knot-type (the endpoints lie in the same planar region of the diagram). Then, the trace of the matrix determined by the partition function $Z(K)$ is equal to $i$ times the value of the quantum state sum on the virtual closure of $K$. 

\end{lemma}
\begin{proof}
It is clear that the virtual closure connects the endpoints of $K$ by creating a cup and a cap, oriented in the same direction. Thus, the contributions of the added cup and cap to the partition function of the virtual closure knot $\overline{v}(K)$ are the same and either as $\sqrt{i}$ or as $\sqrt{-i}$. The partition function of $\overline{v}(K)$ is then given as $Z(\overline{v}(K))= \sum _{a,b\in \{1,2\}}  \sqrt{\pm i} Z(K_{b}^{a}) \sqrt{\pm i} = i(Z^1_1 + Z^2_2)= i tr([Z(K)]$.

\end{proof}
\begin{theorem}\label{thm:sawollek}
Let $K$ be a Morse knotoid diagram in $\mathbb{R}^2$ that is of knot-type. Then, the Sawollek polynomial of the virtual closure of $K$ is equal to $-\sigma\tau^{-1} tr(W(K))$.

\end{theorem}
\begin{proof}
This follows directly by Lemma \ref{lem:virtclos}.
\end{proof}
\begin{theorem}
The Sawollek polynomial of a knot-type knotoid is trivial.

\end{theorem}
\begin{proof}
The virtual closure of a knot-type knotoid is a classical knot \cite{GK1}.The Sawollek polynomial vanishes on classical knots as discussed in \cite{KaRad}. Then by Theorem \ref{thm:sawollek}, it vanishes on knot type knotoids.

\end{proof}

\begin{corollary}
If the Sawollek polynomial of a knotoid in $\mathbb{R}^2$ is not trivial then the knotoid is a proper knotoid.
\end{corollary}
\begin{note}\normalfont
In \cite{KaRad}, it is shown that the state sum polynomial $Z(K)$ on a closed virtual knot is equal to the Sawollek polynomial of the virtual knot originally defined via the generalized Burau representation. For the Alexander and Sawollek polynomials of virtual knots, the state sum uses a nontrivial matrix at virtual crossings. The relationship of our quantum knotoid versions and the virtual closures needs further investigation.
\end{note}



\subsection{An infinity of specializations of the Homflypt polynomial}
A quantum state sum model that yields an infinite number of specializations of the Homflypt polynomial of classical knots and links is discussed in In \cite{Ka5, Jones1}. We adapt this model to the knotoid case to obtain specializations of the Homflypt polynomial.

In this state sum model, \textit{states} of a given knotoid diagram are obtained by replacing locally each tangle containing a crossing by a combination of decorated diagrams as given in Figure \ref{fig:h-expansion}.
\begin{figure}[H]
\centering
\includegraphics[width=.75\textwidth]{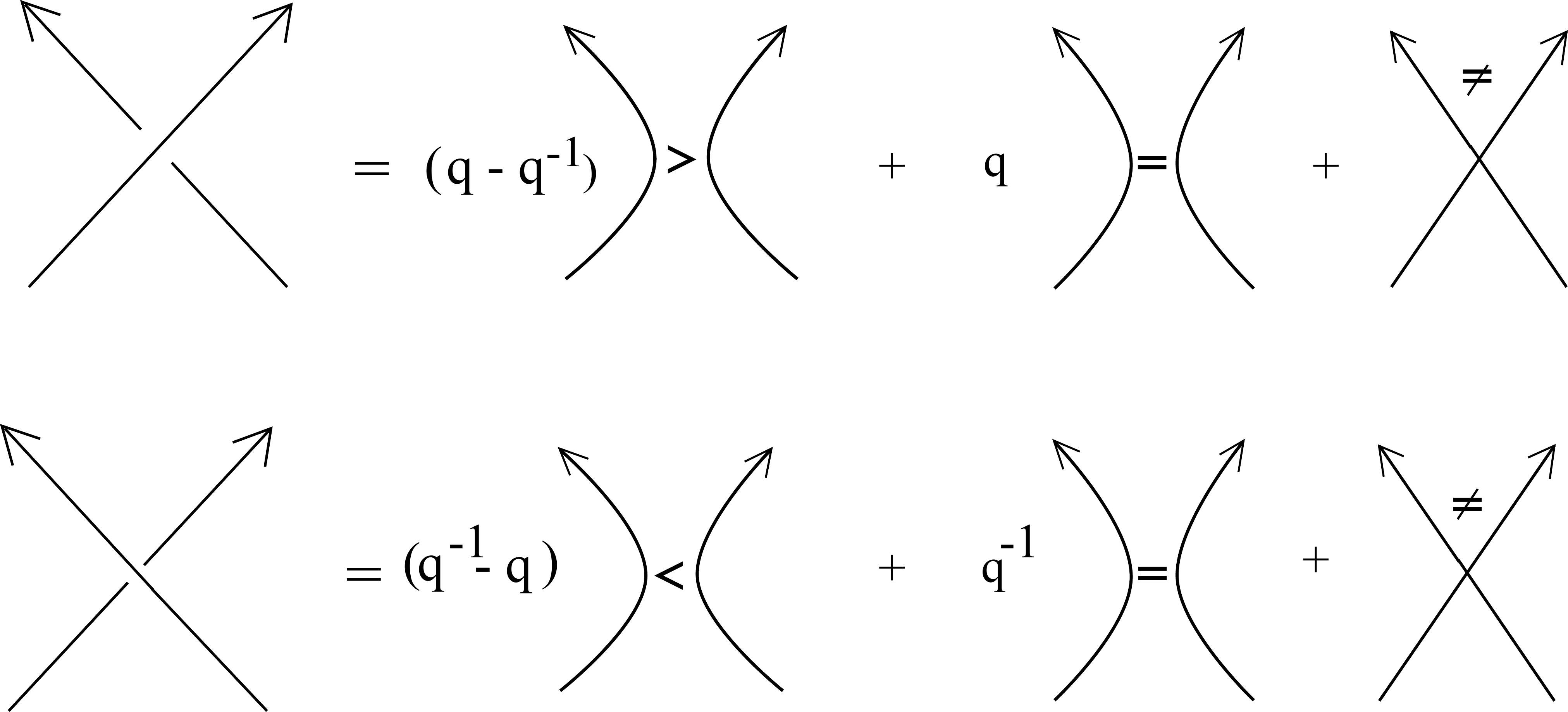}
\caption{The crossing expansions}
\label{fig:h-expansion}
\end{figure}
The $<$, $>$, $=$ and $\neq$ signs on the tangles refer to a labeling of the edges at a crossing by an index set $\mathcal{I} \subset \mathbb{Z}$. Precisely, $<$ appearing at an oriented smoothing site of a positive crossing indicates that the label of the strand on the left hand side of the sign is less than the  label of strand on the right hand side of the sign, $=$ indicates that the index labelings of strands on the left and on the right of the sign are equal, and $\neq$ indicates that labels of the crossed strands are not equal.

We pick an index set in the form $I_{n} = \{-n, -n+2, -n+4, ..., n-2, n\}$ for any positive integer $n$ to label the strands of a Morse knotoid diagram $K$. 
 We consider a state as an \textit{admissable state} if it admits a well-defined labeling from $I_{n}$, according to the signs between its components that are put in place of crossings. State components that are either a union of circular components and an open-ended component or an open-ended component may intersect with each other but none of the components can intersect itself as the labeling would not be well-defined in that case. See Figure \ref{fig:state}.

\begin{figure}[H]
\centering
\includegraphics[width=1\textwidth]{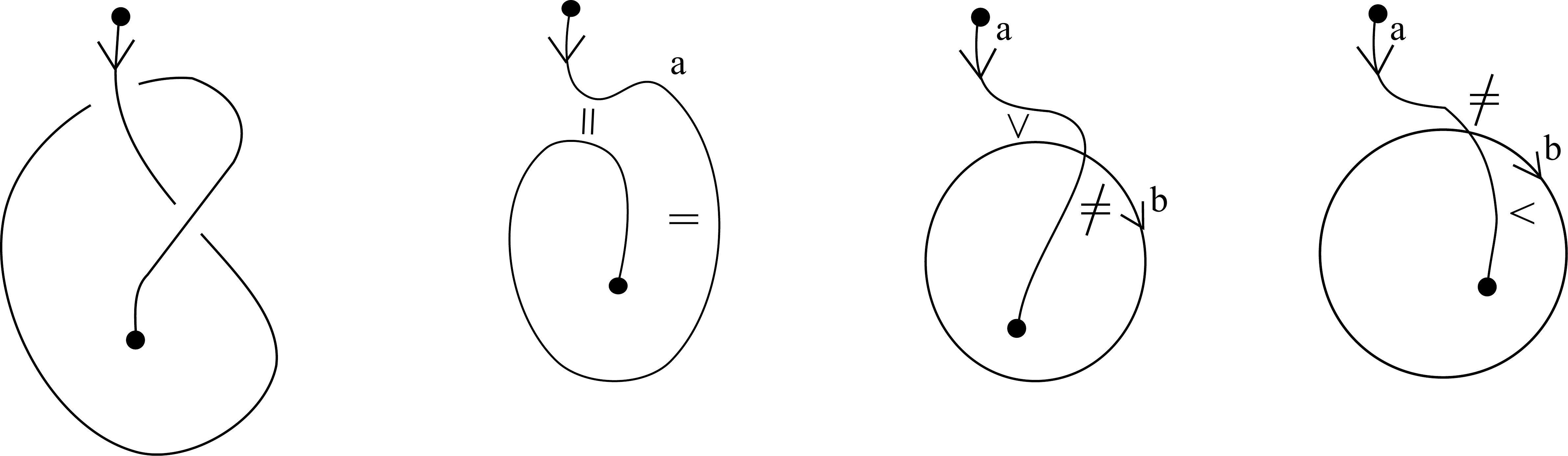}
\caption{A knotoid diagram and its admissable states}
\label{fig:state}
\end{figure}

The rotation number $rot(s)$ for a state component $s$ is defined in the same way as defined in Section \ref{sec:rotation}. Each circular state component or open-ended state component $s$ contribute to the polynomial with $q^{rot(s).label(s)}$.  

\begin{definition}\normalfont
The state sum polynomial of a knotoid diagram $K^{a}_{b}$ with fixed indices at its endpoints, denoted by $<K^{a}_{b}>$, is defined as the sum of the evaluations of all admissable state components as follows.
$$<K^{a}_{b}> = \sum_{\sigma} <K^{a}_{b}|\sigma> q^{||\sigma||},$$
where $<K^{a}_{b}|\sigma>$ is the multiplication of the coefficients at the smoothing sites of the admissable state $\sigma$, $|| \sigma || = \sum_{s \in \sigma} rot(s).label(s)$.
\end{definition}
The state sum polynomial of $K^{a}_{b}$ determines a $2n \times 2n$ matrix, for each $a, b\in I_{n}$. This matrix is diagonal since the non-diagonal entries are given by the labeling of $K$ with two different indices from $I _n$ at its endpoints. This causes a not well-defined labeling on the open-ended state components and so vanishing polynomials as there would be no admissable states.

\begin{example}\normalfont

We study the knotoid diagram given in Figure \ref{fig:state}. Assuming both endpoints are labeled with $a$, where $a \in I_n$, and summing up the contributions of all states of $K^{a}_{a}$ we find 
$$<K^{a}_{a}>= q^2 
q^{-a} + (q-q^{-1}) (\sum_{a > b} q^{-b} + \sum _{b > a} q^{-b} ) $$.



\end{example}

\begin{proposition}
The state sum model satisfies the regular isotopy version of the Homflypt skein relation:
\begin{figure}[H]
\centering
\includegraphics[width=.7\textwidth]{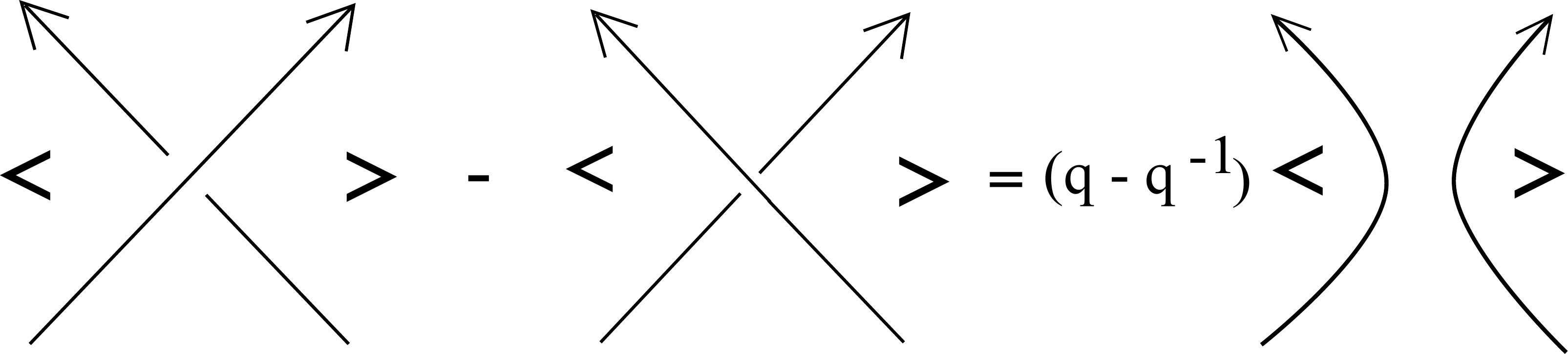}
\caption*{}
\label{fig:}
\end{figure}

\end{proposition}
\begin{proof}
Side by side subtraction of the crossing expansions results in:
\begin{figure}[H]
\centering
\includegraphics[width=.85\textwidth]{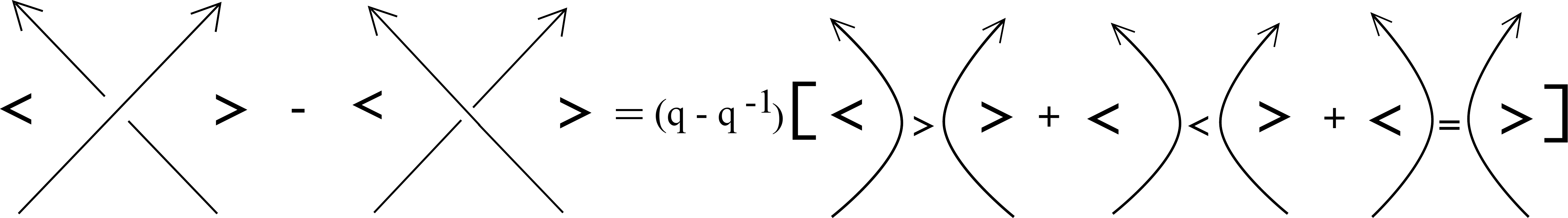}
\caption*{}
\label{fig:}
\end{figure}
Notice that the right hand side of the equality above is:
\begin{figure}[H]
\centering
\includegraphics[width=.7\textwidth]{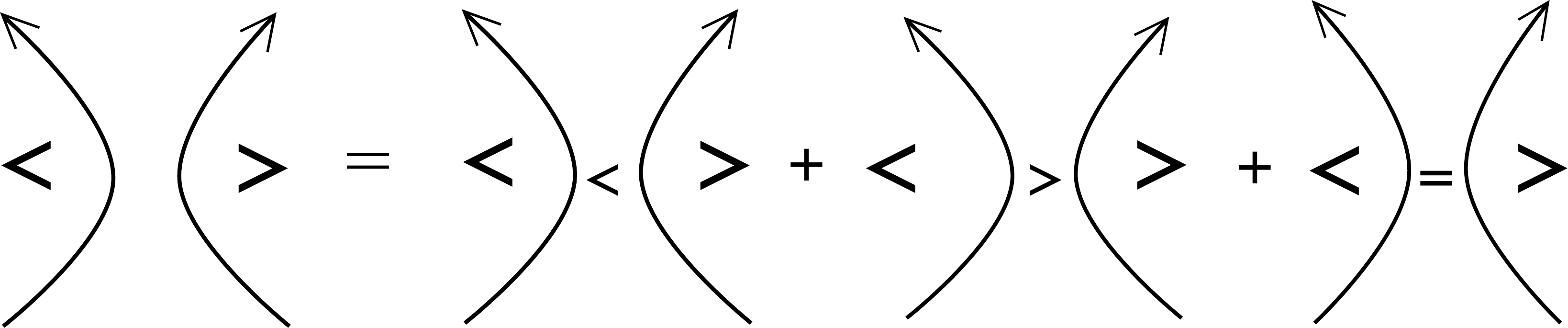}
\caption*{}
\label{fig:}
\end{figure}
Then the skein identity follows.

\end{proof}

This state sum polynomial of a Morse knotoid diagram with a fixed labeling at its endpoints can be viewed as a quantum state sum by rewriting the positive crossing expansion as $$R=  (q - q^ {-1})[i > j] \delta^{i}_{j}\delta^{j}_{l} + q [i = j] \delta^{i}_{k}\delta^{j}_{l} + [i \neq j]\delta^{i}_{l}\delta^{j}_k,$$ where 

$
[i < j ] =  \{
        \begin{array}{ll}
            1 & \quad i < j, \\
            0 & \quad otherwise
        \end{array}
  $

$
[i = j ] =  \{
        \begin{array}{ll}
            1 & \quad i = j, \\
            0 & \quad otherwise
        \end{array}
    $
\hspace{.1cm}

$[i \neq j ] =  \{
        \begin{array}{ll}
            1 & \quad i \neq j, \\
            0 & \quad otherwise
        \end{array}.
$

The matrices corresponding to cups and caps are determined by the following assignments.
\begin{figure}[H]
\centering
\includegraphics[width=.7\textwidth]{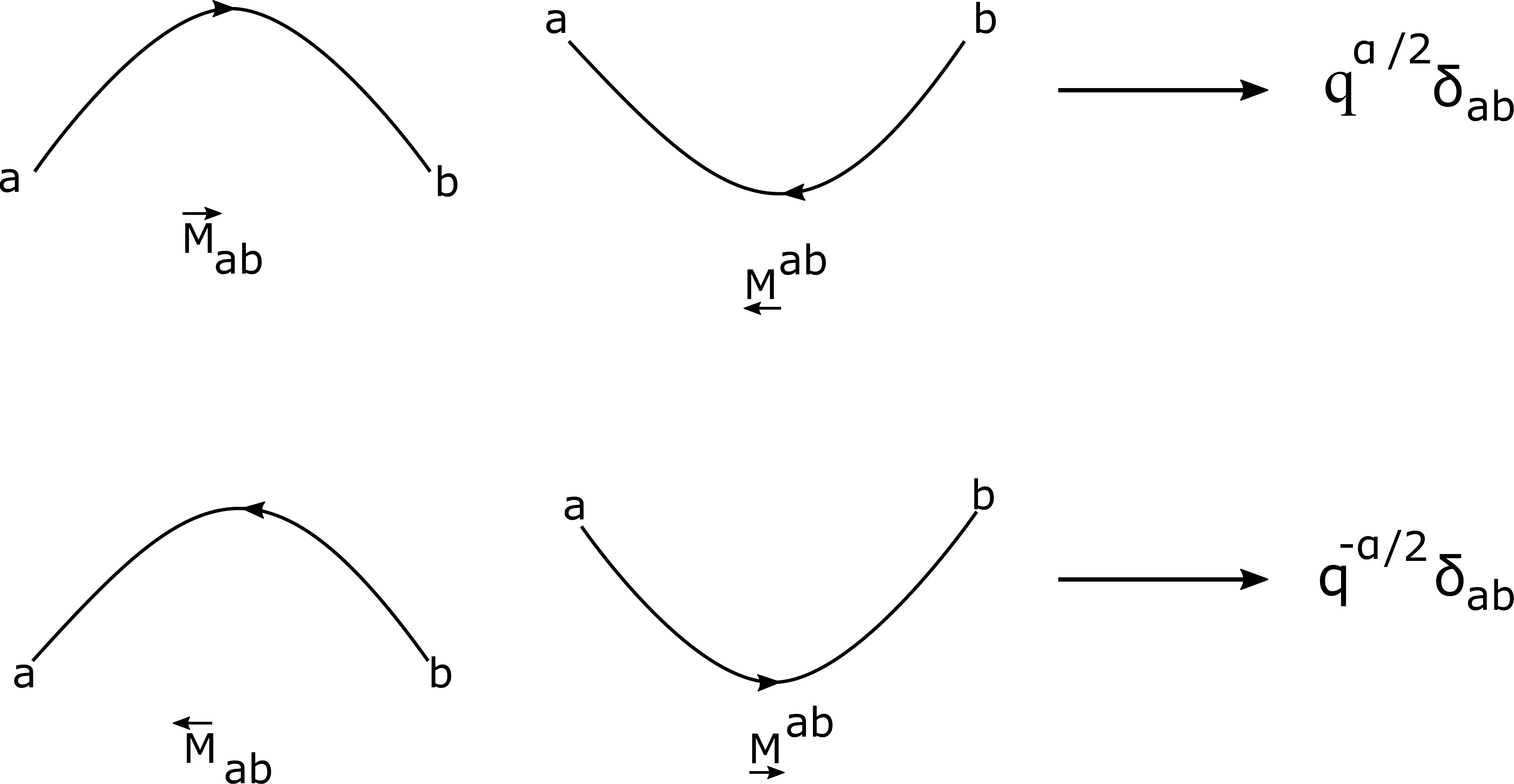}
\caption*{}
\label{fig:}
\end{figure}
where $$\delta_{ab} = \{
        \begin{array}{ll}
            1 & \quad  a= -b, \\
            0 & \quad otherwise.
        \end{array}
    $$

\begin{lemma}
The matrices induced by the positive crossing and the negative crossing expansions given in Figure \ref{fig:h-expansion} are solutions of the Yang-Baxter equation. 
\end{lemma}
\begin{proof}
See \cite{Ka5}.
\end{proof}

\begin{proposition}
The state sum polynomial $< >$ of a Morse knotoid diagram with fixed labels at its endpoints is invariant under the second and third oriented Morse knotoid isotopy moves. 

\end{proposition}

\begin{proof}
See \cite{Ka5} for the case of virtual knots and links. The proof can be applied directly to the case of Morse knotoid diagrams. 

\end{proof}

The state sum polynomial $< >$ changes under type I Reidemeister moves moves, as shown in Figure \ref{fig:homomo}. The verification of this and the remaining variants of the move is left to the reader.

\begin{figure}[H]
\centering
\includegraphics[width=.5\textwidth]{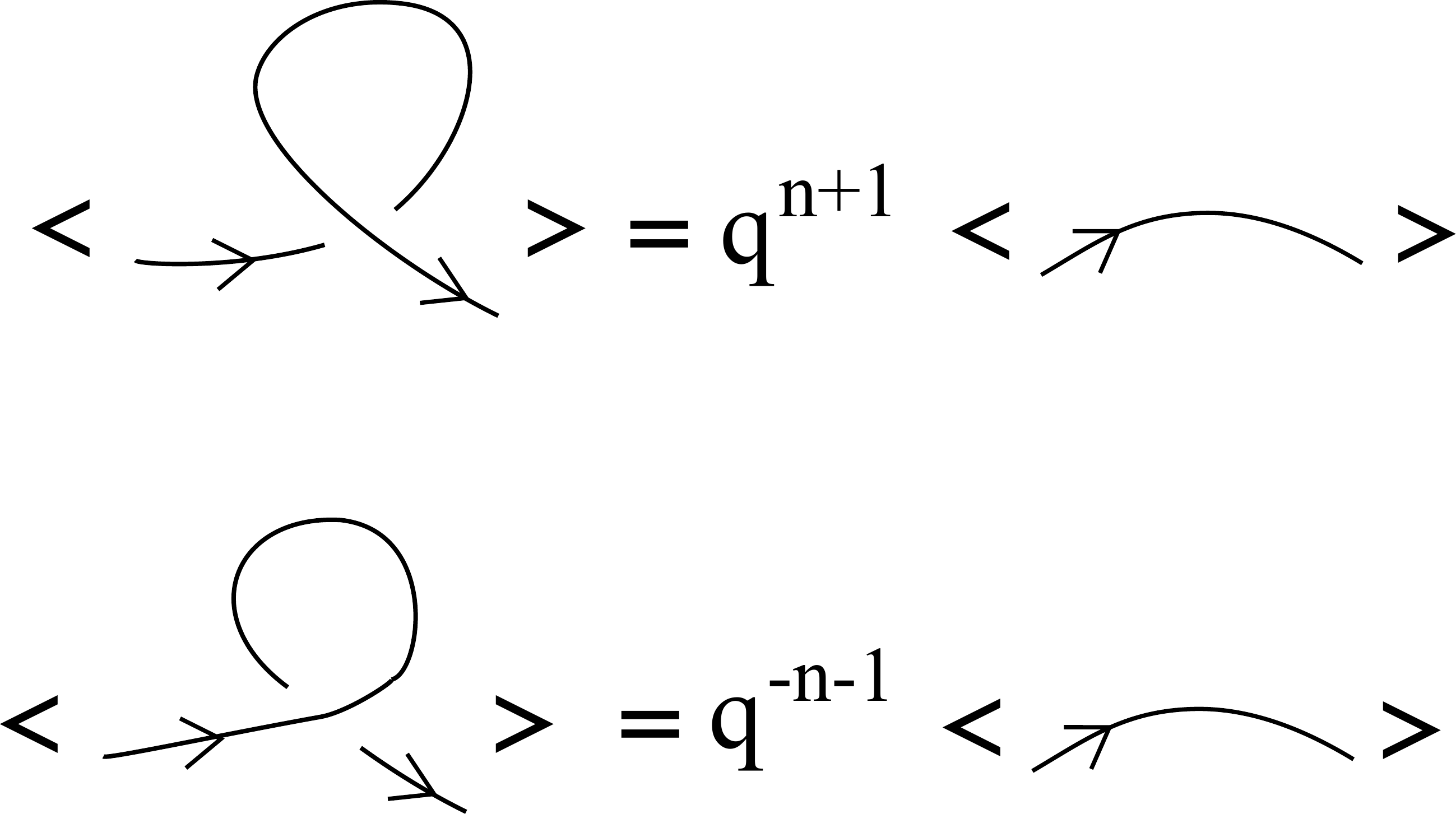}
\caption{$<>$ is not invariant under type I isotopy moves}
\label{fig:homomo}
\end{figure}

\begin{theorem}
Let $\kappa$ be a knotoid in $\mathbb{R}^2$ and $K$ be its standard Morse representation. The normalization of $< >$ with the term $(q^{n+1})^{-w(K)}$ yields a matrix invariant of $\kappa$. 
\end{theorem}

\begin{proof}
It is clear that the term $(q^{n+1})^{-w(K)}$ makes $<>$ invariant under the knotoid isotopy moves. Thus, for all $i,j \in I_n$, for some index set $I_n$, $<K_{i}^{j}>$ is an invariant of $\kappa$. Let $M$ be a matrix whose $ij$ entry is $<K_{i}^{j}>$, for $i,j \in I_n$. Then it follows that $M$ is an invariant of $\kappa$. 
\end{proof}

\begin{definition}\normalfont
Let $n$ be any natural number and $I_{n}$ denote the index set for some $n \in \mathbb{N}$. We define a polynomial $P^{n}_K$ of a Morse knotoid diagram $K$ with fixed indices at its endpoints, and an orientation on it as follows.
$$P^{n}_{K}(q) =\frac{ (q^{n+1})^{-w(K)} <K> }{< O >,}$$
where $w(K)$ denoted the writhe of $K$, and $< O>$ is the value of the unknot that is given by $\sum_{a \in \mathcal{I}{n}} q^{-a} = \sum_{a \in \mathcal{I}_{n}} q^{a}$. The equality holds since the summation is taken over the index set $I_n$ that contains elements symmetric with respect to the origin.

\end{definition}
With the discussion above, it is not hard to see that $P^{n}_{K} (q)$ satisfies the following conditions
\begin{enumerate}
\item $P^{n}_K$ is invariant under the knotoid isotopy moves.
\item $P_{O}^{n} (q) =1$
\item $q^{n+1} P ^{n}_{ K_{+}} - q^{-n-1} P^{n} _{K_{-}} = (q -q^{-1}) P^{n}_{K_0}$, where $K_{+}, K_{-}, K_{0}$ refer to multi-knotoid diagrams obtained by $K$ by replacing a crossing of $K$ by tangle shown in Figure \ref{fig:skein}. 
\end{enumerate}

Therefore, given a natural number $n$, $P_{K}^{n}(q)$ is a one-variable analog of the Homflypt polynomial of classical knots now defined for Morse multi-knotoids. For knotoids in $\mathbb{R}^2$,  the $P_{K}^n(q)$ is related to the corresponding rotational invariant of the virtual knot that is the virtual closure of $K$ by taking the trace of the matrix induced by $P^{n}_K$ for all index labelings at the endpoints of $K$ over the index set $I_n$. 
Let $K$ denote the standard Morse knotoid diagram of a knotoid $\kappa$ in $\mathbb{R}^2$, with fixed labels at its endpoints.
The collection of one-variable invariants  \\
$\{ P^{n}_{K}(q) ~ |~ n \in \mathbb{Z}_+ \}$, yields a unique function in a discrete variable $n$ and a polynomial variable $q$, $P_{K}(l_n, m)$, where $l_n = q^{n+1}$ for $n \in \mathbb{Z}_+$ and $m = q^{-1} -q$. It is straightforward to verify $P_K(l_n, m)$ satisfies the Homflypt polynomial relations for every $n \in \mathbb{N}$. 

It remains a question in this category whether there is an invariant  two-variable polynomial $P_{K}(l,q)$ that specializes to all of the $P_{K}(l_n,q)$, as in the classical case. Answering this question turns on understanding better the specific state sum evaluations that support our invariants. Our invariants are not computed just from the skein relations.

\section{Further Problems}
\begin{itemize}
\item Sawollek polynomial via state sum:\\
\begin{enumerate}
 \item The Sawollek polynomial can detect invertibility of some virtual knots \cite{Saw}. How does the Sawollek polynomial behave in that regard for knotoids in $\mathbb{R}^2$? 
 \item The Sawollek polynomial for virtual knots is divisible by $G = 1-st$. The one variable polynomial obtained by the dividing the Sawollek polynomial by $G$ and setting $s=\frac{1}{t}$ results in the affine index polynomial of virtual knots \cite{Mel, Ka8}. The affine index polynomial is also defined for knotoids in $\mathbb{R}^2$ \cite{GK1}. How do analogs of this theorem work for knotoids, using the state summation models of this paper?
\end{enumerate}

\item Categorification of invariants: Is there a way to categorify the Alexander polynomial and Sawollek polynomials for knotoids in $\mathbb{R}^2$ based on the state sums given in this paper?

\item Quantum invariants and $3$-manifold invariants: 
 \begin{enumerate}
 \item Quantum invariants of classical knots and links can be used to produce invariants of three manifolds \cite{Ka5, Tu2, Oht}. 
 Quantum invariants of knotoids in $\mathbb{R}^2$ can be used to create knotoid invariants that respect the Kirby calculus. Do these invariants give information about three-manifolds?
 
 \item  We intend a sequel to the present paper on Reshsetikhin-Turaev type quantum group and Hopf algebra invariants of planar knotoids.
 
\end{enumerate}

\item Vassiliev invariants can be extracted from many quantum invariants of classical knots and links by a $t=e^{x}$ substitution. By such substitution, we obtain a power series expansion of the quantum invariant in powers of x giving Vassiliev invariants of finite type. See \cite{Ka9} for a more detailed discussion on this. We will investigate derivation of Vassiliev invariants of knotoids \cite{LKM} from the quantum invariants of knotoids presented in this paper.
\end{itemize}

\end{document}